\documentclass[12pt]{amsart}

\usepackage{amsmath,amssymb}
\usepackage{bm}
\usepackage{graphicx,xcolor}
\usepackage{ascmac}
\usepackage{amsthm}
\usepackage{amsfonts}
\usepackage{mathrsfs}
\usepackage{mathtools}
\usepackage[all]{xy}
\usepackage{comment}

\theoremstyle{plain}
\newtheorem{thm}{Theorem}[section]
\newtheorem{cor}[thm]{Corollary}
\newtheorem{prop}[thm]{Proposition}
\newtheorem{lem}[thm]{Lemma}

\theoremstyle{definition}
\newtheorem{dfn}[thm]{Definition}

\newtheorem{rem}[thm]{Remark}

\newtheorem*{nota}{Notation}
\newtheorem*{claim}{Claim}

\newtheorem*{oftp}{Organization of this paper}
\newtheorem*{ack}{Acknowledgments}

\numberwithin{thm}{section}
\numberwithin{equation}{section}

\newcommand{\Zpn}{\mathbb{Z}_{>0}}
\newcommand{\Znn}{\mathbb{Z}_{\geq 0}}
\newcommand{\Z}{\mathbb{Z}}

\newcommand{\Q}{\mathbb{Q}}

\newcommand{\Qp}{{\Q}_p}

\newcommand{\C}{\mathbb{C}}
\newcommand{\F}{\mathbb{F}}
\newcommand{\Fp}{\mathbb{F}_p}

\newcommand{\Fl}{\mathbb{F}_{\ell}}
\newcommand{\A}{\mathbb{A}}

\newcommand{\G}{\mathbb{G}}

\newcommand{\bE}{\mathbf{E}}

\newcommand{\bM}{\mathbf{M}}

\newcommand{\bT}{\mathbf{T}}

\renewcommand{\hbar}{\overline{h}}

\newcommand{\E}{\mathbf{E}}

\newcommand{\bC}{\mathbf{C}}
\newcommand{\bK}{\mathbf{K}}

\newcommand{\bZ}{\mathbf{Z}}

\newcommand{\bj}{\mathbf{j}}

\newcommand{\cC}{\mathcal{C}}
\newcommand{\cD}{\mathcal{D}}

\newcommand{\cH}{\mathcal{H}}

\newcommand{\fA}{\mathfrak{A}}
\newcommand{\fS}{\mathfrak{S}}

\DeclareMathOperator{\Ker}{Ker}
\DeclareMathOperator{\Ima}{Im}

\DeclareMathOperator{\Hom}{Hom}

\DeclareMathOperator{\Aut}{Aut}

\DeclareMathOperator{\ord}{ord}
\DeclareMathOperator{\Gal}{Gal}

\DeclareMathOperator{\GL}{GL}
\DeclareMathOperator{\PGL}{PGL}

\DeclareMathOperator{\PSL}{PSL}

\DeclareMathOperator{\N}{N}

\DeclareMathOperator{\Spec}{Spec}

\DeclareMathOperator{\red}{red}

\DeclareMathOperator{\Res}{Res}
\DeclareMathOperator{\Cor}{Cor}
\DeclareMathOperator{\Inf}{Inf}

\DeclareMathOperator{\diag}{diag}

\DeclareMathOperator{\Br}{Br}

\DeclareMathOperator{\Pic}{Pic}

\DeclareMathOperator{\sep}{sep}
\DeclareMathOperator{\der}{der}

\DeclareMathOperator{\Stab}{Stab}

\DeclareMathOperator{\Ind}{Ind}
\DeclareMathOperator{\pr}{pr}

\DeclareMathOperator{\Ad}{Ad}
\DeclareMathOperator{\Fix}{Fix}

\DeclareMathOperator{\et}{\acute{e}t}

\DeclareMathOperator{\set}{set}

\usepackage[OT2,T1]{fontenc}
\DeclareSymbolFont{cyrletters}{OT2}{wncyr}{m}{n}
\DeclareMathSymbol{\Sha}{\mathalpha}{cyrletters}{"58}

\usepackage{geometry}
\title[Hasse norm principle]{The Hasse norm principle for some extensions of degree having square-free prime factors}
\author[Y.~Oki]{Yasuhiro Oki}
\address{Department of Mathematics, College of Science, Rikkyo University, 3-34-1, Nishi-Ikebukuro, Toshima-ku, Tokyo, 171-8501, Japan. }
\email{oki@rikkyo.ac.jp}

\begin{document}

\begin{abstract}
We determine the structure of the obstruction group of the Hasse norm principle for a finite separable extension $K/k$ of a global field of degree $d$, where $d$ has a square-free prime factor $p$ and a $p$-Sylow subgroup of the Galois group $G$ of the Galois closure of $K/k$ is normal in $G$. Specifically, we give a partial classification of the validity of the Hasse norm principle for $K/k$ in the case where (1) $[K:k]=p\ell$ where $p$ and $\ell$ are two distinct prime numbers; or (2) $[K:k]=4p$ where $p$ is an odd prime. The result (1) gives infinitely many new existences of finite extensions of arbitrary number fields for which the Hasse norm principle fail. Furthermore, we prove that there exist infinitely many separable extensions of square-free degree for which the exponents of the obstruction groups to the Hasse norm principle are not prime powers. 
\end{abstract}

\maketitle

\tableofcontents

\section{Introduction}\label{sect:intr}

Let $k$ be a global field, and $K/k$ a finite separable field extension. Put
\begin{equation*}
\Sha(K/k):=(k^{\times}\cap \N_{K/k}(\A_{K}))/\N_{K/k}(K^{\times}),
\end{equation*}
where $\A_{K}^{\times}$ is the id{\`e}le group of $K$. In the paper of Scholz (\cite{Scholz1940}), the group $\Sha(K/k)$ is denoted by $\kappa(K/k)$ and is called the \emph{number knot}. We say that the \emph{Hasse norm principle holds for $K/k$} if
\begin{equation*}
\Sha(K/k)=1. 
\end{equation*}
The study of $\Sha(K/k)$ is one of the classic problems in algebraic number theory. It is known by Hasse (\cite{Hasse1931}) that the Hasse norm principle holds for all cyclic extensions of $k$. On the other hand, he also clarified in the same paper that the Hasse norm principle fails for the biquadratic extension $\Q(\sqrt{-39},\sqrt{-3})/\Q$. After that, the classification of the validity of the Haase norm principle is developed by many people. In 1967, Tate (\cite{Tate1967}) gave a description of $\Sha(K/k)$ for a finite Galois extension $K/k$ by means of cohomology of the associated Galois group. On the other hand, it is much less known if $K/k$ is Galois. For a non-Galois extension $K/k$, we denote by $\widetilde{K}/k$ the Galois closure of $K/k$. There exist three strategies for the study for non-Galois extensions as follows. 
\begin{enumerate}
\item Investigate certain families of finite extensions indexed by (some) positive integers. 
\item For a fixed $d\in \Zpn$, study all finite separable extensions of degree $d$. 
\item For a given finite group $G$, consider all finite separable extensions $K/k$ that admit an isomorphism $\Gal(\widetilde{K}/k)\cong G$. 
\end{enumerate}

First, we enumerate the results that follow (i). If $[K:k]=d$ and $\Gal(\widetilde{K}/k)$ is isomorphic to the dihedral group of order $2d$, then Bartels gave that the Hasse norm principle holds for $K/k$ (\cite{Bartels1981b}). The same assertion has been proved by Voskresenskii and Kunyavskii in the case $[K:k]=d$ and that $\Gal(\widetilde{K}/k)$ is isomorphic to the symmetric group of $d$-letters (\cite{Voskresenskii1984}). Furthermore, Macedo provided the same assertion as above in the case $[K:k]=d\geq 5$ and $\Gal(\widetilde{K}/k)$ is isomorphic to the alternating group of $d$-letters (\cite{Macedo2020}). 

Second, we recall the previous study based on (ii). If $d$ is a prime number, then Bartels (\cite{Bartels1981a}) proved that the Hasse norm principle always holds. After that, Kunyavskii gave an equivalent condition for the validity of the Hasse norm principle in the case $d=4$ (\cite{Kunyavskii1984}). The same results as above are given by Drakokhrust and Platonov for $d=6$ (\cite{Drakokhrust1987}), and by Hoshi, Kanai and Yamasaki for $d\leq 15$ (\cite{Hoshi2022}, \cite{Hoshi2023}). There also exists a partial result in the case $d=16$, which is given by Hoshi, Kanai and Yamasaki (\cite{Hoshi2025}). Note that there is a possibility that the Hasse norm principle may fail if $d$ is a composite number $\leq 16$. However, as $d$ increases, this approach will become extremely difficult to implement since their results are based on case-by-case computation. 

Finally, we list the research that is in line with (iii). First, the case where $G$ is a symmtetic group or an alternating group is discussed by Macedo and Newton (\cite{Macedo2022}). In particular, it implies that $\Sha(K/k)$ is isomorphic to $\Z/2$ for certain finite extension of degree $30$, $70$ or $210$. On the other hand, the same considerations as above were made by Hoshi, Kanai and Yamasaki in the case that $G$ is the Mathieu group $M_{11}$ or the Janko group $J_{1}$ (\cite{Hoshi2023b}). Note that their study gives an isomorphism $\Sha(K/k)\cong \Z/2$ which may happen when $n=110$ and $G_{0}=M_{11}$. 

In this paper, we study the Hasse norm principle for finite extensions based on (i), however we further assume that $d$ has a square-free prime factor. Note that this study also gives a partial result related to (ii) for infinitely many $d$. Moreover, we can give new examples of the failure of the Hasse norm principle for \emph{all} number fields and some function fields. 

\subsection{Main theorems}

To state our theorem, we use the notations as follows for a finite group $G$ and its subgroup $H'$: 
\begin{itemize}
\item $N_{G}(H'):=\{g\in G\mid gH'g^{-1}=H'\}$ is the normalizer of $H'$ in $G$; 
\item $Z_{G}(H'):=\{g\in G\mid gh'=h'g\text{ for all }h'\in H'\}$ is the centralizer of $H'$ in $G$; 
\item $N^{G}(H'):=\bigcap_{g\in G} gH'g^{-1}$, the normal core of $H'$ in $G$; 
\item $[H',G]:=\langle hgh^{-1}g^{-1}\in G \mid h\in H',g\in G \rangle$. 
\end{itemize}

On the other hand, for a finite abelian group $A$ and a prime number $p$, let
\begin{gather*}
A[p^{\infty}]:=\{a\in A\mid p^{n}a=0\text{ for some }n\in \Zpn\},\\
A^{(p)}:=\{a\in A\mid na=0\text{ for some }n\in \Z \setminus p\Z\}.
\end{gather*}
Note that $A=A[p^{\infty}]\oplus A^{(p)}$ by definition. 

\begin{thm}[{Theorem \ref{thm:ptnt}}]\label{mth1}
Let $K/k$ be a finite separable field extension of a global field with its Galois closure $\widetilde{K}/k$. Put $G:=\Gal(\widetilde{K}/k)$ and $H:=\Gal(\widetilde{K}/K)$. We further assume 
\begin{itemize}
\item $[K:k]\in p\Z \setminus p^{2}\Z$; and 
\item a $p$-Sylow subgroup $S_{p}$ of $G$ is normal in $G$. 
\end{itemize}
\begin{enumerate}
\item If $\Sha(K/k)[p^{\infty}]\neq 1$, then
\begin{itemize}
\item[(a)] $S_{p}\cong (C_{p})^{2}$; 
\item[(b)] $[S_{p},G]=S_{p}$; and
\item[(c)] $N_{G}(S_{p}\cap H)=Z_{G}(S_{p}\cap H)$. 
\end{itemize}
Conversely, if \emph{(a)}, \emph{(b)} and \emph{(c)} hold, then there is an isomorphism
\begin{equation*}
\Sha(K/k)[p^{\infty}]\cong 
\begin{cases}
1&\text{if a decomposition group of $\widetilde{K}/k$ contains $S_{p}$;}\\
\Z/p&\text{otherwise. }
\end{cases}
\end{equation*}
\item Let $K_{0}/k$ be the subextension of $\widetilde{K}/k$ corresponding to $HS_{p}$. Then there is an isomorphism
\begin{equation*}
\Sha(K/k)^{(p)}\cong \Sha(K_{0}/k). 
\end{equation*}
\end{enumerate}
\end{thm}

We note that there is a unique $p$-Sylow subgroup of a finite group $G$ if and only if it is normal in $G$. This follows from Sylow's theorem, which implies that all Sylow subgroups of $G$ are conjugate to each other. 

\vspace{6pt}
If we impose an additional condition on $p$ and $[K:k]$, then we obtain a vanishing of $\Sha(K/k)[p^{\infty}]$. 

\begin{thm}[{Corollary \ref{cor:pttr}}]\label{mth2}
Let $K/k$ be a finite separable field extension of a global field with its Galois closure $\widetilde{K}/k$. Put $G:=\Gal(\widetilde{K}/k)$ and $H:=\Gal(\widetilde{K}/K)$. We further assume 
\begin{itemize}
\item $[K:k]\in p\Z \setminus p^{2}\Z$; 
\item a $p$-Sylow subgroup $S_{p}$ of $G$ is normal in $G$; and
\item $n:=[K:k]/p$ satisfies $\gcd(n,p-1)\leq 2$ and $\gcd(n,p+1)\in 2^{\Znn}$. 
\end{itemize}
Then there is an isomorphism
\begin{equation*}
\Sha(K/k) \cong \Sha(K_{0}/k),
\end{equation*}
where $K_{0}/k$ is the subextension of $\widetilde{K}/k$ corresponding to $HS_{p}$. In particular, $\Sha(K/k)$ has no $p$-primary torsion part.  
\end{thm}

On the other hand, Theorem \ref{mth1} also implies partial classification of the non-vanishing of $\Sha(K/k)$ for an extension $K/k$ of certain degree. We use the following notations. 
\begin{itemize}
\item For each $n\in \Zpn$, $C_{n}=\langle c_{n}\rangle$ is the cyclic group of order $n$. Moreover, we write for $\Delta$ the diagonal map $C_{n}\hookrightarrow C_{n}\times C_{n}$. 
\item For an integer $n\geq 2$, put
\begin{equation*}
D_{n}:=\langle \sigma_{n},\tau_{n} \mid \sigma_{n}^{n}=\tau_{n}^{2}=1,\tau_{n} \sigma_{n} \tau_{n}^{-1}=\sigma_{n}^{-1}\rangle,
\end{equation*}
the dihedral group of order $2n$. 
\item For $n\in \Zpn$, we denote by $\fA_{n}$ the alternating group of $n$-letters. 
\end{itemize}

\begin{thm}[Theorem \ref{thm:pldt}]\label{mth3}
Let $p$ and $\ell$ be two distinct prime numbers. Consider a finite separable extension $K/k$ of degree $p\ell$ of a global field, and write for $\widetilde{K}/k$ its Galois closure. Set $G:=\Gal(\widetilde{K}/k)$ and $H:=\Gal(\widetilde{K}/K)$. We further assume that a $p$-Sylow subgroup $S_{p}$ of $G$ is normal in $G$. If $\Sha(K/k)\neq 1$, then one of the following is satisfied: 
\begin{itemize}
\item[($\alpha$)] $2<\ell \mid p-1$, $G\cong (C_{p})^{2}\rtimes_{\varphi_{0,m}} C_{\ell}$ and $H\cong \Delta(C_{p}) \rtimes_{\varphi_{0,m}} \{1\}$, where $\varphi_{0,m}$ is defined by
\begin{equation*}
C_{\ell} \rightarrow \GL_{2}(\Fp); c_{\ell} \mapsto \diag(\zeta_{\ell},\zeta_{\ell}^{m}),
\end{equation*}
$\zeta_{\ell}\in \F_{p}^{\times}$ is a primitive $\ell$-th root of unity, and $m\in \{2,\ldots,\ell-1\}$ satisfies $mm'\not\equiv 1 \bmod \ell$ for any $m'\in \{1,\ldots,m-1\}$;
\item[($\beta$)] $2<\ell \mid p+1$, $G\cong (C_{p})^{2}\rtimes_{\varphi_{1}} C_{\ell}$ and $H\cong \Delta(C_{p}) \rtimes_{\varphi_{1}} \{1\}$, where $\varphi_{1}$ is defined by
\begin{equation*}
C_{\ell} \rightarrow \GL_{2}(\Fp); c_{\ell} \mapsto 
\begin{pmatrix}
0&-1\\
1&\zeta_{\ell}+\zeta_{\ell}^{-1}
\end{pmatrix}
\end{equation*}
and $\zeta_{\ell}\in \F_{p^{2}}^{\times}\setminus \F_{p}^{\times}$ is a primitive $\ell$-th root of unity. 
\item[($\gamma$)] $p\geq 5$, $2<\ell \mid p^{2}-1$, $G \cong (C_{p})^{2}\rtimes_{\varphi_{2}} D_{\ell}$ and $H\cong \Delta(C_{p}) \rtimes_{\varphi_{2}} \langle \tau_{\ell} \rangle$, where $\varphi_{2}$ is defined by
\begin{equation*}
D_{\ell} \rightarrow \GL_{2}(\Fp); \sigma_{\ell}^{i}\tau_{\ell}^{i'}\mapsto 
\begin{pmatrix}
0&-1\\
1&\zeta_{\ell}+\zeta_{\ell}^{-1}
\end{pmatrix}^{i}
\begin{pmatrix}
0&1\\
1&0
\end{pmatrix}^{i'}
\end{equation*}
and $\zeta_{\ell} \in \F_{\ell^{2}}^{\times}$ is a primitive $\ell$-th root of unity. 
\end{itemize}
Conversely, if \emph{($\alpha$)}, \emph{($\beta$)} or \emph{($\gamma$)} holds, then there is an isomorphism
\begin{equation*}
\Sha(K/k) \cong 
\begin{cases}
1&\text{if a decomposition group of $\widetilde{K}/k$ contains $S_{p}$; }\\
\Z/p &\text{otherwise. }
\end{cases}
\end{equation*}
\end{thm}

Combining Theorem \ref{mth2} with (the proof of) the inverse Galois problem for finite solvable groups due to Shafarevich, we obtain the following. 

\begin{thm}[{Theorem \ref{thm:sfhf}}]\label{mth4}
Let $p$ and $\ell$ be two distinct prime numbers satisfying $2<\ell \mid p^{2}-1$, and $d$ a multiple of $p\ell$ that is square-free. Take a global field $k$ whose characteristic is different from $p$. Then there is a finite extension $K$ of $k$ of degree $d$ for which the Hasse norm principle fails. 
\end{thm}

\begin{thm}[{Theorem \ref{thm:pctw}}]\label{mth5}
Let $p>2$ be an odd prime number. Consider a finite separable extension $K/k$ of degree $4p$ of a global field, and write for $\widetilde{K}/k$ its Galois closure. Put $G:=\Gal(\widetilde{K}/k)$ and $H:=\Gal(\widetilde{K}/K)$. We further assume that a $p$-Sylow subgroup $S_{p}$ of $G$ is normal. 
\begin{enumerate}
\item If $\Sha(K/k)$ is non-trivial, then it is isomorphic to $\Z/2$ or $\Z/p$. 
\item If $\Sha(K/k)\cong \Z/p$, then $p\equiv 1\bmod 4$ and there exist an isomorphisms
\begin{equation}\label{eq:mt5i}
G\cong (C_{p})^{2}\rtimes_{\varphi_{p,4}}C_{4},\quad H\cong \Delta(C_{p}) \rtimes_{\varphi}\{1\}. 
\end{equation}
Here $\varphi_{p,4}$ is defined by the homomorphism
\begin{equation*}
\Z/4 \rightarrow \GL_{2}(\Fp);\,1\bmod 4\mapsto \diag(-1,\sqrt{-1}). 
\end{equation*}
Conversely, if $p\equiv 1\bmod 4$ and \eqref{eq:mt5i} hold, then one has an isomorphism
\begin{equation*}
\Sha(K/k)\cong 
\begin{cases}
1&\text{if a decomposition group of $\widetilde{K}/k$ contains $S_{p}$; }\\
\Z/p &\text{otherwise. }
\end{cases}
\end{equation*}
\item If $\Sha(K/k)\cong \Z/2$, then we have
\begin{itemize}
\item[(a)] $G/N^{G}(HS_{p})\cong (C_{2})^{2}$; or 
\item[(b)] $p\geq 5$ and $G/N^{G}(HS_{p})\cong \fA_{4}$. 
\end{itemize}
Conversely, if \emph{(a)} or \emph{(b)} is satisfied, then there is an isomorphism
\begin{equation}\label{eq:chvp}
\Sha(K/k)\cong 
\begin{cases}
1&\text{if a decomposition group of $\widetilde{K}/k$ contains $(C_{2})^{2}$; }\\
\Z/2&\text{otherwise. }
\end{cases} 
\end{equation}
\end{enumerate}
\end{thm}

\begin{rem}
Theorem \ref{mth3} for $\{p,\ell\}=\{2,3\},\{2,5\},\{2,7\}$ or $\{3,5\}$ and Theorem \ref{mth4} for $p=3$ partially recover the classification of the structure of $\Sha(K/k)$ where $[K:k]\in \{6,10,12,14,15\}$, which are given by \cite{Drakokhrust1987}, \cite{Hoshi2022} and \cite{Hoshi2023}. See Section \ref{ssec:cppr} for more details. 
\end{rem}

We can also prove the existence of finite separable field extension $K/k$ of degree having square-free prime factors such that the exponent of $\Sha(K/k)$ is not a prime power. It follows from Theorem \ref{mth3} and some additional arguments. 

\begin{thm}[{Theorem \ref{thm:cltt}}]\label{mth6}
Let $p\neq 3$ be a prime number, and $k$ a global field whose characteristic is different from $p$. 
\begin{enumerate}
\item There is a finite separable field extension $K/k$ of degree $9p$ such that
\begin{equation*}
\Sha(K/k)\cong \Z/3p. 
\end{equation*}
\item We further assume that $k$ has characteristic $\neq \ell$, where $\ell\notin \{3,p\}$ is a prime number. Then there is a finite separable field extension $K/k$ of degree $3p\ell$ such that
\begin{equation*}
\Sha(K/k)\cong \Z/p\ell. 
\end{equation*}
\end{enumerate}
\end{thm}

\begin{rem}
\begin{enumerate}
\item If a finite Galois extension $K/k$ has square-free degree, then we have $\Sha(K/k)=1$. This is a consequence of Gurak (\cite{Gurak1978}). Hence, the finite extensions in Theorems \ref{mth4} and \ref{mth6} (i) are always non-Galois. 
\item Let $K/k$ be a finite Galois extension of degree $9p$, where $p\neq 3$ is a prime number. Put $G:=\Gal(K/k)$, and take a $3$-Sylow subgroup $S_{3}$ of $G$. Then the restriction map
\begin{equation*}
\widehat{H}^{-3}(G,\Z)\rightarrow \widehat{H}^{-3}(S_{3},\Z)
\end{equation*}
is injective. Combining this fact with Tate's theorem, we obtain that $\Sha(K/k)$ is trivial or isomorphic to $\Z/3$. Therefore, the finite extension in Theorem \ref{mth6} (ii) must be non-Galois. 
\end{enumerate}
\end{rem}

Now we explain that Theorem \ref{mth1} implies Theorems \ref{mth3} and \ref{mth5}. We first discuss Theorem \ref{mth3}. Let $K/k$, $G$, $H$ and $S_{p}$ be as in Theorem \ref{mth3}. In particular, $K/k$ has degree $p\ell$. It suffices to prove that
\begin{itemize}
\item[(1)] the conditions (a), (b), (c) in Theorem \ref{mth1} are satisfied if and only if ($\alpha$), ($\beta$) or ($\gamma$) is valid; and
\item[(2)] $\Sha(K/k)^{(p)}\cong \Sha(K_{0}/k)$ is trivial. 
\end{itemize}
For (1), take a subgroup $G'$ of $G$ satisfying $G=S_{p}\rtimes G'$. Then we may assume $H=(S_{p}\cap H)\rtimes H'$ for some subgroup $H'$ of $G'$. Then (a), (b), (c) in Theorem \ref{mth1} can be rephrased to those on $2$-dimensional $\Fp$-representations of $G'$, which is called in this paper an \emph{$H'$-extremal $\Fp$-representation}. However, if $p>2$, it is archived using Dickson's classification and structure of $2$-Sylow subgroups of $\GL_{2}(\Fp)$. Note that Dickson's classification is a description of all maximal subgroups of $\PGL_{2}(\Fp)$. The case $p=2$ is much easier than the case $p>2$ because there exist isomorphisms $\GL_{2}(\F_{2})\cong \PGL_{2}(\F_{2})\cong \fS_{3}$. On the other hand, (2) is an easy consequence of Bartels' theorem, since $[K_{0}:k]=\ell$ is a prime. The proof of Theorem \ref{mth5} is similar to Theorem \ref{mth3}. However, we use Kunyavskii's theorem (\cite{Kunyavskii1984}) instead of Bartels' theorem. 

\subsection{Results on cohomological invariants of norm one tori}

Here, let $F$ be a field, and fix a separable closure $F^{\sep}$ of $F$. Consider a finite separable field extension $E$ of $F$ that is contained in $F^{\sep}$. Then the $F$-torus
\begin{equation*}
T_{E/F}:=\Ker(\N_{E/F}\colon \Res_{E/F}\G_{m}\rightarrow \G_{m}). 
\end{equation*}
is called the \emph{norm one torus} associated to $E/F$. Take a smooth compactification $X$ of $T_{E/F}$, and set $\overline{X}:=X\otimes_{F} F^{\sep}$. Then the Picard $\Pic(\overline{X})$ is a finite free abelian group equipped with a continuous action of the absolute Galois group of $F$ (with respect to the discrete topology on $\Pic(\overline{X})$). Note that there exists an isomorphism
\begin{equation*}
H^{1}(F,\Pic(\overline{X}))\cong \Br(X)/\Br(k), 
\end{equation*}
where $\Br(X):=H_{\et}^{2}(X,\G_{m})$ is the {\'e}tale cohomological Brauer group of $X$. See \cite[Theorem 9.5]{ColliotThelene1987}. 

\vspace{6pt}
It is known that $H^{1}(F,\Pic(\overline{X}))$ is related to the Hasse norm principle. Let $K/k$ be a finite extension of a global field, and denote by $\Sigma_{k}$ the set of places of $k$. Then Ono constructed in \cite{Ono1963} an isomorphism
\begin{equation*}
\Sha(K/k)\cong \Sha^{1}(k,T_{K/k}). 
\end{equation*}
Here right-hand side is the \emph{Tate--Shafarevich group} of $T_{K/k}$, which is defined as the kernel of the homomorphism
\begin{equation*}
(\Res_{k_{v}/k})_{v}\colon H^{1}(k,T_{K/k})\rightarrow \prod_{v\in \Sigma_{k}}H^{1}(k_{v},T_{K/k}). 
\end{equation*}
On the other hand, put
\begin{equation*}
A_{k}(T_{K/k}):=\left(\prod_{v\in \Sigma_{k}}T_{K/k}(k_{v})\right)/\overline{T_{K/k}(k)}, 
\end{equation*}
where $\overline{T_{K/k}(k)}$ is the closure of $T_{K/k}(k)$ in $\prod_{v}T_{K/k}(k_{v})$. Then Vosresenskii (\cite{Voskresenskii1969}) gave an exact sequence
\begin{equation*}
0\rightarrow A_{k}(T_{K/k})\rightarrow H^{1}(k,\Pic(\overline{X}))^{\vee} \rightarrow \Sha^{1}(k,T_{K/k}) \rightarrow 0,
\end{equation*}
where $H^{1}(k,\Pic(\overline{X}))^{\vee}$ is the Pontryagin dual of $H^{1}(k,\Pic(\overline{X}))$. 

\vspace{5pt}
For a finite group $G$ and its subgroups $H$, we use the notations as follows: 
\begin{itemize}
\item $N_{G}(H):=\{g\in G\mid gHg^{-1}=H\}$, the normalizer of $H$ in $G$; 
\item $Z_{G}(H):=\{g\in G\mid gh=hg\text{ for any }h\in H\}$, the centralizer of $H$ in $G$; 
\item for another subgroup $D$ of $G$, $[H,D]:=\langle h^{-1}dhd^{-1}\mid h\in H,d\in D\rangle$; 
\end{itemize}

\begin{thm}[{Theorem \ref{thm:mtub}}]\label{mth8}
Let $F$ be a field, and $E/F$ a finite separable field extension with its Galois closure $\widetilde{E}/F$. Put $G:=\Gal(\widetilde{E}/F)$ and $H:=\Gal(\widetilde{E}/E)$. We further assume 
\begin{itemize}
\item $[E:F]\in p\Z \setminus p^{2}\Z$; and 
\item a $p$-Sylow subgroup $S_{p}$ of $G$ is normal in $G$. 
\end{itemize}
\begin{enumerate}
\item The abelian group $H^{1}(F,\Pic(\overline{X}))[p^{\infty}]$ is non-trivial if and only if
\begin{itemize}
\item[(a)] $S_{p}\cong (C_{p})^{2}$; 
\item[(b)] $[S_{p},G]=S_{p}$; and
\item[(c)] $N_{G}(S_{p}\cap H)=Z_{G}(S_{p}\cap H)$. 
\end{itemize}
Moreover, if $H^{1}(F,\Pic(\overline{X}))[p^{\infty}]\neq 0$, then there is an isomorphism
\begin{equation*}
H^{1}(F,\Pic(\overline{X}))\cong \Z/p. 
\end{equation*}
\item Let $E_{0}/F$ be the subextension of $\widetilde{E}/F$ corresponding to $HS_{p}$. Then there is an isomorphism of finite abelian groups
\begin{equation*}
H^{1}(F,\Pic(\overline{X}))^{(p)}\cong H^{1}(F,\Pic(\overline{X}_{0})), 
\end{equation*}
where $\overline{X}_{0}=X_{0}\otimes_{F}F^{\sep}$ and $X_{0}$ is a smooth compactification of $T_{E_{0}/F}$. 
\end{enumerate}
\end{thm}

\subsection{Second cohomology groups}

Here we explain proofs of Theorems \ref{mth1} and \ref{mth8}. Let $G$ be a finite group. For a subgroup $H$ of $G$, put $\Ind_{H}^{G}\Z:=\Hom_{\Z[H]}(\Z[G],M)$. Consider an exact sequence
\begin{equation*}
0\rightarrow \Z \rightarrow \Ind_{H}^{G}\Z \rightarrow J_{G/H}\rightarrow 0, 
\end{equation*}
where the homomorphism $\Z \rightarrow J_{G/H}$ is defined by sending $1$ to the map $g\mapsto 1$ for all $g\in G$. Pick a set of subgroups $\cD$ of $G$ containing all its cyclic subgroups that is stable under conjugation of $G$. We say such a set to be \emph{admissible} in this paper. Then we discuss the structure of the abelian group
\begin{equation*}
\Sha_{\cD}^{2}(G,J_{G/H}):=\Ker \left(H^{2}(G,J_{G/H})\xrightarrow{(\Res_{G/D})_{D \in \cD}} \bigoplus_{D\in \cD}H^{2}(D,J_{G/H})\right). 
\end{equation*}
If $\cD$ is the set of all cyclic subgroups of $G$, then we also denote $\Sha_{\cD}^{2}(G,J_{G/H})$ by $\Sha_{\omega}^{2}(G,J_{G/H})$. 

The $G$-lattice $J_{G/H}$ is related to a norm one torus as follows. Let $E/F$ be a finite separable field extension with Galois closure $\widetilde{E}/F$. Put $G=\Gal(\widetilde{E}/F)$ and $H=\Gal(\widetilde{E}/E)$. Then one has an isomorphism of $G$-lattices $J_{G/H}\cong X^{*}(T_{E/F})$. Here the right-hand side is the \emph{character group} of $T_{E/F}$, which is defined as follows: 
\begin{equation*}
X^{*}(T_{E/F}):=\Hom_{F^{\sep}\text{-group}}(T_{E/F}\otimes_{F}F^{\sep},\G_{m,F^{\sep}}). 
\end{equation*}
Moreover, there is an isomorphism of finite abelian groups
\begin{equation*}
H^{1}(F,\Pic(\overline{X}))\cong \Sha_{\omega}^{2}(G,J_{G/H}),
\end{equation*}
where $\overline{X}=X\otimes_{F}F^{\sep}$ and $X$ is a smooth compactification of $T_{E/F}$. 

Under the above notations, we further assume that $F=k$ is a global field and set $K:=E$. Denote by $\cD$ the set of all decomposition groups of $\widetilde{K}/k$, which is an admissible set of subgroups of $G$. Then the Poitou--Tate duality gives an isomorphism
\begin{equation*}
\Sha^{1}(k,T_{K/k})\cong \Sha_{\cD}^{2}(G,J_{G/H})^{\vee}, 
\end{equation*}
where $\Sha_{\cD}^{2}(G,J_{G/H})^{\vee}$ is the Pontryagin dual of $\Sha_{\cD}^{2}(G,J_{G/H})$. See also Corollary \ref{cor:onhn}. Therefore, Theorem \ref{mth1} and Theorem \ref{mth6} are reduced to the structure theorem of $\Sha_{\cD}^{2}(G,J_{G/H})$ as follows. 

\vspace{5pt}
Let $n$ be a positive integer. Recall that a \emph{transitive group of degree $n$} (or, a \emph{transitive subgroup of $\fS_{n}$}) is a subgroup of the symmetric group $\fS_{n}$ of $n$-letters whose action on $\{1,\ldots,n\}$ induced by the natural action of $\fS_{n}$ is transitive. Moreover, we call a subgroup of $G$ of the form $\Stab_{G}(i)$ where $i\in \{1,\ldots,n\}$ a \emph{corresponding subgroup} of $G$. 

\begin{thm}[Theorem \ref{thm:tspt}]\label{mth9}
Let $p$ be a prime number, $G$ a transitive group of degree $\in p\Z \setminus p^{2}\Z$, and $H$ a corresponding subgroup of $G$. We further assume that a $p$-Sylow subgroup $S_{p}$ of $G$ is normal in $G$.  
\begin{enumerate}
\item One has $\Sha_{\omega}^{2}(G,J_{G/H})[p^{\infty}]\neq 0$ if and only if
\begin{itemize}
\item[(a)] there is an isomorphism $S_{p}\cong (C_{p})^{2}$; 
\item[(b)] $[S_{p},G]=S_{p}$; and
\item[(c)] $N_{G}(S_{p}\cap H)=Z_{G}(S_{p}\cap H)$. 
\end{itemize}
Conversely, assume that \emph{(a)}, \emph{(b)} and \emph{(c)} hold. Then, for an admissible set $\cD$ of subgroups of $G$, there is an isomorphism 
\begin{equation*}
\Sha_{\cD}^{2}(G,J_{G/H})[p^{\infty}] \cong 
\begin{cases}
0&\text{if there is $D \in \cD$ which contains $S_{p}$;}\\
\Z/p &\text{otherwise. }
\end{cases}
\end{equation*}
\item For any admissible set $\cD$ of subgroups of $G$, there is an isomorphism
\begin{equation*}
\Sha_{\cD}^{2}(G,J_{G/H})^{(p)} \cong \Sha_{\cD_{/S_{p}}}^{2}(G/S_{p},J_{G/HS_{p}}). 
\end{equation*}
Here, $\cD_{/S_{p}}$ is the set of subgroups of the form $DS_{p}/S_{p}$ where $D\in \cD$, which is an admissible set of subgroups of $G/S_{p}$. 
\end{enumerate}
\end{thm}

In the following, we sketch a proof of Theorem \ref{mth9}. Since $(G:H)\in p\Z \setminus p^{2}\Z$, it turns out that there is an isomorphism $S_{p}\cong (C_{p})^{m}$ for some $m\in \Zpn$.

\vspace{6pt}
First, we explain the proof of (i). The restriction map $\Res_{G/S_{p}}$ induces an injection
\begin{equation}\label{eq:pieb}
\Sha_{\cD}^{2}(G,J_{G/H})[p^{\infty}]\hookrightarrow \Sha_{\cD_{S_{p}}}^{2}(S_{p},J_{G/H}), 
\end{equation}
where $\cD_{S_{p}}$ is the set of subgroups of $S_{p}$ of the form $D\cap S_{p}$ for some $D\in \cD$. Note that $J_{G/H}$ cannot be written as $J_{S_{p}/L}$ for some subgroup $L$ of $S_{p}$ as an $S_{p}$-lattice. However, we can take finitely many subgroups $L_{1},\ldots,L_{r}$ of $G$ that admit an exact sequence of $S_{p}$-lattices as follows: 
\begin{equation*}
0\rightarrow \Z \rightarrow \bigoplus_{i=1}^{r}\Ind_{L_{i}}^{S_{p}}\Z \rightarrow J_{G/H}\rightarrow 0. 
\end{equation*}
Note that such a $S_{p}$-lattice can be described as the character group of a \emph{multinorm one torus}. In our case, we have $(S_{p}:L_{i})=p$ for any $i$ and $L_{1}\cap \cdots \cap L_{r}=\{1\}$. On the other hand, we know the structure of the right-hand side of \eqref{eq:pieb}, which follows from Bayer-Fluckiger--Lee--Parimala (\cite{BayerFluckiger2019}). If $m\neq 2$, then $\Sha_{\omega}^{2}(S_{p},J_{G/H})$ is trivial, and in particular the assertion holds. On the other hand, it does not vanish if $m=2$, and hence we need a delicate argument. 

Recall that $\cD$ contains a subgroup of $G$ containing $S_{p}$ if and only if $\cD_{S_{p}}$ contains $S_{p}$. If it is satisfied, then we can prove $\Sha_{\cD_{S_{p}}}^{2}(S_{p},J_{G/H})=0$ without difficulty. Otherwise, we consider a short exact sequence of $G$-lattices
\begin{equation}\label{eq:itex}
0\rightarrow J_{G/HS_{p}}\rightarrow J_{G/H}\rightarrow \Ind_{HS_{p}}^{G}J_{HS_{p}/S_{p}}\rightarrow 0. 
\end{equation}
Moreover, we introduce a subgroup $H^{2}(G,J_{G/HS_{p}})^{(\cD)}$ of $H^{2}(G,J_{G/HS_{p}})$, which fits into an exact sequence
\begin{equation}\label{eq:sele}
0\rightarrow H^{1}(G,\Ind_{HS_{p}}^{G}J_{HS_{p}/S_{p}})\xrightarrow{\delta_{G}} H^{2}(G,J_{G/HS_{p}})^{(\cD)} \xrightarrow{\widehat{\N}_{G}} \Sha_{\cD}^{2}(G,J_{G/H}) \rightarrow 0. 
\end{equation}
See Definition \ref{dfn:h2ly}. Note that $H^{2}(G,J_{G/HS_{p}})^{(\cD)}$ plays a role of Selmer groups. This method is inspired by the study of \emph{CM tori} developed by Liang, Yang, Yu and the author (\cite{Liang2024a}). Then we give an isomorphism 
\begin{equation*}
H^{2}(G,J_{G/HS_{p}})^{(\cD)}[p^{\infty}]\cong \frac{(S_{p}/[S_{p},HS_{p}]\cdot [S_{p}\cap H,N_{G}(S_{p}\cap H)])^{\vee}}{(S_{p}/[S_{p},G])^{\vee}}, 
\end{equation*}
which is is archived by using Mackey's decomposition (Proposition \ref{prop:lciv}). The above description implies the desired assertion for $m=2$. 

Second, we discuss the proof of (ii). It suffices to prove that $\widehat{\N}_{G}$ in the exact sequence \eqref{eq:sele} induces the desired isomorphism. By the normality of $S_{p}$, the group $G$ can be expressed as a semi-direct product of $S_{p}$ and its subgroup $G'$ whose order is coprime to $p$. Then we get a sequence of abelian groups
\begin{equation*}
\Sha_{\cD_{/S_{p}}}^{2}(G/S_{p},J_{G/HS_{p}})\cong \Sha_{\cD}^{2}(G,J_{G/HS_{p}})
\xrightarrow{\widehat{\N}_{G}}\Sha_{\cD}^{2}(G,J_{G/H})^{(p)} \xrightarrow{\Res_{G/G'}}\Sha_{\cD_{G'}}^{2}(G',J_{G/H}),
\end{equation*}
where the homomorphisms are injective. Using Mackey's decomposition, one can prove that $J_{G/H}$ is a direct sum of $J_{G'/H'}$ and finitely many $G'$-lattices of the forms $\Ind_{H''}^{G'}\Z$ for some $H''<G'$. Moreover, the natural isomorphism $G'\cong G/S_{p}$ induces a bijection $\cD_{G'}\cong \cD_{/S_{p}}$. Hence we obtain an isomorphism
\begin{equation*}
\Sha_{\cD_{/S_{p}}}^{2}(G/S_{p},J_{G/HS_{p}})\cong \Sha_{\cD_{G'}}^{2}(G',J_{G/H}),
\end{equation*}
which concludes the desired assertion. 

\begin{rem}
One can produce \eqref{eq:itex} from $E/F$ in Theorem \ref{mth5} as follows. Recall that $\widetilde{E}/F$ is the Galois closure of $E/F$, $G:=\Gal(\widetilde{E}/F)$ and $H:=\Gal(\widetilde{E}/E)$. By assumption, a $p$-Sylow subgroup $S_{p}$ of $G$ is normal. Now, let $E_{0}/F$ be the subextension of $\widetilde{E}/F$ corresponding to $HS_{p}$. Then the norm map $\N_{E/E_{0}}$ induces an exact sequence of $F$-tori
\begin{equation*}
1\rightarrow \Res_{E_{0}/F}T_{E/E_{0}}\rightarrow T_{E/F}\xrightarrow{\N_{E/E_{0}}} T_{E_{0}/F} \rightarrow 1. 
\end{equation*}
Taking character groups to these tori, we obtain the desired sequence. 
\end{rem}

\begin{oftp}
First, we review basic facts on group cohomology, flasque resolutions, and their relation with the Hasse norm principle in Section \ref{sect:prlm}. Next, we recall in Section \ref{sect:igpd} previous results on the inverse Galois problem with conditions on decomposition groups. In Section \ref{sect:trgp}, we recall the notion of transitive groups, and give some properties on them. In Section \ref{sect:mnot}, we specify some known results on the cohomology groups of character groups of multinorm one tori. Section \ref{sect:thts} is the technical heart of this paper, which gives a proof of Theorem \ref{mth9}. In Section \ref{sect:fnrp}, we study $2$-dimensional $\Fp$-representations of finite groups of order coprime to $p$, which will be needed to prove the main theorems. Finally, in Section \ref{sect:pfmt}, we give proofs of the main theorems using Theorem \ref{mth9} and the results in Section \ref{sect:fnrp}. Also, we compare our theorems with previous results of Hoshi--Kanai--Yamasaki on the Hasse norm principle for degree $\leq 15$. 
\end{oftp}

\begin{ack}
The author would like to thank Akinari Hoshi, Kazuki Kanai, Yoichi Mieda, Seidai Yasuda and Chia-Fu Yu for various helpful comments. Moreover, he also appreciate Noriyuki Abe for answering my questions on representation theory of finite groups. This work was carried out with the support of the JSPS Research Fellowship for Young Scientists and KAKENHI Grant Number JP22KJ0041.
\end{ack}

\begin{nota}
Let $G$ be a finite group. 
\begin{itemize}
\item Set $G^{\vee}:=\Hom(G,\Q/\Z)$. It is the Pontryagin dual if $G$ is abelian. For any normal subgroup $N$ of $G$, the natural surjection $G\twoheadrightarrow G/N$ induces an isomorphism
\begin{equation*}
(G/N)^{\vee}\xrightarrow{\cong} \{f\in G^{\vee}\mid f\!\mid_{N}=0\}. 
\end{equation*}
We regard $(G/N)^{\vee}$ as a subgroup of $G^{\vee}$ by the above isomorphism. 
\item For each $g\in G$, let $\Ad(g)\colon G\rightarrow G;\,h\mapsto ghg^{-1}$, which is an automorphism of $G$. 
\end{itemize}
\end{nota}

\section{Preliminaries}\label{sect:prlm}

\subsection{Cohomology of finite groups}\label{ssec:chfg}

Let $G$ be a finite group. For a \emph{$G$-module}, we mean a finitely generated abelian groups equipped with \emph{left} actions of $G$. Moreover, a $G$-module that is torsion-free over $\Z$ is called a \emph{$G$-lattice}. On the other hand, for a subgroup $H$ of $G$ and an $H$-module $M$, put $\Ind_{H}^{G}M:=\Hom_{\Z[H]}(\Z[G],M)$. We define a left action of $G$ on $\Ind_{H}^{G}M$ by the map
\begin{equation*}
G \times \Ind_{H}^{G}M \rightarrow \Ind_{H}^{G}M; (g,\varphi) \mapsto [g'\mapsto \varphi(g'g)]. 
\end{equation*}

\begin{lem}[{\cite[(1.6.5) Proposition]{Neukirch2000}}]\label{lem:shpr}
Let $G$ be a finite group, and $H$ its subgroup. Consider a $G$-module $M$. Then the composite
\begin{equation*}
H^{j}(G,M)\rightarrow H^{j}(G,\Ind_{H}^{G}M)\cong H^{j}(H,M),
\end{equation*}
where the second homomorphism follows from Shapiro's lemma, coincides with the restriction map $\Res_{G/H}$. 
\end{lem}

\begin{prop}[{Mackey's decomposition; cf.~\cite[Section 7.3, Proposition 22]{Serre1977}}]\label{prop:mcky}
Let $G$ be a finite group, and $H$ and $D$ subgroups of $G$. Take a complete representative $R(D,H)$ of $D\backslash G/H$ in $G$. Consider an $H$-module $(M,\rho)$. Then the homomorphism of abelian groups
\begin{equation*}
\pi_{D,g}\colon \Ind_{H}^{G}M\rightarrow \Ind_{D\cap gHg^{-1}}^{D}M^{g};\varphi \mapsto [d\mapsto \varphi(g^{-1}d)]
\end{equation*}
for all $g\in R(D,H)$ induces an isomorphism of $D$-modules
\begin{equation*}
\Ind_{H}^{G}M \cong \bigoplus_{g\in R(D,H)}\Ind_{D\cap gHg^{-1}}^{D}M^{g}. 
\end{equation*}
Here $M^{g}$ is the abelian group $M$ equipped with the action of $gHg^{-1}$ defined by
\begin{equation*}
gHg^{-1}\rightarrow \Aut(M); h\mapsto (\rho \circ \Ad(g^{-1}))(h). 
\end{equation*}
\end{prop}

The independence of the choice of $R(D,H)$ in Proposition \ref{prop:mcky} follows from Lemma \ref{lem:doub} as follows. 

\begin{lem}\label{lem:doub}
Let $G$ be a finite group, and $H$ and $D$ subgroups of $G$. 
\begin{enumerate}
\item Let $g\in G$. For $d,d'\in D$, we have $dgH=d'gH$ if and only if $d'd^{-1}\in D\cap gHg^{-1}$. In particular, the subset $DgH$ of $G$, which is an element of $D\backslash G/H$, has cardinality $(\#D \cdot \#H)/\#(D\cap gHg^{-1})$. 
\item Assume that $g_1,g_2 \in G$ satisfies $Dg_{1}H=Dg_{2}H$. Then the subgroups $D \cap g_{1}Hg_{1}^{-1}$ and $D\cap g_{2}Hg_{2}^{-1}$ of $D$ are conjugate in $D$. 
\end{enumerate}
\end{lem}

\begin{proof}
(i): Let $d,d'\in D$. The equality $dgH=d'gH$ is equivalent to the condition $g^{-1}d^{-1}d'g\in H$. Moreover, Since $d$ and $d'$ lie in $D$, the condition $g^{-1}d^{-1}d'g \in H$ holds if and only if the desired condition $d'd^{-1}\in D\cap gHg^{-1}$ is satisfied. 

(ii): By assumption, there are $h\in H$ and $d\in D$ such that $g_2=dg_{1}h$. Then we have
\begin{equation*}
D\cap g_{2}Hg_{2}^{-1}=dDd^{-1}\cap dg_{1}(hHh^{-1})g_{1}^{-1}d^{-1}=d(D\cap g_{1}Hg_{1}^{-1})d^{-1}. 
\end{equation*}
Hence the assertion holds. 
\end{proof}

We also need a description of the restriction map on the cohomology of $\Ind_{H}^{G}M$. 

\begin{lem}\label{lem:hmts}
Let $G$ be a finite group, $H$ a subgroup of $G$, and $M$ a $G$-module. 
\begin{enumerate}
\item Consider a homomorphism of abelian groups
\begin{equation*}
\Phi_{G/H,M}\colon \Ind_{H}^{G}M \xrightarrow{\cong} \Z[G]\otimes_{\Z[H]}M; \varphi \mapsto \sum_{g\in R}g^{-1}\otimes \varphi(g),
\end{equation*}
where $R$ is a complete representative of $H\backslash G$ in $G$. Then the map $\Phi_{G/H,M}$ is independent of the choice of $R$, and is a homomorphism of $G$-modules. 
\item Under the notations in Proposition \ref{prop:mcky}, the composite of $D$-homomorphisms
\begin{align*}
\Ind_{D\cap gHg^{-1}}^{D}M^{g} & \xrightarrow{\Phi_{D/(D\cap gHg^{-1}),M^{g}}} \Z[D]\otimes_{\Z[D\cap gHg^{-1}]}M^{g} \xrightarrow{1\otimes x\mapsto g\otimes x}\Z[G]\otimes_{\Z[H]}M \\
& \xrightarrow{\Phi_{G/H,M}^{-1}} \Ind_{H}^{G}M \xrightarrow{\pi_{D,g}} \Ind_{D\cap gHg^{-1}}^{D}M^{g}
\end{align*}
coincides with the identity map on $\Ind_{D\cap gHg^{-1}}^{D}M^{g}$. 
\end{enumerate}
\end{lem}

\begin{proof}
The assertions can be proved by direct computation. 
\end{proof}

\begin{lem}\label{lem:gpac}
Let $G$ be a finite group, and $M$ a $G$-module. Pick a subgroup $D$ of $G$ and a direct summand $M_{0}$ of $M$, and denote by $\pr_{0}\colon M\rightarrow M_{0}$ the projection as $H$-modules. Then, for any $g\in G$ and $j\in \Znn$, one has a commutative diagram
\begin{equation*}
\xymatrix@C=60pt{
H^{j}(G,M)\ar[r]^{\Res_{G/D}}\ar@{=}[d] & H^{j}(D,M)\ar[r]^{\pr_{0*}}\ar[d]^{\cong} & H^{j}(D,M_{0})\ar[d]^{\cong}\\
H^{j}(G,M)\ar[r]^{\Res_{G/gDg^{-1}}\hspace{15pt}}& H^{j}(gDg^{-1},M)\ar[r]^{g\circ \pr_{0*}\circ g^{-1}}& H^{j}(gDg^{-1},g(M_{0})). 
}
\end{equation*}
Here the right vertical homomorphisms are induced by the inner automorphism $\Ad(g^{-1})$ on $G$ and the homomorphism $M\rightarrow M;x\mapsto gx$. 
\end{lem}

\begin{proof}
By \cite[Chap.~I, \S5, 1.~Conjugation]{Neukirch2000}, we obtain a commutative diagram
\begin{equation*}
\xymatrix@C=50pt{
H^{j}(G,M)\ar[d]\ar[r]^{\Res_{G/D}}& H^{j}(D,M)\ar[r]^{\pr_{0*}}\ar[d]^{\cong}& H^{j}(D,M_{0})\ar[d]^{\cong}\\
H^{j}(G,M)\ar[r]^{\Res_{G/gDg^{-1}}\hspace{10pt}}& H^{j}(gDg^{-1},M)\ar[r]^{g\circ \pr_{0*}\circ g^{-1}}& H^{j}(gDg^{-1},g(M_{0})). 
}
\end{equation*}
The horizontal maps are isomorphisms by construction. On the other hand, the upper horizontal map coincides with the identity map by \cite[Chap.~VII, \S 5, Proposition 3]{Serre1979}. Hence the proof is complete. 
\end{proof}

\begin{prop}\label{prop:htch}
Let $G$ be a finite group, and $H$ its subgroup. Then, there is a commutative diagram
\begin{equation*}
\xymatrix{
H^{2}(G,\Z)\ar[r]\ar[d]^{\cong}& H^{2}(G,\Ind_{H}^{G}\Z)\ar[d]^{\cong}\\
G^{\vee}\ar[r]^{f\mapsto f\mid_{H}}& H^{\vee}. }
\end{equation*}
Here, the top horizontal map is induced by the homomorphism of $G$-lattices
\begin{equation*}
\varepsilon_{G/H}\colon \Z \rightarrow \Ind_{H}^{G}\Z
\end{equation*}
defined by $\varepsilon_{G/H}(1)(g)=1$ for any $g\in G$. 
\end{prop}

\begin{proof}
By Lemma \ref{lem:shpr}, the composite
\begin{equation*}
H^{2}(G,\Z)\rightarrow H^{2}(G,\Ind_{H}^{G}\Z)\rightarrow H^{2}(H,\Z)
\end{equation*}
coincides with the restriction map $\Res_{G/H}$. Hence, the assertion follows from \cite[p.~52]{Neukirch2000}. 
\end{proof}

\begin{prop}[{cf.~\cite[Section 6, pp.~310--311]{Platonov1994}}]\label{prop:rsch}
Let $G$ be a finite group, and $H$ and $D$ subgroups of $G$. Then the following diagram is commutative: 
\begin{equation*}
\xymatrix@C=90pt{
H^{2}(G,\Ind_{H}^{G}\Z)\ar[d]^{\cong}\ar[r]^{\Res_{G/D}}& H^{2}(D,\Ind_{H}^{G}\Z) \ar[d]^{\cong}\\
H^{\vee}\ar[r]^{f\mapsto ((f\circ \Ad(g^{-1}))\mid_{D\cap gHg^{-1}})_{g}\hspace{20mm}}& \bigoplus_{g\in R(D,H)}(D\cap gHg^{-1})^{\vee}. 
}
\end{equation*}
Here $R(D,H)$ is a complete representative of $D\backslash G/H$ in $G$.
\end{prop}

\begin{proof}
Note that $\Z$ is a direct summand of $\Ind_{H}^{G}\Z$ as an $H$-module, which is a consequence of Lemmas \ref{lem:hmts} (ii). Take $g\in R(D,H)$, then we have a commutative diagram as follows by Lemma \ref{lem:gpac}: 
\begin{equation*}
\xymatrix@C=50pt{
H^{2}(G,\Ind_{H}^{G}\Z)\ar[r]^{\Res_{G/H}}\ar@{=}[d] & H^{2}(H,\Ind_{H}^{G}\Z)\ar[r]^{\varphi \mapsto \varphi(1)} \ar[d]^{\cong} & H^{2}(H,\Z)\ar[d]^{\cong}\\
H^{2}(G,\Ind_{H}^{G}\Z)\ar[r]^{\Res_{G/gHg^{-1}}\hspace{13pt}}& H^{2}(gHg^{-1},\Ind_{H}^{G}\Z)\ar[r]^{\hspace{13pt}\varphi \mapsto \varphi(g^{-1})}& H^{2}(gHg^{-1},\Z). 
}
\end{equation*}
Here the vertical isomorphisms are induced by $\Ad(g^{-1})$ and the action of $g$ on $\Ind_{H}^{G}\Z$. Moreover, the composite of the upper horizontal maps is an isomorphism by Shapiro's lemma. On the other hand, the following is commutative: 
\begin{equation*}
\xymatrix@C=40pt{
H^{2}(gHg^{-1},\Ind_{H}^{G}\Z)\ar[r]^{\varphi \mapsto \varphi(g^{-1})} \ar[d]^{\pi_{D,g*}} & H^{2}(gHg^{-1},\Z)\ar[d]^{\Res_{gHg^{-1}/(D\cap gHg^{-1})}} \\
H^{2}(D\cap gHg^{-1},\Ind_{D\cap gHg^{-1}}^{D}\Z)\ar[r]^{\hspace{25pt}\varphi \mapsto \varphi(1)} & H^{2}(D\cap gHg^{-1},\Z). 
}
\end{equation*}
Here the lower horizontal map is an isomorphism by Shapiro's lemma. In summary, we obtain a commutative diagram
\begin{equation*}
\xymatrix@C=30pt{
H^{2}(G,\Ind_{H}^{G}\Z)\ar[d]^{\cong} \ar[r]^{\pi_{D,g*}\hspace{15pt}} & H^{2}(D,\Ind_{D\cap gHg^{-1}}^{D}\Z)\ar[d]^{\Res_{gHg^{-1}/(D\cap gHg^{-1})}} \\
H^{2}(H,\Z)\ar[r]^{\Ad(g^{-1})^{*}\hspace{25pt}} & H^{2}(D\cap gHg^{-1},\Z). 
}
\end{equation*}
Furthermore, Proposition \ref{prop:htch} implies that the lower horizontal homomorphism can be written as 
\begin{equation*}
H^{\vee}\rightarrow (D\cap gHg^{-1})^{\vee};\,f\mapsto (f\circ \Ad(g^{-1}))\mid_{D\cap gHg^{-1}}. 
\end{equation*}
Consequently, the assertion follows from Proposition \ref{prop:mcky}. 
\end{proof}

\subsection{Specific subgroups of cohomology groups}

Let $G$ be a finite group, and $M$ a $G$-module. For a finite set $\cD$ of subgroups of $G$ and $j\in \Znn$, put
\begin{equation*}
\Sha_{\cD}^{j}(G,M):=\Ker\left(H^{j}(G,M)\xrightarrow{(\Res_{G/D})_{D}} \bigoplus_{D\in \cD}H^{j}(D,M) \right). 
\end{equation*}
Here, $\Res_{G/D} \colon H^{j}(G,M)\rightarrow H^{j}(D,M)$ is the restriction map. We set $\Sha_{\emptyset}^{j}(G,M):=H^{j}(G,M)$ for convention. Moreover, we denote $\Sha_{\cC_{G}}^{j}(G,M)$ by $\Sha_{\omega}^{j}(G,M)$. 

\vspace{5pt}

For an abelian group $A$ and a prime number $p$, we use the following notations: 
\begin{align*}
A[p^{\infty}]&:=\{a\in A\mid p^{m}a=0\text{ for some }m\in \Zpn\},\\
A^{(p)}&:=\{a\in A\mid na=0\text{ for some $n \in \Z\setminus p\Z$}\}. 
\end{align*}
We call $A[p^{\infty}]$ and $A^{(p)}$ the \emph{$p$-primary torsion part} and the \emph{prime-to-$p$ torsion part} of $A$, respectively. 

\begin{prop}\label{prop:tsan}
Let $G$ be a finite group, $M$ a $G$-lattice, $H$ a subgroup of $G$, and $j$ a positive ineteger. Consider a finite set $\cD$ of subgroups of $G$, and set 
\begin{equation*}
\cD_{\cap H}:=\{D\cap H \mid D\in \cD\}. 
\end{equation*}
\begin{enumerate}
\item If $\Sha_{\cD_{\cap H}}^{j}(H,M)=0$, then $\Sha_{\cD}^{j}(G,M)$ is annihilated by $(G:H)$. 
\item Let $p$ be a prime number that does not divide $(G:H)$. Then the composite
\begin{equation*}
\Sha_{\cD}^{j}(G,M)[p^{\infty}]\hookrightarrow \Sha_{\cD}^{j}(G,M)\xrightarrow{\Res_{G/H}} 
\Sha_{\cD_{\cap H}}^{j}(H,M)
\end{equation*}
is injective. 
\item Suppose that $(G:H)$ is a power of a prime number $p$. Then the composite
\begin{equation*}
\Sha_{\cD}^{j}(G,M)^{(p)}\hookrightarrow \Sha_{\cD}^{j}(G,M)\xrightarrow{\Res_{G/H}} \Sha_{\cD_{\cap H}}^{j}(H,M). 
\end{equation*}
is injective. 
\end{enumerate}
\end{prop}

\begin{proof}
These follow from \cite[Chap.~VII, \S 7, Proposition 6]{Serre1979}, which asserts that the composite
\begin{equation*}
H^{j}(G,M)\xrightarrow{\Res_{G/H}} H^{j}(H,M)\xrightarrow{\Cor_{G/H}}H^{j}(G,M)
\end{equation*}
coincides with the multiplication by $(G:H)$. 
\end{proof}

We introduce the notion for finite sets of subgroups in $G$ that contains set of decomposition groups in the Galois groups of finite Galois extensions of global fields. 

\begin{dfn}
Let $G$ be a finite group. We say that a finite set of subgroups $\cD$ of $G$ is \emph{admissible} if 
\begin{itemize}
\item $\cD$ contains all cyclic subgroups of $G$ (in particular, $\cD\neq \emptyset$); and
\item $D\in \cD$ implies $gDg^{-1}\in \cD$ for every $g\in G$. 
\end{itemize}
\end{dfn}

It is clear that the set of all cyclic groups $\cC_{G}$ of $G$ is admissible. Moreover, we can confirm the following without difficulty: 

\begin{lem}\label{lem:adsq}
Let $G$ be a finite group, and $\cD$ an admissible set of subgroups of $G$. 
\begin{enumerate}
\item For a subgroup $H$ of $G$, let $\cD_{\cap H}$ be the finite set of subgroups in $H$ as in Proposition \ref{prop:tsan}. 
Then, it is admissible. 
\item For a normal subgroup $N$ of $G$, set
\begin{equation*}
\cD_{/N}:=\{DN/N \mid D\in \cD\}, 
\end{equation*}
which is a finite set of subgroups of $G/N$. Then it is admissible. 
\end{enumerate}
\end{lem}

\begin{lem}\label{lem:qtts}
Let $G$ be a finite group, and $H$ its subgroup. Then one has\
\begin{equation*}
\Sha_{\omega}^{2}(G,\Ind_{H}^{G}\Z)=0. 
\end{equation*}
\end{lem}

\begin{proof}
Take a complete representative $R(D,H)$ of $D\backslash G/H$ which contains $1$. By Proposition \ref{prop:rsch}, it suffices to prove the following: 
\begin{itemize}
\item If $f\in H^{\vee}$ satisfies $(f\circ \Ad(g^{-1}))\!\mid_{D\cap gHg^{-1}}$ is trivial for any $D\in \cC_{G}$ and $g\in R(D,H)$, then $f=0$. 
\end{itemize}
Take $f\in H^{\vee}$. For each $h\in H$, the assumption for $D=\langle h\rangle$ and $1 \in R(D,H)$ implies $f(h)=0$. Hence, we obtain $f=0$ as desired. 
\end{proof}

\begin{prop}\label{prop:ifts}
Let $G$ be a finite group, and $N$ its normal subgroup. Take an admissible set of subgroups $\cD$ of $G$ and a $G/N$-module $M$ which is torsion-free as an abelian group. Then the inflation map $\Inf_{G/(G/N)}\colon H^{2}(G/N,M)\rightarrow H^{2}(G,M)$ induces an isomorphism
\begin{equation*}
\Sha_{{\cD}_{/N}}^{2}(G/N,M) \cong \Sha_{\cD}^{2}(G,M). 
\end{equation*}
\end{prop}

\begin{proof}
We follow the proof of \cite[Lemma 3.2 (ii)]{Huang2024}. Take $D \in \cD$ and $\widetilde{D} \in q_{\widetilde{G}/G}^{-1}(D)$, and put $D':=\widetilde{D} \cap N$. Since $N$ acts on $M$ trivially, we have the following: 
\begin{equation*}
H^{1}(N,M)=\Hom(N,M),\quad H^{1}(D',M)\cong \Hom(D',M). 
\end{equation*}
Moreover, the assumption that $M$ is torsion-free implies that $\Hom(N,M)$ and $\Hom(D',M)$ are trivial. Hence, by the Hochschild--Serre spectral sequence, we obtain commutative diagrams
\begin{gather*}
\xymatrix@C=30pt{
0\ar[r] & H^{1}(G,M)\ar[d]_{\Res_{G/D}}\ar[r]^{\Inf_{\widetilde{G}/G}}& H^{1}(\widetilde{G},M)\ar[d]^{\Res_{\widetilde{G}/\widetilde{D}}}\ar[r]& 0\\
0\ar[r] & H^{1}(D,M)\ar[r]^{\Inf_{\widetilde{D}/D}}& H^{1}(\widetilde{D},M) \ar[r]& 0,
}\\
\xymatrix@C=40pt{
0\ar[r] & H^{2}(G,M)\ar[d]^{\Res_{G/D}}\ar[r]^{\Inf_{\widetilde{G}/G}}& H^{2}(\widetilde{G},M)\ar[d]^{\Res_{\widetilde{G}/\widetilde{D}}} \ar[r]^{\Res_{\widetilde{G}/N}}& H^{2}(N,M) \ar[d]^{\Res_{N/D'}}\\
0\ar[r] & H^{2}(D,M)\ar[r]^{\Inf_{\widetilde{D}/D}}& H^{2}(\widetilde{D},M) \ar[r]^{\Res_{\widetilde{D}/D'}}& H^{2}(D',M). 
}
\end{gather*}
Here all the horizontal sequences are exact. Hence, the assertion for $i=1$ holds. On the other hand, the set $\{D'\cap N'\mid D'\in \cD\}$ contains the set $\cC_{N'}$, which follows from the inclusion $\cC_{G'}\subset \cD'$. Hence, if $\Sha_{\omega}^{2}(N',M)=0$, then we obtain the desired isomorphism for $i=2$. It suffices to prove the triviality of $\Sha_{\omega}^{2}(N',\Z)$. However, this follows from Lemma \ref{lem:qtts} for $G=H=N'$. 
\end{proof}

\subsection{Flasque resolutions}

\begin{dfn}[{\cite[\S 1]{ColliotThelene1977}}]
Let $G$ be a finite group. 
\begin{enumerate}
\item We say that $G$-lattice $P$ is \emph{permutation} (or, \emph{induced}) if $P=\bigoplus_{i=1}^{r}\Ind_{H_i}^{G}\Z$ for some subgroups $H_1,\ldots,H_r$ of $G$. 
\item A $G$-lattice $Q$ is said to be \emph{flasque} (or, \emph{flabby}) if $\widehat{H}^{-1}(H,Q)=0$ for any subgroup $H$ of $G$. Here, $\widehat{H}^{-1}$ denotes the $(-1)$-st Tate cohomology. 
\item A \emph{flasque resolution} (or, a \emph{flabby resolution}) of $M$ is an exact sequence of $G$-modules
\begin{equation*}
0 \rightarrow M \rightarrow P \rightarrow Q \rightarrow 0,
\end{equation*}
where $P$ is permutation and $Q$ is flasque. 
\end{enumerate}
\end{dfn}

Note that a flasque resolution exists for any $G$-lattice, which follows from \cite[Lemma 1.1]{Endo1975} or \cite[Lemme 3]{ColliotThelene1977}. 

\begin{lem}\label{lem:wrfl}
Let $G$ be a finite group, $H$ a subgroup of $G$, and $M$ a $G$-lattice. Then a flasque resolution $0 \rightarrow M\rightarrow P \rightarrow Q \rightarrow 0$ of $M$ induces a flasque resolution of $\Ind_{H}^{G}M$: 
\begin{equation*}
0 \rightarrow \Ind_{H}^{G}M \rightarrow \Ind_{H}^{G}P \rightarrow \Ind_{H}^{G}Q \rightarrow 0. 
\end{equation*}
\end{lem}

\begin{proof}
It suffices to prove $\widehat{H}^{-1}(D,\Ind_{H}^{G}Q)=0$ for every subgroup $D$ of $G$. There is an isomorphism
\begin{equation*}
\widehat{H}^{-1}(D,\Ind_{H}^{G}Q)\cong \bigoplus_{g\in R(D,H)}\widehat{H}^{-1}(D,\Ind_{D\cap gHg^{-1}}^{D}Q^{g}),
\end{equation*}
which follows from Proposition \ref{prop:mcky}. On the other hand, Shaipro's lemma gives rise to an isomorphism
\begin{equation*}
\widehat{H}^{-1}(D,\Ind_{D\cap gHg^{-1}}^{D}Q^{g})\cong \widehat{H}^{-1}(D\cap gHg^{-1},Q^{g})
\end{equation*}
for any $g\in R(D,H)$. Moreover, the right-hand side is isomorphic to $\widehat{H}^{-1}(g^{-1}Dg\cap H,Q)$. However, $\widehat{H}^{-1}(g^{-1}Dg\cap H,Q)$ is trivial since $Q$ is flasque. Hence, we obtain the desired assertion. 
\end{proof}

\begin{prop}\label{prop:enmy}
Let $G$ be a finite group, and $p$ a prime divisor of $G$. Assume that a $p$-Sylow subgroup of $G$ is cyclic. Then we have $H^{1}(G,Q)[p^{\infty}]=0$ for any flasque $G$-lattice $Q$. 
\end{prop}

\begin{proof}
Take a $p$-Sylow subgroup $S_{p}$ of $G$, which is cyclic by assumption. Then \cite[Theorem 1.5]{Endo1975} implies that $Q$ is a direct summand of a permutation $S_{p}$-lattice. Hence $H^{1}(S_{p},Q)$ is trivial. Since $(G:S_{p})$ is not divisible by $p$, the assertion follows from Proposition \ref{prop:tsan} (ii). 
\end{proof}

\begin{prop}[{cf.~\cite[Proposition 9.8]{Sansuc1981}}]\label{prop:snsc}
Let $G$ be a finite group, and $M$ a $G$-lattice. Take a flasque resolution of $M$: 
\begin{equation*}
0 \rightarrow M \rightarrow P\rightarrow Q \rightarrow 0. 
\end{equation*}
\begin{enumerate}
\item For an admissible set $\cD$ of subgroups of $G$, there is an isomorphism of finite abelian groups
\begin{equation*}
\Sha_{\cD}^{1}(G,Q) \cong {\Sha}_{\cD}^{2}(G,M). 
\end{equation*}
\item There is an isomorphism of finite abelian groups
\begin{equation*}
H^{1}(G,Q) \cong {\Sha}_{\omega}^{2}(G,M). 
\end{equation*}
\end{enumerate}
\end{prop}

\begin{proof}
We follow the argument in \cite[8.3, pp.~97--98]{Voskresenskii1998}. For any subgroup $D$ of $G$, the given flasque resolution of $M$ induces a commutative diagram
\begin{equation*}
\xymatrix{
0\ar[r]& H^{1}(G,Q)\ar[d]^{\Res_{G/D}} \ar[r]& H^{2}(G,M)\ar[d]^{\Res_{G/D}} \ar[r]& H^{2}(G,P)\ar[d]^{\Res_{G/D}}\\
0\ar[r]& H^{1}(D,Q) \ar[r]& H^{2}(D,M) \ar[r]& H^{2}(D,P). }
\end{equation*}

(i): It suffices to prove $\Sha_{\omega}^{2}(G,P)=0$. This assertion follows from Lemma \ref{lem:qtts} since $P$ is permutation. 

(ii): If $D$ is cyclic, then we have $H^{1}(D,Q)=0$ by Proposition \ref{prop:enmy}. Combining this with (i), we obtain the desired isomorphism. 
\end{proof}

\begin{cor}\label{cor:cyan}
Let $G$ be a finite group of which all Sylow subgroups are cyclic. Then
\begin{equation*}
\Sha_{\omega}^{2}(G,M)=0
\end{equation*}
for any $G$-lattice $M$. 
\end{cor}

\begin{proof}
Take a flasque resolution $0 \rightarrow M \rightarrow P \rightarrow Q \rightarrow 0$ of $M$. Then we have $H^{1}(G,Q)=0$ by Proposition \ref{prop:enmy}. Hence, Proposition \ref{prop:snsc} implies the desired assertion. 
\end{proof}

\begin{cor}\label{lem:ids2}
Let $G$ be a finite group, $H$ a subgroup of $G$. For a $H$-lattice $M$, there is an isomorphism
\begin{equation*}
\Sha_{\omega}^{2}(G,\Ind_{H}^{G}M)\cong \Sha_{\omega}^{2}(H,M). 
\end{equation*}
\end{cor}

\begin{proof}
Take a flasque resolution $0\rightarrow M \rightarrow P \rightarrow Q \rightarrow 0$ of $M$. By Lemma \ref{lem:wrfl}, this induces a flasque resolution of $\Ind_{H}^{G}M$: 
\begin{equation*}
0 \rightarrow \Ind_{H}^{G}M \rightarrow \Ind_{H}^{G}P \rightarrow \Ind_{H}^{G}Q \rightarrow 0. 
\end{equation*}
Moreover, Proposition \ref{prop:snsc} (ii) implies isomorphisms
\begin{equation*}
H^{1}(G,\Ind_{H}^{G}Q)\cong \Sha_{\omega}^{2}(G,\Ind_{H}^{G}M),\quad 
H^{1}(H,Q)\cong \Sha_{\omega}^{2}(H,M). 
\end{equation*}
On the other hand, we have $H^{1}(G,\Ind_{H}^{G}Q)\cong H^{1}(H,Q)$ by Shapiro's lemma, and hence the assertion holds. 
\end{proof}

\subsection{Relations with cohomological invariants and Tate--Shafarevich groups}

Let $F$ be a field, and fix a separable closure $F^{\sep}$ of $F$. For a finite Galois extension $F'/F$, denote by $\Gal(F'/F)$ the Galois group of $F'/F$. 

Denote by $\G_{m}:=\Spec F[t^{\pm 1}]$ the multiplicative group scheme over $F$. Recall that a \emph{torus} (or, an \emph{algebraic torus}) over $F$ is an algebraic group scheme over $F$ that satisfies $T\otimes_{F}F^{\sep}\cong \G_{m,F^{\sep}}^{N}$ for some $N\in \Znn$. For a torus $T$ over $F$, the \emph{character group} of $T$ is defined as follows: 
\begin{equation*}
X^{*}(T):=\Hom_{F^{\sep}\text{-group}}(T\otimes_{F}F^{\sep},\G_{m,F^{\sep}}). 
\end{equation*}
If $T$ splits over a finite Galois extension $E$ of $F$, then $X^{*}(T)$ is a $\Gal(E/F)$-lattice. 

\vspace{6pt}
Let $X$ a smooth proper variety over $F$. We denote by $\Br(X)$ the Brauer group of $X$, that is, 
\begin{equation*}
\Br(X):=H_{\et}^{2}(X,\G_{m}). 
\end{equation*}

\begin{prop}[{\cite[Theorem 9.5 (ii)]{ColliotThelene1987}, \cite[Theorem 2.3, Theorem 2.4]{BayerFluckiger2020}}]\label{prop:host}
Let $F$ be a field, and $T$ a torus over $F$ that splits over a finite Galois extension $\widetilde{E}$ of $F$. Put $G:=\Gal(\widetilde{E}/F)$. Take a smooth compactification $X$ of $T$ over $F$, and put $\overline{X}:=X\otimes_{F}F^{\sep}$. Then there exist isomorphisms
\begin{equation*}
\Br(X)/\Br(F)\cong H^{1}(F,\Pic(\overline{X}))\cong \Sha_{\omega}^{2}(G,X^{*}(T)). 
\end{equation*}
\end{prop}

We remark that a smooth compactifications exist for all tori over $F$. It follows from Hironaka (\cite{Hironaka1964}) if $F$ has characteristic zero. On the other hand, if $F$ has positive characteristic, the assertion is a consequence of Colliot-Th{\'e}l{\`e}ne--Harari--Skorobogatov (\cite{ColliotThelene2005}). 

\vspace{5pt}
In the sequel of this section, let $k$ be a global field. Fix a separable closure $k^{\sep}$ of $k$. For a finite Galois extension $K$ of $k$, write $\Sigma_{K}$ for the set of places of $K$. 

Let $M$ be an abelian group equipped with an action of $\Gal(k^{\sep}/k)$. For $j\in \Znn$, put
\begin{equation*}
\Sha^{j}(k,M):=\Ker\left(H^{j}(k,M)\xrightarrow{(\Res_{k_v/k})_{v}}\prod_{v\in \Sigma_{k}}H^{j}(k_{v},M)\right),
\end{equation*}
where $\Res_{k_v/k}\colon H^{j}(k,M)\rightarrow H^{j}(k_{v},M)$ denotes the restriction map for each $v\in \Sigma_{k}$. In particular, $\Sha^{1}(k,M)$ is called the \emph{Tate--Shafarevich group} of $M$. 

\vspace{5pt}
For a torus $T$ over $k$, there is a direct relation between $H^{1}(k,\Pic(\overline{X}))$ and $\Sha^{1}(k,T)$ as follows, where $X$ is a smooth compactification of $T$ over $k$ and $\overline{X}:=X\otimes_{k}k^{\sep}$. 

\begin{prop}[{\cite[Theorem 5]{Voskresenskii1969}}]\label{prop:vskr}
Let $k$ be a global field, and $T$ a torus over $k$. Then there is an exact sequence
\begin{equation}\label{eq:vsex}
0\rightarrow A_{k}(T)\rightarrow H^{1}(k,\Pic(\overline{X}))^{\vee}\rightarrow \Sha^{1}(k,T)\rightarrow 0, 
\end{equation}
where $X$ is a smooth compactification of $T$ over $k$, $\overline{X}:=X\otimes_{k}k^{\sep}$ and $A_{k}(T)$ is the quotient of $\prod_{v\in \Sigma_{k}}T(k_{v})$ by the closure of $T(k)$. 
\end{prop}

\begin{prop}[{Poitou--Tate duality; \cite[(8.6.8) Proposition]{Neukirch2000}}]\label{prop:pttt}
Let $k$ be a global field, and $T$ a torus over $k$. Then, there is an isomorphism of finite groups
\begin{equation*}
\Sha^{1}(k,T)\cong \Sha^{2}(k,X^{*}(T))^{\vee}. 
\end{equation*}
\end{prop}

\begin{prop}\label{prop:almt}
Let $k$ be a global field, and $T$ a $k$-torus which splits over a finite Galois extension $K$ of $k$. Put $G:=\Gal(K/k)$. Denote by $\cD$ the set of decomposition groups of $K/k$, which is an admissible set of subgroups of $G$. Then there is an isomorphism
\begin{equation*}
\Sha^{1}(k,T)\cong \Sha_{\cD}^{2}(G,X^{*}(T))^{\vee}. 
\end{equation*}
\end{prop}

\begin{proof}
By Proposition \ref{prop:pttt}, it suffices to construct an isomorphism
\begin{equation*}
\Sha^{2}(k,X^{*}(T))\cong \Sha_{\cD}^{2}(G,X^{*}(T)). 
\end{equation*}
Let $\widetilde{K}/k$ be a finite Galois extension that contains $K$. Put $\widetilde{G}:=\Gal(\widetilde{K}/k)$, and denote by $\widetilde{\cD}$ the set of decomposition groups of $\widetilde{K}/k$. Then Proposition \ref{prop:ifts} gives an isomorphism
\begin{equation*}
\Sha_{\cD}^{2}(G,X^{*}(T))\cong \Sha_{\widetilde{\cD}}^{2}(\widetilde{G},X^{*}(T)). 
\end{equation*}
Hence, the assertion follows from this isomorphism. 
\end{proof}

\section{Inverse Galois problem with conditions on decomposition groups}\label{sect:igpd}

In this section, we recall results on the inverse Galois problem which will be needed to prove the main theorems. 

\begin{prop}\label{prop:shf1}
Let $p$ be a prime number, $k$ a global field, and $G$ a finite group. Assume that 
\begin{itemize}
\item the characteristic of $k$ is different from $p$; 
\item a $p$-Sylow subgroup $S_{p}$ of $G$ is normal; and
\item there is a finite Galois extension $\widetilde{K}_{0}/k$ with Galois group $G/S_{p}$ in which all decomposition groups are cyclic. 
\end{itemize}
Then there exists a finite Galois extension $\widetilde{K}/k$ with Galois group $G$ such that
\begin{enumerate}
\item $\widetilde{K}_{0}$ corresponds to $S_{p}$; and
\item all decomposition groups in $\widetilde{K}/k$ are cyclic. 
\end{enumerate}
\end{prop}

\begin{proof}
By the Schur--Zassenhaus theorem, the canonical exact sequence
\begin{equation*}
1\rightarrow S_{p}\rightarrow G \rightarrow G/S_{p}\rightarrow 1
\end{equation*}
splits. Then the assertion follows from a proof of Shafarevich's theorem, which states the inverse Galois theory holds for finite solvable groups. See \cite[(9.6.7) Theorem]{Neukirch2000}. 
\end{proof}

\begin{prop}\label{prop:shf2}
Let $G:=C_{\ell}\times C_{n}$, where $\ell$ is a prime number and $n$ is a positive integer. Then, for any global field $k$, there is a finite Galois extension $\widetilde{K}/k$ with Galois group $G$ in which all decomposition groups are cyclic. 
\end{prop}

\begin{proof}
By Proposition \ref{prop:shf1}, we may assume that 
\begin{itemize}
\item $k$ has characteristic $\ell$; and
\item $n$ is a multiple of $\ell$. 
\end{itemize}
In particular, $k$ is a function field. Let $\F$ be the field of constants in $k$, and denote by $\F'$ the unique finite field extension of $\F$ of degree $n_{0}$. Then the field extension $\F'k/k$ is unramified and satisfies $\Gal(\F'k/k)\cong \Z/n_{0}$. On the other hand, fix a finite separable field extension $k/\F_{\ell}(t)$. Then Chebotarev's density theorem implies that there exists a place $v$ of $\F_{\ell}(t)$ that splits totally in $\F'k$. Moreover, by the proof of \cite[Theorem A.2]{Liang2024a}, we can take a cyclic field extension $K_{0}/\F_{\ell}(t)$ of degree $\ell_{0}$ which is ramified only at $v$. Then the fields $K_{0}$ and $\F'k$ are linearly disjoint over $\F_{\ell}(t)$. Hence, if we set $\widetilde{K}:=K_{0}\F'k$, then we obtain an isomorphism $\Gal(\widetilde{K}/k)\cong G$. Moreover, the assumption on $v$ implies that all decomposition groups of $\widetilde{K}/k$ are cyclic. This completes the proof in this case. 
\end{proof}

\section{Transitive groups}\label{sect:trgp}

Let $n$ be a positive integer. A \emph{transitive group of degree $n$} (or, a \emph{transitive subgroup of $\fS_{n}$}) is a subgroup $G$ of the symmetric group $\fS_{n}$ of $n$-letters such that the canonical action of $G$ on $\{1,\ldots,n\}$ is transitive. Moreover, we say a subgroup of $G$ of the form $\Stab_{G}(i)$, where $i\in \{1,\ldots,n\}$, a \emph{corresponding subgroup} of $G$. 

It is convenient to give a relation between transitive groups and pairs $(G,H)$ of finite groups $G$ and their subgroups $H$. For a finite group $G$ and its subgroup $H$, let
\begin{equation*}
N^{G}(H):=\bigcap_{g\in G}gHg^{-1}. 
\end{equation*}
By definition, $N^{G}(H)$ is the maximum normal subgroup of $G$ that is contained in $H$. 

\begin{prop}\label{prop:tggp}
Let $n$ be a positive integer. 
\begin{enumerate}
\item Let $G$ be a transitive group of degree $n$, and $H$ its subgroup. If $H$ is a corresponding subgroup of $G$, then we have $(G:H)=n$ and $N^{G}(H)=\{1\}$. 
\item Let $G$ be a finite group, and $H$ a subgroup of index $n$ in $G$. Then the action of $G$ on $G/H$, which is induced by left multiplication, gives an injection $G/N^{G}(H)\hookrightarrow \fS_{n}$. In particular, $G/N^{G}(H)$ is a transitive group of degree $n$. 
\end{enumerate}
\end{prop}

\begin{proof}
(i): We may assume $H=\Stab_{G}(1)$. Consider the action of $G$ on $\{1,\ldots,n\}$. For each $i\in \{2,\ldots,n\}$, take an element $g_{i}$ of $G$ so that $g_{i}(1)=i$. Then there is a bijection $G/H\cong \fS_{n}/\fS_{n-1}$. Hence we have the following as desired: 
\begin{equation*}
N^{G}(H)=\bigcap_{i=1}^{n}g_{i}Hg_{i}^{-1}=\{1\}. 
\end{equation*}

(ii): By definition, we obtain a homomorphism $\alpha \colon G\rightarrow \fS_{n}$. It suffices to prove $\Ker(\alpha)=N^{G}(H)$. By definition, $\Ker(\alpha)$ is contained in the stabilizer of the point corresponding to $H$. Hence we obtain $\Ker(\alpha)\subset N^{G}(H)$. Take a complete representative $S$ of $G/H$ in $G$, then we have $N^{G}(H)\subset \bigcap_{s\in S}sHs^{-1}\subset \Ker(\alpha)$. Consequently, we obtain the desired equality. 
\end{proof}

\begin{rem}
The converse of Proposition \ref{prop:tggp} (i) does not hold in general. Let $G$ be the transitive group $15T13$ of degree $15$, which is isomorphic to $(C_{5})^{2}\rtimes_{\varphi} C_{3}$ with $\varphi$ faithful. Then, a corresponding subgroup of $G$ is isomorphic to $D_{5}$. However, there exists a cyclic subgroup $H'$ of index $15$, and it satisfies $N^{G}(H')=\{1\}$. Then, $H'$ is not a corresponding subgroup of $G$ since $C_{10}\not\cong D_{5}$. Note that the pair $G$ and $H'$ give the transitive group $15T14$. 
\end{rem}

We frequently use the notion of transitive groups because of the following. 

\begin{lem}\label{lem:extt}
Let $F$ be a field, and $E/F$ a finite separable field extension with Galois closure $\widetilde{E}/F$. Put $G:=\Gal(\widetilde{E}/F)$ and $H:=\Gal(\widetilde{E}/E)$. Then $G$ is a transitive group of degree $[E:F]$, and $H$ is a corresponding subgroup of $G$.
\end{lem}

\begin{proof}
By Galois theory, we have $gHg^{-1}=\Gal(\widetilde{E}/g(E))$ for any $g\in G$. This implies $N^{G}(H)=\{1\}$, and hence the assertion follows from Proposition \ref{prop:tggp}. 
\end{proof}

Finally, we give some facts that will be used in the proof of Proposition \ref{prop:bart}. 

\begin{lem}\label{lem:trdp}
Let $G$ be a transitive group of degree a prime number $p$. Then we have $\#G\in p\Z \setminus p^{2}\Z$. In particular, a $p$-Sylow subgroup of $G$ is cyclic. 
\end{lem}

\begin{proof}
The assertion follows from $\#\fS_{p}=p!\in p\Z \setminus p^{2}\Z$. 
\end{proof}

\section{Multinorm one tori and the multinorm principle}\label{sect:mnot}

\subsection{Multinorm one tori}

Let $F$ be a field, and fix its separable closure $F^{\sep}$. Consider a finite {\'e}tale algebra $\bE$ over $F$, that is, a finite product of finite seprable field extension of $F$ that are contained in $F^{\sep}$. Then, the \emph{multinorm one torus} associated to $\bE/F$ is defined as follows: 
\begin{equation*}
T_{\bE/F}:=\{t\in \Res_{\bE/F}\G_{m}\mid \N_{\bE/F}(t)=1\}. 
\end{equation*}
If $\bE$ is a field, we call it the \emph{norm one torus}. 

\vspace{5pt}
Following \cite[\S 3]{Hasegawa2025}, we give a group-theoretic description of $X^{*}(T_{\bE/F})$. Let $G$ be a finite group, and $\cH$ a finite multiset of subgroups in $G$. Here, a finite multiset refers to a pair $(Z,m)$, where $Z$ is a finite set, and $m$ is a map from $Z$ into $\Zpn$. We define a $G$-lattice $J_{G/\cH}$ by the exact sequence
\begin{equation*}
0\rightarrow \Z \xrightarrow{(\varepsilon_{G/H}^{\circ})_{H\in \cH}}\Ind_{H}^{G}\Z\rightarrow J_{G/H}\rightarrow 0.
\end{equation*}
Here, $\varepsilon_{G/H}^{\circ}$ is as in Proposition \ref{prop:htch}, which is defined as $\varepsilon_{G/H}^{\circ}(1)(g)=1$ for any $g\in G$. 
If $\cH=\{H\}$, we simply write $J_{G/\cH}$ for $J_{G/H}$, and we call it the \emph{Chevalley module}. Moreover, we denote $J_{G/H}$ by $J_{G}$ if $H=\{1\}$, . 

\begin{prop}\label{prop:mnjg}
Let $\E=E_{1}\times \cdots \times E_{r}$, where $E_{1},\ldots,E_{r}$ are finite separable field extensions of $F$. Take a finite Galois extension $\widetilde{E}$ of $F$ that contains $E_{i}$ for all $i\in \{1,\ldots,r\}$. Set 
\begin{equation*}
G:=\Gal(\widetilde{E}/F),\quad \cH:=\{\Gal(\widetilde{E}/E_{i})<G\mid i\in \{1,\ldots,r\}\}. 
\end{equation*}
Then, there is an isomorphism of $G$-lattices
\begin{equation*}
X^{*}(T_{\E/F}) \cong J_{G/\cH}. 
\end{equation*}
\end{prop}

\begin{proof}
Consider the canonical exact sequence
\begin{equation*}
0\rightarrow T_{\bE/F}\rightarrow \Res_{\bE/F}\G_{m,\bE}\xrightarrow{\N_{\E/F}} \G_{m}\rightarrow 1. 
\end{equation*}
Taking character groups, we obtain an exact sequence of $G$-lattices
\begin{equation*}
0\rightarrow \Z \xrightarrow{(\varepsilon_{G/H}^{\circ})_{H\in \cH}} \bigoplus_{H\in \cH}\Ind_{H}^{G}\Z \rightarrow X^{*}(T_{\bE/F}) \rightarrow 0. 
\end{equation*}
Hence, we obtain the desired assertion. 
\end{proof}

In the following, we recall fundamental properties in \cite[\S\S 3.2--3.3]{Hasegawa2025}. 

\begin{dfn}[{\cite[Definition 3.12]{Hasegawa2025}}]
Let $G$ be a finite group, and $\cH$ a finite \emph{set} of its subgroups. We say that $\cH$ is \emph{reduced} if $H\not\subset H'$ for any $H,H\in \cH$ with $H\neq H'$. 
\end{dfn}

Let $G$ a finite group, and $\cH$ a finite multiset of its subgroups. We use the notations as follows: 
\begin{itemize}
\item Denote by $\cH^{\set}$ the underlying set of $\cH$. 
\item For each $\cH^{\set}$, $m_{\cH}(H)$ denotes the multiplicity of $H$ in $\cH$. 
\item Write $\cH^{\red}$ for the subset of $\cH^{\set}$ consisting of all elements that are maximal with respect to inclusion, that is,
\begin{equation*}
\cH^{\red}:=\{H\in \cH^{\set}\mid H\not\subset H'\text{ for any }H'\in \cH^{\set}\setminus \{H\}\}. 
\end{equation*}
\end{itemize}

\begin{prop}[{\cite[Corollary 3.13 (ii)]{Hasegawa2025}}]\label{prop:endo}
Let $G$ be a finite group, and $\cH$ a finite multiset of its subgroups. Then, there is an isomorphism of $G$-lattices
\begin{equation*}
J_{G/\cH} \cong J_{G/\cH^{\red}}\oplus \left(\bigoplus_{H\in \cH^{\red}}\Ind_{H}^{G}\Z^{m_{\cH}(H)-1}\right) \oplus \left(\bigoplus_{H\in \cH^{\set}\setminus \cH^{\red}}\Ind_{H}^{G}\Z^{m_{\cH}(H)}\right). 
\end{equation*}
\end{prop}

\begin{prop}[{\cite[Proposition 3.20]{Hasegawa2025}}]\label{prop:hkor}
Let $G$ be a finite group, and $\cH$ a finite multiset of its subgroups. For each subgroup $D$ of $G$, there is an isomorphism of $D$-lattices
\begin{equation*}
J_{G/\cH}\cong J_{D/\cH_{D}}. 
\end{equation*}
Here, $\cH_{D}$ is the finite multiset of subgroups in $D$ consisting of $D\cap gHg^{-1}$ with $H\in \cH$ and $g\in R(D,H)$ (the definition of $R(D,H)$ is in Proposition \ref{prop:mcky}). 
\end{prop}

\begin{proof}
This follows from Proposition \ref{prop:mcky} for $M=\Z$. 
\end{proof}

\subsection{The multinorm principle}

Here, let $k$ be a global field. For a finite {\'e}tale algebra $\bK$ over $k$, put
\begin{equation*}
\Sha(\bK/k):=(k^{\times}\cap \N_{\bK/k}(\A_{\bK}^{\times}))/\N_{\bK/k}(\bK^{\times}),
\end{equation*}
where $\A_{\bK}^{\times}$ is the product of the id{\`e}le groups of all factors of $\bK$. We say that the \emph{multinorm principle holds for $\bK/k$} if
\begin{equation}\label{eq:mlpr}
\Sha(\bK/k)=1. 
\end{equation}
If $\bK$ is a field, we call \eqref{eq:mlpr} being valid the \emph{Hasse norm principle holds for $\bK/k$}. 

\begin{prop}[{\cite[p.~8, (2.5)]{Liang2024b}}]\label{prop:loyt}
Let $k$ be a global field, and $\bK$ be a finite {\'e}tale algebra over $k$. Then, there is an isomorphism
\begin{equation*}
\Sha(\bK/k)\cong \Sha^{1}(k,T_{\bK/k}). 
\end{equation*}
\end{prop}

If $\bK$ is a field, then Proposition \ref{prop:loyt} is given by Ono (\cite[p.~70]{Ono1963}). 

\begin{cor}\label{cor:onhn}
Let $K/k$ be a finite separable field extension with Galois closure $\widetilde{K}/k$. Put $G:=\Gal(\widetilde{K}/k)$ and $H:=\Gal(\widetilde{K}/K)$. We denote by $\cD$ the set of decomposition groups of $\widetilde{K}/k$. Then there is an isomorphism
\begin{equation*}
\Sha(K/k)\cong \Sha_{\cD}^{2}(G,J_{G/\cH})^{\vee}. 
\end{equation*}
\end{cor}

\begin{proof}
By Proposition \ref{prop:loyt}, it suffices to prove
\begin{equation*}
\Sha^{1}(k,T_{K/k})\cong \Sha_{\cD}^{2}(G,J_{G/\cH})^{\vee}. 
\end{equation*}
This follows from Proposition \ref{prop:almt}, since Proposition \ref{prop:mnjg} gives an isomorphism $X^{*}(T_{K/k})\cong J_{G/\cH}$. 
\end{proof}

\subsection{Cohomology groups: general theory}

\begin{lem}\label{lem:mnex}
Let $G$ be a finite group, and $\cH$ a finite multiset of its subgroups. Then, there is an exact sequence
\begin{align*}
0\rightarrow H^{1}(G,J_{G/\cH})\rightarrow H^{1}(G,\Z) &\xrightarrow{(\varepsilon_{G/\cH,2}^{\circ*})_{H\in \cH}}\bigoplus_{H\in \cH}H^{2}(G,\Ind_{H}^{G}\Z) \rightarrow H^{2}(G,J_{G/\cH})\\
\rightarrow H^{3}(G,\Z)&\xrightarrow{(\Res_{G/H})_{H\in \cH}}\bigoplus_{H\in \cH}H^{3}(H,\Z). 
\end{align*}
\end{lem}

\begin{proof}
The assertion follows from taking cohomology groups for the canonical exact sequence
\begin{equation*}
0\rightarrow \Z \rightarrow \bigoplus_{H\in \cH}\Ind_{H}^{G}\Z \rightarrow J_{G/\cH} \rightarrow 0. 
\end{equation*}
Note that the rightmost homomorphisms follows from Lemma \ref{lem:shpr}. 
\end{proof}

\begin{lem}\label{lem:stts}
Let $G$ be a finite group, and $\cH$ a finite multiset of its subgroups. Then, for any admissible set of subgroups in $G$, there is an isomorphism
\begin{equation*}
\Sha_{\cD}^{2}(G,J_{G/\cH})\cong \Sha_{\cD}^{2}(G,J_{G/\cH^{\red}}). 
\end{equation*}
\end{lem}

\begin{proof}
By Proposition \ref{prop:endo}, there is an isomorphism between $\Sha_{\cD}^{2}(G,J_{G/\cH})$ and the direct sum
\begin{equation*}
\Sha_{\cD}^{2}(G,J_{G/\cH^{\red}})\oplus \bigoplus_{H\in \cH^{\red}}\Sha_{\cD}^{2}(G,\Ind_{H}^{G}\Z)^{\oplus m_{\cH}(H)-1}\oplus \bigoplus_{H\in \cH^{\set}\setminus \cH^{\red}}\Sha_{\cD}^{2}(G,\Ind_{H}^{G}\Z)^{\oplus m_{\cH}(H)}. 
\end{equation*}
Hence, the assertion follows from Lemma \ref{lem:qtts}. 
\end{proof}

\begin{prop}\label{prop:andg}
Let $G$ be a finite group, and $\cH$ a finite multiset of its subgroups. Then, the abelian group $\Sha_{\omega}^{2}(G,J_{G/\cH})$ is annihilated by the great common divisor of $(G:H)$ for all $H\in \cH$. 
\end{prop}

\begin{proof}
By Lemma \ref{lem:stts}, we may assume $\cH=\cH^{\red}$. Pick $H\in \cH$. By Proposition \ref{prop:hkor}, there is an isomorphism of $H$-lattices
\begin{equation*}
J_{G/\cH}\cong J_{H/\cH_{H}}. 
\end{equation*}
Then, the multiset $\cH_{H}$ consists of $H\cap gH'g^{-1}$ with $H\in \cH$ and $g\in R(H,H')$. This implies that $H$ is the unique maximal element of $\cH_{H}^{\set}$ with respect to inclusion. Hence, Lemma \ref{lem:stts} gives an isomorphism
\begin{equation*}
\Sha_{\omega}^{2}(H,J_{G/\cH})\cong \Sha_{\omega}^{2}(H,J_{H/H})=0. 
\end{equation*}
Combining this with Proposition \ref{prop:tsan} (i), we obtain that $\Sha_{\omega}^{2}(G,J_{G/\cH})$ is annihilated by $(G:H)$. This implies the desired assertion. 
\end{proof}

\begin{lem}\label{lem:cyij}
Let $G$ be a finite cyclic group, and $\cH$ a finite multiset of its subgroups. We further assume that there is a subgroup $N$ of $G$ such that $H<N$ for all $H\in \cH$. Then, for any subgroup $D$ of $G$ containing $N$, the restriction map
\begin{equation*}
H^{2}(G,J_{G/\cH})\rightarrow H^{2}(D,J_{G/\cH})
\end{equation*}
is injective. 
\end{lem}

\begin{proof}
We may assume $D=N$. Consider a commutative diagram
\begin{equation*}
\xymatrix{
H^{2}(G,\Z)\ar[r]^{\varepsilon_{G/H*}\hspace{30pt}}\ar[d]&\bigoplus_{H\in \cH}H^{2}(G,\Ind_{H}^{G}\Z)\ar[r]\ar[d]& H^{2}(G,J_{G/\cH})\ar[r]\ar[d]& H^{3}(G,\Z)\ar[d]\\
H^{2}(N,\Z)\ar[r]^{\varepsilon_{N}^{*}\hspace{30pt}}&\bigoplus_{H\in \cH}H^{2}(N,\Ind_{H}^{G}\Z)\ar[r]& H^{2}(N,J_{G/\cH})\ar[r]& H^{3}(N,\Z). 
}
\end{equation*}
We have $H^{3}(G,\Z)=0$ and $H^{3}(N,\Z)=0$. Hence, it suffices to prove the equality
\begin{equation*}
\Res_{G/N}^{-1}(\Ima(\varepsilon_{N}))=\Ima(\varepsilon). 
\end{equation*}
It is clear that $\Res_{G/N}^{-1}(\Ima(\varepsilon_{N}))$ contains $\Ima(\varepsilon)$. In the following, we give a proof of the reverse inclusion. Fix $H\in \cH$. Since $H$ is contained in $N$, Proposition \ref{prop:rsch} gives a commutative diagram
\begin{equation*}
\xymatrix@C=35pt{
H^{2}(G,\Ind_{H}^{G}\Z)\ar[r]^{\Res_{G/N}}\ar[d]^{\cong}&H^{2}(N,\Ind_{H}^{G}\Z)\ar[d]^{\cong}\\
H^{\vee}\ar[r]^{\diag}&\bigoplus_{G/N}H^{\vee}. 
}
\end{equation*}
In particular, the map $\Res_{G/N}\colon \bigoplus_{H\in \cH}H^{2}(G,\Ind_{H}^{G}\Z)\rightarrow \bigoplus_{H\in \cH}H^{2}(N,\Ind_{H}^{G}\Z)$ is injective. Moreover, since $G$ is abelian, Proposition \ref{prop:htch} implies that the map $\Res_{G/N}\colon H^{2}(G,\Z)\rightarrow H^{2}(D,\Z)$ is surjective. Therefore, $\Res_{G/N}^{-1}(\Ima(\varepsilon_{N}))$ is contained in $\Ima(\varepsilon)$. This completes the proof. 
\end{proof}

\begin{thm}\label{thm:upcj}
Let $G$ be a finite group, $\cH$ a finite multiset of its subgroups, and $\cD$ an admissible set of subgroups in $G$. Then, we have
\begin{equation*}
\Sha_{\cD}^{2}(G,J_{G/\cH})=\Sha_{\cD_{(\cH)}}^{2}(G,J_{G/\cH}),
\end{equation*}
where $\cD_{(\cH)}$ is defined as follows: 
\begin{equation*}
\cD_{(\cH)}:=\cD \setminus \{D \in \cC_{G} \mid D\not\subset H\text{ for any }H\in \cH\}. 
\end{equation*}
\end{thm}

\begin{proof}
Pick $D\in \cC_{G}$ that satisfies $D\not\subset H$ for any $H\in \cH$. It suffices to prove 
\begin{equation*}
\Ker(\Res_{G/D})=\bigcap_{i=1}^{m}\Ker(\Res_{G/D^{(i)}})
\end{equation*}
for some $D^{(1)},\ldots,D^{(m)}\in \bigcup_{H\in \cH}\cC_{H}$. By Proposition \ref{prop:tsan} (ii), the homomorphism
\begin{equation*}
(\Res_{D/D_{p}})_{p\mid \#D}\colon H^{2}(D,J_{G/H})\rightarrow \bigoplus_{p\mid \#D}H^{2}(D_{p},J_{G/H})
\end{equation*}
is injective. This implies an equality
\begin{equation*}
\Ker(\Res_{G/D})=\bigcap_{p\mid \#D}\Ker(\Res_{G/D_{p}}). 
\end{equation*}
Hence, we may assume that $D$ has prime power order. Pick $H_{0}\in \cH$ and $g_{0}\in G$ so that
\begin{equation*}
\#(g_{0}Dg_{0}^{-1}\cap H_{0})=\max\{\#(gDg^{-1}\cap H)\in \Zpn \mid H\in \cH,g\in G\}. 
\end{equation*}
Put $D':=g_{0}Dg_{0}^{-1}$. Then, Lemma \ref{lem:cyij} implies an equality 
\begin{equation*}
\Ker(\Res_{G/D})=\Ker(\Res_{G/D'}). 
\end{equation*}
On the other hand, Proposition \ref{prop:hkor} gives an isomorphism of $D'$-lattices
\begin{equation*}
J_{G/H}\cong J_{G/\cH_{D'}}. 
\end{equation*}
Now, $\cH_{D'}$ consists of subgroups of $D'$ of the form $D'\cap gHg^{-1}$ with $g\in R(D',H)\subset G$. In particular, all elements of $\cH_{D'}$ are contained in $D'\cap H$, since $D$ has prime power order. Hence, by Lemma \ref{lem:cyij}, the map $\Res_{D'/D'\cap H}$ is injective. This gives an equality
\begin{equation*}
\Ker(\Res_{G/D'})=\Ker(\Res_{G/D'\cap H}), 
\end{equation*}
and hence the proof is complete. 
\end{proof}

\begin{thm}\label{thm:ppup}
Let $G$ be a finite group, $H$ its subgroup, and $\cD$ an admissible set of subgroups in $G$. Then, for any prime divisor $p$ of $(G:H)$, we have
\begin{equation*}
\Sha_{\cD}^{2}(G,J_{G/H})[p^{\infty}]=\Sha_{\cD_{(H,p)}}^{2}(G,J_{G/H})[p^{\infty}],
\end{equation*}
where $\cD_{(H,p)}$ is the set of $p$-Sylow subgroups of all elements of $\cD_{(H)}:=\cC_{H}\cup (\cD\setminus \cC_{G})$. In particular, one has an equality
\begin{equation*}
\Sha_{\omega}^{2}(G,J_{G/H})[p^{\infty}]=\Sha_{\cC_{H,p}}^{2}(G,J_{G/H})[p^{\infty}],
\end{equation*}
where $\cC_{H,p}$ is the set of all cyclic groups of order powers of $p$ in $H$. 
\end{thm}

\begin{proof}
Take $D\in \cD_{(H)}$, and choose a $p$-Sylow subgroup $D_{p}$ of $G$. Then, the composite 
\begin{equation*}
H^{2}(D,J_{G/H})[p^{\infty}]\hookrightarrow H^{2}(D,J_{G/H}) \xrightarrow{\Res_{D/D_{p}}}H^{2}(D_{p},J_{G/H})
\end{equation*}
is injective, which follows from Proposition \ref{prop:tsan} (ii). This implies
\begin{equation*}
\Ker(\Res_{G/D})[p^{\infty}]=\Ker(\Res_{G/D_{p}})[p^{\infty}]. 
\end{equation*}
Hence, we obtain an equality
\begin{equation*}
\Sha_{\cD_{(H)}}^{2}(G,J_{G/H})[p^{\infty}]=\Sha_{\cD_{(H,p)}}^{2}(G,J_{G/H})[p^{\infty}]. 
\end{equation*}
Therefore, the assertion follows from Theorem \ref{thm:upcj}. 
\end{proof}

\subsection{Cohomology groups: some particular cases}

If $G$ and $\cH$ are of particular type, the structure of $\Sha_{\omega}^{2}(G,J_{G/\cH})$ is determined. 

\begin{thm}[{cf.~\cite[Theorem 8.3]{BayerFluckiger2019}}]\label{thm:blpl}
Let $G$ be a finite group, and $H_1,\ldots,H_r$ normal subgroups of $G$ of index a prime number $p$ that satisfy $H_{i}\not\subset H_{j}$ for distinct $i,j\in \{1,\ldots,r\}$. Put $N:=\bigcap_{i=1}^{r}H_{i}$ and write $(G:N)=p^{m}$ with $m\in \Zpn$. Take $e_1,\ldots,e_r\in \Zpn$. Then there is an isomorphism
\begin{equation*}
\Sha_{\omega}^{2}(G,J_{G/\cH})\cong 
\begin{cases}
(\Z/p)^{\oplus r-2}&\text{if $m=2$ and $\#\cH^{\red}\geq 3$},\\
0&\text{otherwise}. 
\end{cases}
\end{equation*}
\end{thm}

\begin{proof}
We may assume $N=\{1\}$, which is a consequence of Proposition \ref{prop:ifts}. Moreover, by virtue of Lemma \ref{lem:stts}, we may assume $\cH=\cH^{\red}$. Take a finite abelian extension $K/\Q$ with Galois group $(C_{p})^{m}$ in which all decomposition groups are cyclic. This is possible by Proposition \ref{prop:shf1}. Let $K_{i}/k$ be the subextension of $K/k$ corresponding to $H_{i}$, and let $\bK:=K_{1}\times \cdots \times K_{r}$ and $\cH:=\{H_{1},\ldots,H_{r}\}$. Then, by Proposition \ref{prop:almt}, one has an isomorphism
\begin{equation*}
\Sha^{2}(k,X^{*}(T_{K/k}))\cong \Sha_{\omega}^{2}(G,J_{G/\cH}). 
\end{equation*}
Furthermore, \cite[Lemma 3.1, Proposition 5.16]{BayerFluckiger2019} gives an isomorphism 
\begin{equation*}
\Sha^{2}(k,X^{*}(T_{K/k}))\cong \Sha(\bK/k). 
\end{equation*}
Hence, the assertion follows from \cite[Theorem 8.3]{BayerFluckiger2019}. 
\end{proof}

\begin{rem}
One can prove Theorem \ref{thm:blpl} for $m=1$ by using Corollary \ref{cor:cyan}. Moreover, Theorem \ref{thm:blpl} in the case $m=r=2$ is known by H{\"u}rlimann (\cite[Proposition 3.3]{Hurlimann1984}). 
\end{rem}

\begin{prop}[{cf.~\cite[Lemma 4]{Bartels1981a}}]\label{prop:bart}
Let $G$ be a finite group, and $H$ a subgroup of prime index in $G$. Then
\begin{equation*}
\Sha_{\omega}^{2}(G,J_{G/H})=0. 
\end{equation*}
\end{prop}

\begin{proof}
Put $p:=(G:H)$, which is a prime number. By Proposition \ref{prop:ifts}, we may assume $N^{G}(H)=\{1\}$. Then, Proposition \ref{prop:tggp} (ii) implies that $G$ is a transitive group of degree $p$. In particular, a $p$-Sylow subgroup $S_{p}$ of $G$ is isomorphic to $C_{p}$ by Lemma \ref{lem:trdp}. This concludes $\Sha_{\omega}^{2}(S_{p},J_{G/H})=0$ by Corollary \ref{cor:cyan}. Combining this equality with Proposition \ref{prop:tsan} (i), we obtain that $(G:S_{p})$ annihilates $\Sha_{\omega}^{2}(G,J_{G/H})$. On the other hand, Proposition \ref{prop:andg} implies that $\Sha_{\omega}^{2}(G,J_{G/H})$ is a $p$-group, and hence the assertion holds. 
\end{proof}

\begin{prop}\label{prop:abts}
Put $G:=C_{n_{1}} \times C_{n_{2}}$, where $n_{1},n_{2}\in \Zpn$. Then there is an isomorphism
\begin{equation*}
\Sha_{\omega}^{2}(G,J_{G})\cong \Z/\gcd(n_{1},n_{2}). 
\end{equation*}
\end{prop}

\begin{proof}
By Therem \ref{thm:upcj}, we have an equality
\begin{equation*}
\Sha_{\omega}^{2}(G,J_{G/H})=H^{2}(G,J_{G}). 
\end{equation*}
On the other hand, Lemma \ref{lem:mnex} gives an isomorphism
\begin{equation*}
H^{2}(G,J_{G})\cong H^{3}(G,\Z). 
\end{equation*}
Hence, the assertion follows from \cite[Lemma 6.4, Lemma 6.5]{Frei2018}. 
\end{proof}

\section{Computation of second cohomology groups}\label{sect:thts}

Let $G$ be a group, and $H$ a subgroup of $G$. Then, we consider subgroups of $G$ as follows: 
\begin{equation*}
N_{G}(H):=\{g\in G\mid gHg^{-1}=H\},\quad Z_{G}(H):=\{g\in G\mid gh=hg\text{ for any }h\in H\}. 
\end{equation*}
Moreover, for another subgroup $D$ of $G$, let 
\begin{equation*}
[H,D]:=\langle hdh^{-1}d^{-1}\in G \mid h\in H,d\in D\rangle. 
\end{equation*}
In particular, if $H=D$, then we write $[H,H]$ for $H^{\der}$. 

\subsection{Exact sequences}\label{ssec:fcex}

\begin{lem}\label{lem:jgcg}
Let $G$ be a finite group, and $H\subset H'$ subgroups of $G$. Then there is an exact sequence
\begin{equation*}
\xymatrix{
0\ar[r] & \Ind_{H'}^{G}\Z \ar[d] \ar[r]& \Ind_{H}^{G}\Z \ar[d] \ar[r]& \Ind_{H'}^{G}J_{H'/H} \ar@{=}[d] \ar[r]& 0\\
0\ar[r] & J_{G/H'} \ar[r]& J_{G/H} \ar[r]& \Ind_{H'}^{G}J_{H'/H} \ar[r]& 0, }
\end{equation*}
where the horizontal sequences are exact. 
\end{lem}

\begin{proof}
We define the upper horizontal sequence by taking $\Ind_{H'}^{G}$ to the canonical exact sequence
\begin{equation*}
0\rightarrow \Z \rightarrow \Ind_{H}^{H'}\Z \rightarrow J_{H'/H}\rightarrow 0. 
\end{equation*}
On the other hand, by definition, one has a commutative diagram of $G$-lattices
\begin{equation*}
\xymatrix{
0\ar[r] & \Z \ar@{=}[d] \ar[r]& \Ind_{H'}^{G}\Z \ar[d] \ar[r]& J_{G/H'} \ar[d] \ar[r]& 0\\
0\ar[r] & \Z \ar[r]& \Ind_{H}^{G}\Z \ar[r]& J_{G/H} \ar[r]& 0. }
\end{equation*}
This implies that the composite $\Z \rightarrow \Ind_{H}^{G}\Z \rightarrow \Ind_{H'}^{G}J_{H'/H}$ is zero, and hence the upper horizontal sequence induces a lower exact sequence. 
\end{proof}

\begin{rem}
The exact sequence in Lemma \ref{lem:jgcg} can be interpreted as a sequence of tori. Let $k$ be a field, $K/k$ a finite field extension, and $K'/k$ a subextension of $K/k$. Put $G:=\Gal(\widetilde{K}/k)$, $H:=\Gal(\widetilde{K}/K)$ and $H':=\Gal(\widetilde{K}/K')$. Then we obtain a commutative diagram of $k$-tori
\begin{equation*}
\xymatrix@C=30pt{
1\ar[r]& \Res_{K'/k_0}T_{K/K'} \ar@{=}[d] \ar[r]& T_{K/k} \ar[d] \ar[r]^{\N_{K/K'}}& T_{K'/k} \ar[d] \ar[r]& 1\\
1\ar[r]& \Res_{K'/k_0}T_{K/K'} \ar[r]& \Res_{K/k}\G_m \ar[r]^{\N_{K/K'}}& \Res_{K'/k}\G_m \ar[r]& 1. 
}
\end{equation*}
Here the horizontal sequences are exact. Then, the exact sequence of $G$-lattices induced by Proposition \ref{prop:mnjg} is the desired one. 
\end{rem}

For a prime number $p$, we denote by $\ord_{p}\colon \Q^{\times}\rightarrow \Z$ the homomorphism defined as follows for each prime number $\ell$: 
\begin{equation*}
\ord_{p}(\ell)=
\begin{cases}
1&\text{if }\ell=p;\\
0&\text{if }\ell \neq p. 
\end{cases}
\end{equation*}

\begin{lem}\label{lem:sbsd}
Let $p$ be a prime number, and $G$ a finite group of order a multiple of $p$ of which a $p$-Sylow subgroup $S_{p}$ is normal in $G$. 
\begin{enumerate}
\item There exist a subgroup $G'$ of $G$ and an isomorphism $G\cong S_{p}\rtimes G'$. 
\item Let $D$ be a subgroup of $G$. Then the subgroup $S_{p}\cap gDg^{-1}$ of $S_{p}$ has index $p^{\ord_{p}(G:D)}$ for any $g\in G$. 
\item We further assume that $S_{p}$ is abelian. Let $D$ be a subgroup of $G$. Then there exist a subgroup $D'$ of $G'$ and $s\in S_{p}$ such that the isomorphism in \emph{(i)} induces
\begin{equation*}
sDs^{-1}=(S_{p}\cap D)\rtimes D'. 
\end{equation*}
\item We further assume that there is a subgroup $H$ of $G$ with $N^{G}(H)=\{1\}$. 
\begin{itemize}
\item[(a)] Let $R(S_{p},H)$ be a complete representative of $S_{p}\backslash G/H$ in $G$. Then
\begin{equation*}
\bigcap_{g\in R(S_{p},H)}(S_{p}\cap gHg^{-1})=\bigcap_{g\in G}(S_{p}\cap gHg^{-1})=N^{G}(S_{p}\cap H)=\{1\}. 
\end{equation*}
\item[(b)] If $S_{p}$ is abelian, then the exponent of $S_{p}$ is equal to that of $S_{p}/(S_{p}\cap H)$. 
\end{itemize}
\item We further assume that there is a subgroup $H$ of $G$ such that $\ord_{p}(G:H)=1$ and $N^{G}(H)=\{1\}$. Then there is an isomorphism $S_{p}\cong (C_{p})^{m}$ for some $m\in \Zpn$. 
\end{enumerate}
\end{lem}

\begin{proof}
(i): This is contained in the Schur--Zassenhaus theorem. 

(ii): Since the normality of $S_{p}$ in $G$ implies $S_{p}\cap gDg^{-1}=g^{-1}(S_{p}\cap D)g$ for every $g\in G$, we may assume $g=1$. Put $r:=\ord_{p}(G:D)$, and take a $p$-Sylow subgroup $S'_{p}$ of $H$, which has order $\#S_{p}/p^{r}$. Using the assumption $S_{p}\triangleleft G$ and Sylow's theorem, we have $S'_{p}\subset S_{p}$. Hence $S'_{p}$ is contained in $S_{p}\cap D$. This inclusion is an equality by the assumption on $S'_{p}$, which implies the desired equality $(S_{p}:S_{p}\cap D)=p^{r}$. 

(iii): Let $D'$ be the subgroup of $G'$ corresponding to $S_{p}D/S_{p}$ under the canonical isomorphism $G'\cong G/S_{p}$. Then, for each $d \in D'$ there is $\widetilde{f}(d)\in S_{p}$ such that $\widetilde{f}(d)d \in D$. Moreover, if $\widetilde{f'}(d)\in S_{p}$ also satisfies $\widetilde{f'}(d)d \in D$, then one has
\begin{equation*}
\widetilde{f}(d)\widetilde{f'}(d)^{-1}=(\widetilde{f}(d)d)(\widetilde{f'}(d)d)^{-1}\in S_{p}\cap D. 
\end{equation*}
Hence we obtain a map
\begin{equation*}
f\colon D'\rightarrow S_{p}/(S_{p}\cap D);d \mapsto \widetilde{f}(d)\bmod S_{p}\cap D. 
\end{equation*}
By direct computation, it turns out that $f$ is a $1$-cocycle. On the other hand, we have
\begin{equation*}
H^{1}(D',S_{p}/(S_{p}\cap D))=0
\end{equation*}
since $\#D'$ is coprime to $p$. Hence, there is $s\in S_{p}$ such that $f(d)=s^{-1}d(s)\bmod S_{p}\cap D$ for all $d\in D'$. Then we obtain that the subgroup $sDs^{-1}$ is equal to $(S_{p}\cap D)\rtimes D'$. 

(iv): First, we give a proof of (a). By Lemma \ref{lem:doub} (ii) and the normality of $S_{p}$ in $G$, we have
\begin{equation*}
\bigcap_{g\in R(S_{p},H)}(S_{p}\cap gHg^{-1})=\bigcap_{g\in G}(S_{p}\cap gHg^{-1})=\bigcap_{g\in G}g(S_{p}\cap H)g^{-1}=N^{G}(S_{p}\cap H). 
\end{equation*}
Furthermore, since $N^{G}(S_{p}\cap H)$ is contained in $N^{G}(H)$, the triviality of $N^{G}(S_{p}\cap H)$ follows from $N^{G}(H)=\{1\}$. This completes the proof of (a). On the other hand, for a proof of (b), recall that the subgroup $S_{p}\cap H$ of $S_{p}$ has index $p^{\ord_{p}(G:H)}$, which is a consequence of (ii) for $D=H$. This implies that the quotient $S_{p}/N^{G}(S_{p}\cap H)$ is an abelian group whose exponent coincides with that of $S_{p}/(S_{p}\cap H)$. Hence (b) follows from (a). 

(v): This follows from the same argument as (iv) (b) since all subgroups of $S_{p}$ of index $p$ are normal, and hence they contain $S_{p}^{\der}$. 
\end{proof}

\begin{cor}\label{cor:cdaj}
Let $p$ be a prime number, and $G$ a finite group of order a multiple of $p$ of which $p$-Sylow subgroup $S_{p}$ is abelian and normal in $G$. Take a subgroup $G'$ of $G$ satisfying $G=S_{p}\rtimes G'$. Consider an admissible set of subgroups $\cD$ of $G$. For any $\overline{D}\in {\cD}_{/S_{p}}$, there exists a subgroup $N$ of $S_{p}$ such that $N\rtimes D'\in \cD$. Here $D'$ is the subgroup of $G'$ corresponding to $\overline{D}$ under the canonical isomorphism $G'\cong G/S_{p}$. 
\end{cor}

\begin{proof}
By the definition of $\cD_{/S_{p}}$, we can take $D\in \cD$ whose image in $G/S_{p}$ coincides with $\overline{D}$. Then it follows from Lemma \ref{lem:sbsd} (iii) that there is $s\in S_{p}$ so that $sDs^{-1}\cong (D\cap S_{p})\rtimes D'$. Hence the assertion holds. 
\end{proof}

Let $p$ be a prime number, $G$ a finite group, and $H$ a subgroup of $G$ satisfying $(G:H)\in p\Z$. We further assume that a $p$-Sylow subgroup $S_{p}$ of $G$ is normal. Then the exact sequence
\begin{equation*}
0\rightarrow J_{G/HS_{p}}\rightarrow J_{G/H}\rightarrow \Ind_{HS_{p}}^{G}J_{HS_{p}/H}\rightarrow 0,
\end{equation*}
which follows from Lemma \ref{lem:jgcg}, induces the diagram as follows for any subgroup $D$ of $G$: 
\begin{equation}\label{res1}
\xymatrix{
H^{1}(G,\Ind_{HS_{p}}^{G}J_{HS_{p}/H})\ar[r]^{\hspace{7mm}\delta_{G}} \ar[d]^{\Res_{G/D}}& H^{2}(G,J_{G/HS_{p}})\ar[d]^{\Res_{G/D}}\\
H^{1}(D,\Ind_{HS_{p}}^{G}J_{HS_{p}/H})\ar[r]^{\hspace{7mm}\delta_{D}} & H^{2}(D,J_{G/HS_{p}}). 
}
\end{equation}

\begin{dfn}[{cf.~\cite[(4.23)]{Liang2024a}}]\label{dfn:h2ly}
Let $p$, $G$, $H$ and $S_{p}$ be as above. For a set of subgroups $\cD$ of $G$, put
\begin{equation*}
H^{2}(G,J_{G/HS_{p}})^{(\cD)}:=\{f\in H^{2}(G,J_{G/HS_{p}})\mid \Res_{G/D}(f)\in \Ima(\delta_{D})\text{ for any }D\in \cD \}. 
\end{equation*}
\end{dfn}

\begin{prop}\label{prop:h2s2}
Let $p$ be a prime number, $G$ a finite group, and $H$ a subgroup of $G$ satisfying $(G:H)\in p\Z$. We further assume that a $p$-Sylow subgroup $S_{p}$ of $G$ is normal. Take an admissible set $\cD$ of subgroups of $G$. There is an exact sequence
\begin{equation*}
H^{1}(G,\Ind_{HS_{p}}^{G}J_{HS_{p}/H})\xrightarrow{\delta_{G}} H^{2}(G,J_{G/HS_{p}})^{(\cD)}\xrightarrow{\widehat{\N}_{G}} \Sha_{\cD}^{2}(G,J_{G/H}). 
\end{equation*}
Moreover, the homomorphism $\widehat{\N}_{G}$ is surjective if $(G:H)\notin p^{2}\Z$. 
\end{prop}

\begin{proof}
By the definition of $H^{2}(G,J_{G/HS_{p}})^{(\cD)}$, the exact sequence
\begin{equation*}
H^{1}(G,\Ind_{HS_{p}}^{G}J_{HS_{p}/H})\xrightarrow{\delta_{G}} H^{2}(G,J_{G/HS_{p}})\xrightarrow{\widehat{\N}_{G}} H^{2}(G,J_{G/H})\rightarrow H^{2}(G,\Ind_{HS_{p}}^{G}J_{HS_{p}/H})
\end{equation*}
induces an exact sequence
\begin{equation*}
H^{1}(G,\Ind_{HS_{p}}^{G}J_{HS_{p}/H})\xrightarrow{\delta_{G}} H^{2}(G,J_{G/HS_{p}})^{(\cD)}\xrightarrow{\widehat{\N}_{G}} \Sha_{\cD}^{2}(G,J_{G/H})\rightarrow \Sha_{\cD}^{2}(G,\Ind_{HS_{p}}^{G}J_{HS_{p}/H}), 
\end{equation*}
which implies the desired exactness. 

In the following, assume $(G:H)\notin p^{2}\Z$. It suffices to prove $\Sha_{\omega}^{2}(G,\Ind_{HS_{p}}^{G}J_{HS_{p}/H})=0$. By Corollary \ref{lem:ids2}, one has an isomorphism
\begin{equation*}
\Sha_{\omega}^{2}(G,\Ind_{HS_{p}}^{G}J_{HS_{p}/H}) \cong \Sha_{\omega}^{2}(HS_{p},J_{HS_{p}/H})
\end{equation*}
Since $(HS_{p}:S_{p})=p$ is a prime, Proposition \ref{prop:bart} implies $\Sha_{\omega}^{2}(HS_{p},J_{HS_{p}/H})=0$. This completes the proof. 
\end{proof}

\subsection{The prime-to-$p$ torsion parts}

\begin{lem}\label{lem:exh1}
Let $p$ be a prime number, $G$ a finite group, and $H$ a subgroup of $G$ satisfying $(G:H)\in p\Z$. We further assume that a $p$-Sylow subgroup $S_{p}$ of $G$ is normal. 
\begin{enumerate}
\item There is a commutative diagram
\begin{equation*}
\xymatrix@C=35pt{
0\ar[r]& H^{1}(G,J_{G/HS_{p}})\ar[d]^{\cong} \ar[r]& H^{1}(G,J_{G/H})\ar[d]^{\cong}\ar[r]& 
H^{1}(G,\Ind_{HS_{p}}^{G}J_{HS_{p}/H}) \ar[d]^{\cong} \\
0\ar[r]& \Ker(G^{\vee}\rightarrow (HS_{p})^{\vee}) \ar[r]& \Ker(G^{\vee}\rightarrow H^{\vee}) \ar[r]^{f\mapsto f\mid_{HS_{p}}\hspace{5mm}}& 
\Ker((HS_{p})^{\vee}\rightarrow H^{\vee}). }
\end{equation*}
\item The homomorphism $\delta_{G}\colon H^{1}(G,\Ind_{HS_{p}}^{G}J_{HS_{p}/H})\rightarrow H^{2}(G,J_{G/HS_{p}})$ is injective if $S_{p}$ is generated by $[S_{p},G]$ and $S_{p}\cap H$. 
\item The inclusion $S_{p}\subset G$ induces is an isomorphism
\begin{equation*}
\Ker((HS_{p})^{\vee}\rightarrow H^{\vee})\cong (S_{p}/[S_{p},HS_{p}]\cdot (S_{p}\cap H))^{\vee}. 
\end{equation*}
In particular, the abelian group $H^{1}(G,\Ind_{HS_{p}}^{G}J_{HS_{p}/H})$ is annihilated by the exponent of $S_{p}$. 
\end{enumerate}
\end{lem}

\begin{proof}
(i): By Lemma \ref{lem:jgcg}, we have an exact sequence of abelian groups
\begin{align*}
H^{0}(G,\Ind_{HS_{p}}^{G}J_{HS_{p}/H})&\rightarrow H^{1}(G,J_{G/HS_{p}})\\
&\rightarrow H^{1}(G,J_{G/H})\rightarrow H^{1}(G,\Ind_{HS_{p}}^{G}J_{HS_{p}/H}). 
\end{align*}
The assertion follows from Lemma \ref{lem:mnex}. 

In the following, we give proofs of (ii) and (iii). By Lemma \ref{lem:sbsd} (iii), we may assume
\begin{equation*}
H=(S_{p}\cap H)\rtimes H', 
\end{equation*}
where $H'$ is a subgroup of $G'$. Note that we have $HS_{p}=S_{p}\rtimes H'$. 

(ii): It suffices to prove the isomorphy of the homomorphism
\begin{equation*}
\Ker(G^{\vee}\rightarrow (HS_{p})^{\vee})\rightarrow \Ker(G^{\vee}\rightarrow H^{\vee}), 
\end{equation*}
which is a consequence of (i) and Proposition \ref{prop:h2s2}. By definition, we have isomorphisms
\begin{equation*}
\Ker(G^{\vee}\rightarrow (HS_{p})^{\vee})\cong \Ker(G/G^{\der}HS_{p})^{\vee},\quad \Ker(G^{\vee}\rightarrow H^{\vee})\cong \Ker(G/G^{\der}H)^{\vee}. 
\end{equation*}
Moreover, direct computation implies the following: 
\begin{equation*}
G^{\der}HS_{p}=S_{p}\rtimes ((G')^{\der}H'),\quad G^{\der}H=([S_{p},G]\cdot (S_{p}\cap H))\rtimes ((G')^{\der}H'). 
\end{equation*}
Hence, the assumption $S_{p}=[S_{p},G]\cdot (S_{p}\cap H)$ follows the desired isomorphy. 

(iii): Since the derived group of $HS_{p}$ equals $[S_{p},H]\rtimes (H')^{\der}$, one has an isomorphism
\begin{align*}
\Ker((HS_{p})^{\vee}\rightarrow H^{\vee})&\cong ((S_{p}\rtimes H')/(([S_{p},H]\rtimes (H')^{\der})\cdot ((S_{p}\cap H)\rtimes H')))^{\vee}\\
&\cong (S_{p}/([S_{p},HS_{p}]\cdot (S_{p}\cap H)))^{\vee}. 
\end{align*}
Hence the assertion holds. 
\end{proof}

\begin{lem}\label{lem:shpp}
Let $p$ be a prime number, $G$ a finite group, and $H$ a subgroup of $G$ satisfying $(G:H)\in p\Z$. We further assume that a $p$-Sylow subgroup $S_{p}$ of $G$ is abelian and normal. Take an admissible set $\cD$ of subgroups of $G$. 
\begin{enumerate}
\item The inflation map $H^{2}(G/S_{p},J_{G/HS_{p}})\hookrightarrow H^{2}(G,J_{G/HS_{p}})$ induces an injection
\begin{equation*}
\Sha_{\cD_{/S_{p}}}^{2}(G/S_{p},J_{G/HS_{p}})\hookrightarrow H^{2}(G,J_{G/HS_{p}})^{(\cD)}. 
\end{equation*}
\item The composite
\begin{equation*}
\Sha_{\cD_{/S_{p}}}^{2}(G/S_{p},J_{G/HS_{p}})\hookrightarrow H^{2}(G,J_{G/HS_{p}})^{(\cD)}\xrightarrow{\widehat{\N}_{G}}\Sha_{\cD}^{2}(G,J_{G/H})
\end{equation*}
is injective, and factors through $\Sha_{\cD}^{2}(G,J_{G/H})^{(p)}$. 
\end{enumerate}
\end{lem}

\begin{proof}
(i): By definition, the group $H^{2}(G,J_{G/HS_{p}})^{(\cD)}$ contains $\Sha_{\cD}^{2}(G,J_{G/HS_{p}})$. Hence the assertion follows from Lemma \ref{lem:qtts}. 

(ii): By Lemma \ref{lem:exh1} (iii), the abelian group $H^{1}(G,\Ind_{HS_{p}}^{G}J_{HS_{p}/H})$ is annihilated by a power of $p$. Hence so is the kernel of the homomorphism $\widehat{\N}_{G}\colon H^{2}(G,J_{G/HS_{p}})^{(\cD)}\rightarrow \Sha_{\cD}^{2}(G,J_{G/H})$ by Proposition \ref{prop:h2s2}. Since Proposition \ref{prop:andg} implies that $\Sha_{\cD}^{2}(G,J_{G/HS_{p}})$ is annihilated by $(G:HS_{p})\notin p\Z$, the assertion holds. 
\end{proof}

\begin{thm}\label{thm:upts}
Let $p$ be a prime number, $G$ a finite group, and $H$ a subgroup of $G$ satisfying $(G:H)\in p\Z$. We further assume that a $p$-Sylow subgroup $S_{p}$ of $G$ is abelian and normal. Take an admissible set of subgroups $\cD$ of $G$. Then there is an isomorphism
\begin{equation*}
\Sha_{\cD}^{2}(G,J_{G/H})^{(p)} \cong \Sha_{\cD_{/S_{p}}}^{2}(G/S_{p},J_{G/HS_{p}}). 
\end{equation*}
\end{thm}

\begin{proof}
By Lemma \ref{lem:shpp} (ii), we obtain an injection
\begin{equation}\label{eq:prm1}
\widehat{\N}_{G}\colon \Sha_{\cD_{/S_{p}}}^{2}(G/S_{p},J_{G/HS_{p}})\hookrightarrow \Sha_{\cD}^{2}(G,J_{G/H})^{(p)}. 
\end{equation}
In the following, we prove that this map is surjective. By Lemma \ref{lem:sbsd} (iii), we may assume that the isomorphism $G\cong S_{p}\rtimes G'$ induces an isomorphism $H\cong (S_{p}\cap H)\rtimes H'$ for some subgroup $H'$ of $G'$. Since $\#S_{p}$ is a power of $p$, Proposition \ref{prop:tsan} (iii) gives an injection
\begin{equation}\label{eq:prm2}
\Res_{G/G'}\colon \Sha_{\cD}^{2}(G,J_{G/H})^{(p)}\hookrightarrow \Sha_{\cD_{G'}}^{2}(G',J_{G/H}). 
\end{equation}
Take a complete representative of $G'\backslash G/H$ in $G$. Since $G/S_{p}\cong G'$, we may assume $1\in R(G',H)\subset S_{p}\rtimes \{1\}$. Consider an isomorphism of $G'$-lattices given by Proposition \ref{prop:hkor}: 
\begin{equation*}
J_{G/H} \cong J_{G'/\cH_{G'}}. 
\end{equation*}
Here, $\cH_{G'}$ consists of $G'\cap sHs^{-1}$ with $s\in R(G',H)$. Then, $H'$ is the maximal element of $\cH_{G'}^{\set}$ with respect to inclusion. Hence, Lemma \ref{lem:stts} gives an isomorphism
\begin{equation}\label{eq:prm3}
\Sha_{\cD_{G'}}^{2}(G',J_{G/H})\cong \Sha_{\cD_{G'}}^{2}(G',J_{G'/H'}). 
\end{equation}
Composing the injections \eqref{eq:prm1}, \eqref{eq:prm2} and \eqref{eq:prm3}, we get an injection
\begin{equation}\label{eq:prm4}
\Sha_{\cD_{/S_{p}}}^{2}(G/S_{p},J_{G/HS_{p}})\hookrightarrow \Sha_{\cD_{G'}}^{2}(G',J_{G'/H'}). 
\end{equation}
On the other hand, by Corollary \ref{cor:cdaj}, the canonical isomorphism $G'\cong G/S_{p}$ induces an isomorphism $J_{G/HS_{p}}\cong J_{G'/H'}$ and a bijection $\cD_{G'}\cong \cD_{/S_{p}}$. This implies that \eqref{eq:prm4} is an isomorphism, and therefore we obtain the desired assertion. 
\end{proof}

\subsection{The $p$-primary torsion parts}

Let $G$ be a transitive group of degree $\in p\Z \setminus p^{2}\Z$, and $H$ a corresponding subgroup of $G$. Then we have $N^{G}(H)=\{1\}$ by Proposition \ref{prop:tggp} (i). Moreover, if a $p$-Sylow subgroup of $G$ is normal, then it is isomorphic to $(C_{p})^{m}$ for some $m\in \Zpn$ by Lemma \ref{lem:sbsd} (v). We discuss the structure of $\Sha_{\cD}^{2}(G,J_{G/H})[p^{\infty}]$, where $\cD$ is an admissible set of subgroups of $G$, in two separate cases, $m\neq 2$ and $m=2$. 

\begin{thm}\label{thm:pptt}
Let $p$ be a prime number, $G$ a transitive group of degree $\in p\Z \setminus p^{2}\Z$, and $H$ a corresponding subgroup of $G$. We further assume that
\begin{enumerate}
\item a $p$-Sylow subgroup $S_{p}$ of $G$ is normal; and 
\item $S_{p}\cong (C_{p})^{m}$ for some $m\neq 2$. 
\end{enumerate}
Then we have $\Sha^{2}_{\cD}(G,J_{G/H})[p^{\infty}]=0$ for any admissible subset of subgroups $\cD$ of $G$. 
\end{thm}

\begin{proof}
It suffices to prove $\Sha_{\omega}^{2}(S_{p},J_{G/H})=0$. Take a complete representative $R(S_{p},H)$ of $S_{p}\backslash G/H$ in $G$. By Proposition \ref{prop:hkor}, one has an isomorphism of $S_{p}$-lattices 
\begin{equation*}
J_{G/H} \cong J_{S_{p}/\cH_{S_{p}}}. 
\end{equation*}
Here, $\cH_{S_{p}}$ consists of $S_{p}\cap gHg^{-1}$ with $g\in R(S_{p},H)$. Now, Lemma \ref{lem:sbsd} (ii) implies
\begin{equation*}
(S_{p}:S_{p}\cap gHg^{-1})=p
\end{equation*}
for any $g\in R(S_{p},H)$. Moreover, Lemma \ref{lem:sbsd} (iv) (a) gives an equality
\begin{equation*}
\bigcap_{g\in R(S_{p},H)}(S_{p}\cap gHg^{-1})=\{1\}. 
\end{equation*}
Consequently, the assertion follows from Theorem \ref{thm:blpl} since $m\neq 2$. 
\end{proof}

In the following, we consider the case $m\neq 2$, which requires a more careful analysis. 

\begin{lem}\label{lem:rdsp}
Let $p$ be a prime number, $G$ a finite group of order a multiple of $p$, and $H$ a subgroup of $G$ satisfying $(G:H)\in p\Z$. We further assume that a $p$-Sylow subgroup $S_{p}$ of $G$ is normal. 
Take an admissible set $\cD$ of subgroups of $G$ such that every $D\in \cD$ does not contain $S_{p}$. Then, $H^{2}(G,J_{G/HS_{p}})^{(\cD)}[p^{\infty}]$ coincides with the subgroup of $H^{2}(G,J_{G/HS_{p}})$ as follows: 
\begin{equation*}
M_{HS_{p}}:=\!\left\{f\in H^{2}(G,J_{G/HS_{p}})[p^{\infty}] \,\middle|\, \Res_{G/D}(f)\in \Ima(\delta_{D})\text{ for any }D\in \cD_{S_{p},H}\right\}\!.
\end{equation*}
Here $\cD_{S_{p},H}:=\cC_{S_{p}\cap H}\cup \{D\cap S_{p}<S_{p}\mid D\in \cD\setminus \cC_{G}\}$. 
\end{lem}

\begin{proof}
Recall Proposition \ref{prop:h2s2} that one has an exact sequence
\begin{equation*}
H^{1}(G,\Ind_{HS_{p}}^{G}J_{HS_{p}/H})\rightarrow H^{2}(G,J_{G/HS_{p}})^{(\cD)}[p^{\infty}]\rightarrow \Sha_{\cD}^{2}(G,J_{G/H})[p^{\infty}] \rightarrow 0
\end{equation*}
Moreover, Theorem \ref{thm:ppup} gives an equality
\begin{equation*}
\Sha_{\cD}^{2}(G,J_{G/H})[p^{\infty}]=\Sha_{\cD_{(H,p)}}^{2}(G,J_{G/H})[p^{\infty}]. 
\end{equation*}
Then, one has $\cD_{(H,p)}=\cD_{H,S_{p}}$ since $S_{p}$ is normal in $G$. Therefore, we obtain the desired assertion. 
\end{proof}

We prepare some lemmas to investigate $M_{H,S_{p}}^{(\cD)}$ in Lemma \ref{lem:rdsp}. 

\begin{lem}\label{lem:csds}
Let $p$ be a prime number, $G$ a finite group of order a multiple of $p$, and $H$ a subgroup of $G$ satisfying $(G:H)\in p\Z$. We further assume that a $p$-Sylow subgroup $S_{p}$ of $G$ is normal. 
Then, there is a commutative diagram
\begin{equation*}
\xymatrix{
H^{2}(G,\Ind_{HS_{p}}^{G}J_{HS_{p}/H}) \ar[d]^{\cong} \ar[r]& H^{2}(G,\Ind_{HS_{p}}^{G}\Z)[p^{\infty}] \ar[d]^{\cong}\\
(S_{p}/[S_{p},HS_{p}]\cdot (S_{p}\cap H))^{\vee} \ar[r]& (S_{p}/[S_{p},HS_{p}])^{\vee}. }
\end{equation*}
\end{lem}

\begin{proof}
The left vertical isomorphism is a consequence of Lemma \ref{lem:exh1} (iii). The right isomorphism follows from the fact that the natural inclusion $S_{p}<HS_{p}$ induces an isomorphism
\begin{equation*}\label{eq:hspv}
(HS_{p})^{\vee}[p^{\infty}]\cong (S_{p}/[S_{p},HS_{p}])^{\vee}. 
\end{equation*}
Here, we use the assumption that $S_{p}$ is normal. 
\end{proof}

\begin{lem}\label{lem:dlgc}
Let $p$ be a prime number, $G$ a finite group of order a multiple of $p$, and $H$ a subgroup of $G$ satisfying $(G:H)\in p\Z$. We further assume that a $p$-Sylow subgroup $S_{p}$ of $G$ is normal. Take a $p$-subgroup $D$ of $G$. 
\begin{enumerate}
\item There is a commutative diagram
\begin{equation*}
\xymatrix{
H^{1}(D,\Ind_{HS_{p}}^{G}J_{HS_{p}/H}) \ar[d]^{\cong}\ar[r]^{\hspace{15pt}\delta_{D}} & H^{2}(D,\Ind_{HS_{p}}^{G}\Z) \ar[d]^{\cong}\\
\bigoplus_{g\in R(D,HS_{p})}C_{D,g} \ar@{^{(}->}[r]& C_{D,HS_{p}}. }
\end{equation*}
Here we use the notations as follows: 
\begin{itemize}
\item $R(D,HS_{p})$ is a complete representative of $D\backslash G/HS_{p}$ in $G$,
\item $C_{D,HS_{p}}:=\bigoplus_{g\in R(D,HS_{p})}D^{\vee}$,
\item for each $g\in R(D,HS_{p})$, a subset $R(D,g)$ of $HS_{p}$ is a complete representative of $g^{-1}Dg \backslash HS_{p}/H$, and $C_{D,g}$ is the kernel of the homomorphism
\begin{equation*}
D^{\vee}\rightarrow \bigoplus_{h \in R(D,g)}(D\cap \Ad(gh)H)^{\vee};f\mapsto (f\!\mid_{D\cap \Ad(gh)H})_{h}. 
\end{equation*}
\end{itemize}
Furthermore, the composite of the upper horizontal maps coincides with $\delta_{D}$. 
\item There is a commutative diagram
\begin{equation*}
\xymatrix@C=40pt{
H^{2}(G,\Ind_{HS_{p}}^{G}\Z)[p^{\infty}] \ar[d]^{\cong}\ar[r]^{\hspace{5pt}\Res_{G/D}} & H^{2}(D,\Ind_{HS_{p}}^{G}\Z) \ar[d]^{\cong} \\
(S_{p}/[S_{p},HS_{p}])^{\vee} \ar[r]^{\hspace{20pt}f\mapsto f\circ (\Ad(g^{-1})\mid_{D})_{g}}& C_{D,HS_{p}}. }
\end{equation*}
\end{enumerate}
\end{lem}

\begin{proof}
(i): Since $S_{p}$ is normal in $G$, we obtain $D<S_{p}$. Hence, Proposition \ref{prop:mcky} gives isomorphisms of $D$-lattices
\begin{gather*}
\Ind_{HS_{p}}^{G}\Z \cong \bigoplus_{g\in R(D,HS_{p})}\Z, \\
\Ind_{HS_{p}}^{G}J_{HS_{p}/H}\cong \bigoplus_{g\in R(D,HS_{p})}(J_{HS_{p}/H})^{g}\cong \bigoplus_{g\in R(D,HS_{p})}J_{g(HS_{p})g^{-1}/gHg^{-1}}. 
\end{gather*}
Combining these with Proposition \ref{prop:mcky}, we obtain a commutative diagram
\begin{equation*}
\xymatrix{
H^{1}(D,\Ind_{HS_{p}}^{G}J_{HS_{p}/H})\ar[r] \ar[d]^{\cong} & H^{2}(D,\Ind_{HS_{p}}^{G}\Z)\ar[d]^{\cong}\\
\bigoplus_{g\in R(D,HS_{p})}H^{1}(D,J_{g(HS_{p})g^{-1}/gHg^{-1}})\ar[r] & \bigoplus_{g\in R(D,HS_{p})}H^{2}(D,\Z). }
\end{equation*}
Here the lower homomorphism is the direct sum of the connecting homomorphism induced by the canonical exact sequence of $D$-lattices
\begin{equation*}
0 \rightarrow \Z \rightarrow \Ind_{gHg^{-1}}^{g(HS_{p})g^{-1}}\Z \rightarrow J_{g(HS_{p})g^{-1}/gHg^{-1}} \rightarrow 0. 
\end{equation*}
Hence it suffices to prove the commutativity of the following diagram for any $g\in R(D,HS_{p})$: 
\begin{equation}\label{eq:dcf1}
\xymatrix{
H^{1}(D,J_{g(HS_{p})g^{-1}/gHg^{-1}})\ar[r] \ar[d]^{\cong} & H^{2}(D,\Z)\ar[d]^{\cong}\\
C_{D,g}\ar[r] & D^{\vee}. }
\end{equation}
Fix $g\in R(D,HS_{p})$. Since every $h\in R(D,g)$ satisfies
\begin{equation*}
\Ad(ghg^{-1})(gHg^{-1})=\Ad(gh)H<g(HS_{p})g^{-1},
\end{equation*}
Proposition \ref{prop:mcky} implies an isomorphism of $D$-lattices
\begin{equation*}
\Ind_{gHg^{-1}}^{g(HS_{p})g^{-1}}\Z \cong \bigoplus_{h \in R(D,g)}\Ind_{D\cap \Ad(gh)H}^{D}\Z. 
\end{equation*}
Hence, Proposition \ref{prop:htch} induces the desired commutative diagram \eqref{eq:dcf1}. 

(ii): This follows from Proposition \ref{prop:rsch}. 
\end{proof}

\begin{prop}\label{prop:lciv}
Let $p$ be a prime number, $G$ a finite group of order a multiple of $p$, and $H$ a subgroup of $G$ satisfying $(G:H)\in p\Z$. We further assume that a $p$-Sylow subgroup $S_{p}$ of $G$ is normal. 
Denote by $\widetilde{M}_{HS_{p}}^{(\cD)}$ the preimage of $H^{2}(G,J_{G/HS_{p}})^{(\cD)}[p^{\infty}]$ under the homomorphism $H^{2}(G,\Ind_{HS_{p}}^{G}\Z)[p^{\infty}]\rightarrow H^{2}(G,J_{G/HS_{p}})[p^{\infty}]$ in Lemma \ref{lem:csds}. 
\begin{enumerate}
\item There is an isomorphism of abelian groups
\begin{equation*}
H^{2}(G,J_{G/HS_{p}})^{(\cD)}[p^{\infty}]\cong \widetilde{M}_{HS_{p}}^{(\cD)}/(S_{p}/[S_{p},G])^{\vee}. 
\end{equation*}
\item If $S_{p}$ is abelian, then $H^{2}(G,J_{G/H})[p^{\infty}]\cong (S_{p}/[S_{p},HS_{p}])^{\vee}$ in Lemma \ref{lem:csds} induces an isomorphism
\begin{equation*}
\widetilde{M}_{HS_{p}}^{(\cD)} \cong \{f\in (S_{p}/[S_{p},HS_{p}])^{\vee}\mid ((f\circ \Ad(g^{-1}))\!\mid_{D})_{g\in R(D,HS_{p})}\in \Delta_{D}+I_{D}\text{ for all }D\in \cD_{S_{p},H}\}. 
\end{equation*}
Here, $R(D,HS_{p})$ is a complete representative of $D\backslash G/HS_{p}$ in $G$, $\Delta_{D}$ is the diagonal subset in $\bigoplus_{g\in R(D,HS_{p})}D^{\vee}$, and
\begin{equation*}
I_{D}:=\bigoplus_{g\in R(D,HS_{p})}(D/(D\cap gHg^{-1}))^{\vee}. 
\end{equation*}
\item If $S_{p} \cong (C_{p})^{2}$, $S_{p}\cap H\cong C_{p}$ and $\cD_{S_{p},H}=\{S_{p}\cap H\}$, then we have
\begin{equation*}
\widetilde{M}_{HS_{p}}^{(\cD)}=(S_{p}/[S_{p},HS_{p}]\cdot [S_{p}\cap H,N_{G}(S_{p}\cap H)])^{\vee}. 
\end{equation*}
\end{enumerate}
\end{prop}

\begin{proof}
(i): Since $(G:HS_{p})$ is coprime to $p$, Proposition \ref{prop:tsan} (ii) gives an exact sequence
\begin{equation*}
0\rightarrow H^{2}(G,\Z)[p^{\infty}]\rightarrow H^{2}(G,\Ind_{HS_{p}}^{G}\Z)[p^{\infty}]\rightarrow H^{2}(G,J_{G/HS_{p}})[p^{\infty}] \rightarrow 0. 
\end{equation*}
We have $H^{2}(G,\Z)[p^{\infty}]\cong (S_{p}/[S_{p},G])^{\vee}$ by Lemma \ref{lem:csds} in the case $H<G$ with $\#H\cong (G:S_{p})$. Hence, the assertion follows from the definition of $\widetilde{M}_{HS_{p}}$. 

(ii): Recall that one has a commutative diagram as follows for each $D\in \cD_{H,S_{p}}$: 
\begin{equation*}
\xymatrix{
&H^{2}(G,\Ind_{HS_{p}}^{G}\Z)[p^{\infty}]\ar[r]\ar[d]^{\Res_{G/D}}&H^{2}(G,J_{G/HS_{p}})[p^{\infty}]\ar[r]\ar[d]^{\Res_{G/D}}&0\\
H^{2}(D,\Z)\ar[r]^{\varepsilon_{D}\hspace{20pt}}&H^{2}(D,\Ind_{HS_{p}}^{G}\Z)\ar[r]&H^{2}(D,J_{G/HS_{p}})\ar[r]&0. 
}
\end{equation*}
Moreover, by Lemma \ref{lem:jgcg}, the homomorphism $\delta_{D}$ is factored as
\begin{equation*}
H^{1}(D,\Ind_{HS_{p}}^{G}J_{HS_{p}/H})\rightarrow H^{2}(D,\Ind_{HS_{p}}^{G}\Z)\rightarrow H^{2}(D,J_{G/HS_{p}}). 
\end{equation*}
Hence, we have
\begin{equation*}
\widetilde{M}_{HS_{p}}=\{f\in H^{2}(G,\Ind_{HS_{p}}^{G}\Z)[p^{\infty}]\mid \Res_{G/D}(f)\in \Ima(\delta_{D})+\Ima(\varepsilon_{D})\text{ for any }D\in \cD_{H,S_{p}}\},
\end{equation*}
where $\delta_{D}$ is as in Lemma \ref{lem:dlgc} (i). In particular, we have
\begin{equation*}
\Ima(\delta_{D})\cong \bigoplus_{g\in R(D,HS_{p})}C_{D,g}=I_{D}. 
\end{equation*}
Furthermore, since $S_{p}$ is abelian, we have $D\cap \Ad(gh)H=D\cap gHg^{-1}$ for any $D\in \cD_{H,S_{p}}$, $g\in G$ and $h\in H$. This implies an equality
\begin{equation*}
\bigoplus_{g\in R(D,HS_{p})}C_{D,g}=I_{D}. 
\end{equation*}
On the other hand, by Proposition \ref{prop:rsch}, $H^{2}(D,\Ind_{HS_{p}}^{G}\Z)\cong C_{D,HS_{p}}$ in Lemma \ref{lem:dlgc} (i) induces an isomorphism
\begin{equation*}
\Ima(\varepsilon_{D})\cong \Delta_{D}. 
\end{equation*}
Hence, applying Lemma \ref{lem:dlgc} (ii), we obtain the desired isomorphism. 

(iii): Since $S_{p}\cong (C_{p})^{2}$ and $S_{p}\cap H \cong C_{p}$, one has the following for $g\in R(S_{p}\cap H,HS_{p})$: 
\begin{equation*}
C_{S_{p}\cap H,g}=
\begin{cases}
(S_{p}\cap H)^{\vee}&\text{if }g\in N_{G}(S_{p}\cap H),\\
0&\text{if }g\notin N_{G}(S_{p}\cap H). 
\end{cases}
\end{equation*}
Note that $N_{G}(S_{p}\cap H)$ contains $HS_{p}$. we have
\begin{equation*}
\Delta_{S_{p}\cap H}+I_{S_{p}\cap H}=\left\{(f_{g})\in \bigoplus_{g\in R(S_{p}\cap H,HS_{p})}(S_{p}\cap H)^{\vee}\,\middle|\, f_{g}=f_{g'}\text{ if }g,g'\in N_{G}(S_{p}\cap H)\right\}. 
\end{equation*}
Using this equality, we obtain the following: 
\begin{equation}\label{eq:mtsp}
\widetilde{M}_{HS_{p}}=\{f\in (S_{p}/[S_{p},HS_{p}])^{\vee}\mid f\circ \Ad(g^{-1})\mid_{S_{p}\cap H}=f\!\mid_{S_{p}\cap H}\text{ for any }g\in N_{G}(S_{p}\cap H)\}. 
\end{equation}
The right-hand side of \eqref{eq:mtsp} coincides with $(S_{p}/([S_{p},HS_{p}]\cdot [S_{p}\cap H,N_{G}(S_{p}\cap H)])^{\vee}$. Therefore, the assertion holds. 
\end{proof}

\begin{lem}\label{lem:msck}
Let $p$ be a prime number, $G$ a finite group of order a multiple of $p$, and $H$ a subgroup of $G$ satisfying $(G:H)\in p\Z$. We further assume that a $p$-Sylow subgroup $S_{p}$ of $G$ is elementary $p$-abelian and normal in $G$. If $[S_{p},HS_{p}]$ contains $S_{p}\cap H$, then we have 
\begin{equation*}
N_{G}(S_{p}\cap H)\neq Z_{G}(S_{p}\cap H). 
\end{equation*}
\end{lem}

\begin{proof}
Recall that there is an isomorphism $G\cong S_{p}\rtimes G'$ with $p\nmid \#G'$. Moreover, there is an isomorphism $S_{p}\cong (C_{p})^{m}$ for some $m\in \Zpn$ by Lemma \ref{lem:sbsd} (v). We regard $S_{p}$ as a representation of $G'$ over $\Fp$. By Lemma \ref{lem:sbsd} (iii), we may assume $H\cong (S_{p}\cap H)\rtimes H'$, where $H'$ is a subgroup of $G'$. Since $\#H'\notin p\Z$, Maschke's theorem (cf.~\cite[Chap.~1, Theorem 1]{Serre1977}) implies that there is a subgroup $S'_{p}$ of $S_{p}$ which is stable under $H'$ such that $S_{p}=(S_{p}\cap H)\times S'_{p}$. Under this decomposition, we have the following: 
\begin{equation}\label{cidc}
[S_{p},HS_{p}]=[S_{p}\cap H,H]\times [S'_{p},H]. 
\end{equation}

Assume $S_{p}\cap H\subset [S_{p},H]$. By \eqref{cidc}, one has an equality $[S_{p}\cap H,H]=S_{p}\cap H$. This implies the existence of $h\in H$ and $s\in S_{p}\cap H$ such that $hsh^{-1}s^{-1}\neq 1$. Hence we have $h\in H \setminus Z_{G}(S_{p}\cap H)$, in particular the desired assertion holds. 
\end{proof}

Now, we obtain the structure of $\Sha_{\cD}^{2}(G,J_{G/H})$ in the case $m=2$. 

\begin{thm}\label{thm:2dpt}
Let $p$ be a prime number, $G$ a transitive group of degree $\in p\Z \setminus p^{2}\Z$, and $H$ a corresponding subgroup of $G$. We further assume that
\begin{itemize}
\item a $p$-Sylow subgroup $S_{p}$ of $G$ is normal in $G$; and
\item $S_{p}\cong (C_{p})^{2}$. 
\end{itemize}
\begin{enumerate}
\item The abelian group $\Sha_{\omega}^{2}(G,J_{G/H})[p^{\infty}]$ is trivial or isomorphic to $\Z/p$. Moreover, it is non-zero if and only if
\begin{itemize}
\item[(b)] $[S_{p},G]=S_{p}$; and
\item[(c)] $N_{G}(S_{p}\cap H)=Z_{G}(S_{p}\cap H)$. 
\end{itemize}
\item Assume that (b) and (c) in (i) are satisfied. Let $\cD$ be an admissible set of subgroups of $G$. Then there is an isomorphism
\begin{equation*}
\Sha_{\cD}^{2}(G,J_{G/H})[p^{\infty}] \cong 
\begin{cases}
0&\text{if there is $D \in \cD$ which contains $S_{p}$,}\\
\Z/p&\text{otherwise. }
\end{cases}
\end{equation*}
\end{enumerate}
\end{thm}

\begin{proof}
It suffices to prove the following: 
\begin{itemize}
\item[(I)] if $S_{p}\subset D$ for some $D\in \cD$, then $\Sha_{\cD}^{2}(G,J_{G/H})[p^{\infty}]=0$. 
\item[(I\hspace{-1pt}I)] if $\cD$ does not satisfy the assumption in (I), then one has
\begin{equation*}
\Sha_{\cD}^{2}(G,J_{G/H})[p^{\infty}] \cong 
\begin{cases}
\Z/p &\text{if (b) and (c) hold, }\\
0 &\text{otherwise. }
\end{cases}
\end{equation*}
\end{itemize}

\textbf{Proof of (I): }Assume that $D\in \cD$ contains $S_{p}$. It suffices to prove that the composite
\begin{equation*}
H^{2}(G,J_{G/HS_{p}})[p^{\infty}] \hookrightarrow H^{2}(G,J_{G/HS_{p}})\xrightarrow{\Res_{G/S_{p}}} 
H^{2}(S_{p},J_{G/HS_{p}})
\end{equation*}
is injective. This follows from Proposition \ref{prop:tsan} (ii).

\textbf{Proof of (I\hspace{-1pt}I): }By Lemma \ref{lem:sbsd} (iv) (a) and $S_{p}\cong (\Z/p)^{2}$, the group $[S_{p},G]$ is non-trivial and $S_{p}=[S_{p},G]\cdot (S_{p}\cap H)$. Hence, Proposition \ref{prop:h2s2} and Lemma \ref{lem:exh1} (ii) imply an exact sequence
\begin{equation*}
0\rightarrow H^{1}(G,\Ind_{HS_{p}}^{G}J_{HS_{p}/H})[p^{\infty}]\xrightarrow{\delta_{G}} H^{2}(G,J_{G/HS_{p}})^{(\cD)}[p^{\infty}] \xrightarrow{\widehat{\N}_{G}} \Sha_{\cD}^{2}(G,J_{G/H})[p^{\infty}] \rightarrow 0. 
\end{equation*}
This sequence can be written as follows by Lemma \ref{lem:exh1} (iii) and Proposition \ref{prop:lciv} (i): 
\begin{equation*}
0\rightarrow (S_{p}/[S_{p},HS_{p}]\cdot (S_{p}\cap H))^{\vee} \rightarrow \widetilde{M}_{HS_{p}}/(S_{p}/[S_{p},G])^{\vee} \rightarrow \Sha_{\cD}^{2}(G,J_{G/H})[p^{\infty}] \rightarrow 0. 
\end{equation*}
Furthermore, since we have $\cD_{HS_{p}}=\{S_{p}\cap H\}$ by assumption, one has an isomorphism by Proposition \ref{prop:lciv} (iii): 
\begin{equation*}
\widetilde{M}_{HS_{p}}\cong (S_{p}/[S_{p},HS_{p}]\cdot [S_{p}\cap H,N_{G}(S_{p}\cap H)])^{\vee}. 
\end{equation*}

\textbf{Case 1.~$N_{G}(S_{p}\cap H)\neq Z_{G}(S_{p}\cap H)$. }

Since $S_{p}\cap H\cong \Z/p$, the assumption implies an equality $[S_{p}\cap H,N_{G}(S_{p}\cap H)]=S_{p}\cap H$. Therefore, the homomorphism $\delta_{G}$ is an isomorphism. This follows the desired assertion
\begin{equation*}
\Sha_{\cD}^{2}(G,J_{G/H})[p^{\infty}]=0. 
\end{equation*}

\textbf{Case 2.~$N_{G}(S_{p}\cap H)=Z_{G}(S_{p}\cap H)$. }

In this case, the group $[S_{p}\cap H,N_{G}(S_{p}\cap H)]$ is trivial. Moreover, the non-triviality of $[S_{p},G]$ induces an isomorphism
\begin{equation*}
(S_{p}/[S_{p},G])^{\vee} \cong 
\begin{cases}
0&\text{if }[S_{p},G]=S_{p},\\
\Z/p&\text{if }[S_{p},G]\neq S_{p},
\end{cases}
\end{equation*}

First, suppose that $[S_{p},H]$ is trivial. Then there is an isomorphism
\begin{equation*}
(S_{p}/[S_{p},HS_{p}]\cdot (S_{p}\cap H))^{\vee}\cong \Z/p. 
\end{equation*}
On the other hand, the following holds: 
\begin{equation*}
\widetilde{M}_{HS_{p}}/(S_{p}/[S_{p},G])^{\vee}\cong S_{p}^{\vee}/(S_{p}/[S_{p},G])^{\vee} \cong 
\begin{cases}
(\Z/p)^{2} &\text{if }[S_{p},G]=S_{p},\\
\Z/p &\text{if }[S_{p},G]\neq S_{p}. 
\end{cases}
\end{equation*}
Therefore, we obtain the desired isomorphism on $\Sha_{\cD}^{2}(G,J_{G/H})$. 

Second, assume that $[S_{p},H]$ is non-trivial. By Lemma \ref{lem:msck}, the group $[S_{p},H]$ does not contain $S_{p}\cap H$. In particular, the group $[S_{p},HS_{p}]$ has order $p$ and $S_{p}/([S_{p},H]\cdot (S_{p}\cap H))$ is trivial. Moreover, one has an isomorphism
\begin{equation*}
\widetilde{M}_{HS_{p}}/(S_{p}/[S_{p},G])^{\vee}\cong 
(S_{p}/[S_{p},HS_{p}])^{\vee}/(S_{p}/[S_{p},G])^{\vee}
\cong 
\begin{cases}
\Z/p&\text{if }[S_{p},G]=S_{p},\\
0&\text{if }[S_{p},G]\neq S_{p}. 
\end{cases}
\end{equation*}
This induces the desired assertion in Case 2. 
\end{proof}

As a summarization of Theorems \ref{thm:upts}, \ref{thm:pptt} and \ref{thm:2dpt}, we obtain the full structure of the abelian group $\Sha_{\cD}^{2}(G,J_{G/H})$ under the normality of a $p$-Sylow subgroup of $G$. 

\begin{thm}\label{thm:tspt}
Let $p$ be a prime number, $G$ a transitive group of degree $\in p\Z \setminus p^{2}\Z$, and $H$ a corresponding subgroup of $G$. We further assume that a $p$-Sylow subgroup $S_{p}$ of $G$ is normal in $G$.  
\begin{enumerate}
\item One has $\Sha_{\omega}^{2}(G,J_{G/H})[p^{\infty}]\neq 0$ if and only if
\begin{itemize}
\item[(a)] there is an isomorphism $S_{p}\cong (C_{p})^{2}$; 
\item[(b)] $[S_{p},G]=S_{p}$; and
\item[(c)] $N_{G}(S_{p}\cap H)=Z_{G}(S_{p}\cap H)$. 
\end{itemize}
Conversely, assume that \emph{(a)}, \emph{(b)} and \emph{(c)} hold. Then, for an admissible set $\cD$ of subgroups of $G$, there is an isomorphism 
\begin{equation*}
\Sha_{\cD}^{2}(G,J_{G/H})[p^{\infty}] \cong 
\begin{cases}
0&\text{if there is $D \in \cD$ which contains $S_{p}$;}\\
\Z/p &\text{otherwise. }
\end{cases}
\end{equation*}
\item For any admissible set $\cD$ of subgroups of $G$, there is an isomorphism
\begin{equation*}
\Sha_{\cD}^{2}(G,J_{G/H})^{(p)} \cong \Sha_{\cD_{/S_{p}}}^{2}(G/S_{p},J_{G/HS_{p}}). 
\end{equation*}
Here $\cD_{/S_{p}}$ is the set of subgroups of the form $DS_{p}/S_{p}$ where $D\in \cD$, which is an admissible set of subgroups of $G/S_{p}$. 
\end{enumerate}
\end{thm}

\section{Representations of finite groups over finite prime fields}\label{sect:fnrp}

Let $q$ be a power of a prime number. For an $\F_{q}$-representation of a finite group $G'$, we mean a pair $(V,\rho)$, where $V$ is an $\F_{q}$-vector space of \emph{finite dimension} and $\rho$ is a homomorphism from $G'$ to $\GL(V)$. We simply denote $(V,\rho)$ by $V$ if we need not express $\rho$ explicitly.  

\subsection{Extremal representations over finite prime fields}

\begin{lem}\label{lem:frrp}
Let $q$ be a power of a prime number, $G'$ a finite group, and $V$ an irreducible $\F_{q}$-representation of $G'$. Then there is a surjection $\F_{q}[G'] \twoheadrightarrow V$. 
\end{lem}

\begin{proof}
This follows from the Frobenius reciprocity. See \cite[p.~63]{Neukirch2000}. 
\end{proof}

Let $p$ be a prime number, and $V$ an $\Fp$-representation of a finite group $G'$. For a subspace $W$ of $V$, we denote by $\Stab_{G'}(W)$ and $\Fix_{G'}(W)$ the stabilizer and the pointwise stabilizer of $L$ in $G'$ respectively. 

\begin{dfn}\label{cdbc}
Let $p$ be a prime number, $G'$ a finite group of order coprime to $p$, and $H'$ its subgroup. We say that an $\Fp$-representation $V$ of $G'$ is \emph{$H'$-extremal} if
\begin{itemize}
\item[(A)] $\dim_{\Fp}(V)=2$; 
\item[(B)] $V^{G'}=\{0\}$; and
\item[(C)] there is a $1$-dimensional subspace $L$ of $V$ such that $H'\subset \Stab_{G'}(L)=\Fix_{G'}(L)$. 
\end{itemize}
Moreover, we call $L$ in (C) an \emph{$H'$-special line} in $V$. 
\end{dfn}

We discuss for $\Fp$-representations to be extremal, which will be used in Section \ref{sect:pfmt}. 

\begin{lem}\label{lem:ttst}
Let $\overline{G}$ be a subgroup of $\GL_{2}(\F_{2})$ of odd order. Then $\overline{G}$ is trivial or the subgroup $S_{3}$ generated by $
\begin{pmatrix}
0&1\\
1&1
\end{pmatrix}
$. Moreover, $S_{3}$ is the unique $3$-Sylow subgroup of $\GL_{2}(\F_{2})$, and has order $3$. 
\end{lem}

\begin{proof}
This follows from the isomorphism $\GL_{2}(\F_{2})\cong \fS_{3}$. 
\end{proof}

\begin{lem}[{\cite[I, pp.~142--143]{Carter1964}}]\label{cf2s}
Let $p>2$ be an odd prime number. 
\begin{enumerate}
\item Assume $p\equiv 1\bmod 4$. Put $s:=\ord_{2}(p-1)\geq 2$, and let $\zeta_{2^s}$ be a primitive $2^{s}$-th root of unity. Put
\begin{equation*}
X:=\begin{pmatrix}
\zeta_{2^s}&0\\
0&1
\end{pmatrix},\quad
Y:=\begin{pmatrix}
1&0\\
0&\zeta_{2^s}
\end{pmatrix},\quad
Z:=\begin{pmatrix}
0&1\\
1&0
\end{pmatrix}. 
\end{equation*}
Then the subgroup $S_{p,2}$ generated by $X$, $Y$ and $Z$ is a $2$-Sylow subgroup of $\GL_{2}(\Fp)$. 
\item Suppose $p\equiv 3\bmod 4$. Put $s:=\ord_{2}(p+1)$. Let $\zeta_{2^{s+1}}$ be a $2^{s+1}$-th root of unity, which is contained in $\F_{p^2}$. Put
\begin{equation*}
X:=\begin{pmatrix}
0&1\\
1&\zeta_{2^{s+1}}+\zeta_{2^{s+1}}^{p}
\end{pmatrix},\quad Y:=
\begin{pmatrix}
0&1\\
-1&0
\end{pmatrix}
\begin{pmatrix}
0&1\\
1&\zeta_{2^{s+1}}+\zeta_{2^{s+1}}^{p}
\end{pmatrix}. 
\end{equation*}
Then the subgroup $S_{p,2}$ generated by matrices $X$ and $Y$ is a $2$-Sylow subgroup of $\GL_{2}(\Fp)$. Moreover, there are relations
\begin{equation*}
X^{2^{s}}=\diag(-1,-1),\quad Y^{2}=1,\quad YXY^{-1}=X^{2^{s}-1}. 
\end{equation*}
\end{enumerate}
\end{lem}

\begin{cor}\label{cor:-1id}
Let $p>2$ be an odd prime number, and $\overline{G}$ a non-trivial $2$-subgroup of $\GL_{2}(\Fp)$. Then at least one of the following is satisfied: 
\begin{enumerate}
\item $\overline{G}$ contains the matrix $\diag(-1,-1)$; 
\item $\overline{G} \hookrightarrow C_{2^{s}}$ and $\overline{G}$ contains a conjugate of $\diag(1,-1)$, where $s:=\ord_{2}(p-1)$. 
\end{enumerate}
\end{cor}

\begin{proof}
By Sylow's theorem, we may assume that $\overline{G}$ is contained in $S_{p,2}$ in Lemma \ref{cf2s}. Recall that $g\in \GL_{2}(\Fp)$ is conjugate to $\diag(1,-1)$ if and only if the characteristic polynomial of $g$ is $x^{2}-1$. 

\textbf{Case 1.~$p\equiv 1\bmod 4$. }

Since the exponent of $S_{p,2}$ equals $2^{s}$, it suffices to prove the following. 
\begin{claim}
If $\overline{G}$ does not contain $\diag(-1,-1)$, then it is cyclic and contains an element $S$ whose characteristic polynomial equals $x^{2}-1$. 
\end{claim}
We shall prove Claim in the following. First, assume $\overline{G}\cap \langle X,Y\rangle=\{1\}$. Take $g\in \overline{G}\setminus \langle X,Y\rangle$ and write $g=
\begin{pmatrix}
0&b\\
c&0
\end{pmatrix}$. Then one has $\diag(bc,bc)=g^{2}=1$, and hence $bc=1$ since $\overline{G}\cap \langle X,Y\rangle=\{1\}$. In particular, the characteristic polynomial of $g$ equals $x^{2}-1$. Therefore Claim holds in this case. 

Second, assume that $\overline{G}\cap \langle X,Y\rangle$ is non-trivial. Then it contains $\diag(1,-1)$ or $\diag(-1,1)$. Since $\diag(-1,1)=Z\diag(1,-1)Z^{-1}$, we may assume $\diag(1,-1)\in \overline{G}$, which implies $\diag(-1,1)\notin \overline{G}$. Now we prove $\overline{G}\subset \langle X,Y\rangle$, which gives the cyclicity of $\overline{G}$ since $\langle X,Y\rangle \cong (C_{2^{s}})^{2}$. Suppose not, that is, there is a matrix $g\in \overline{G}\setminus \langle X,Y\rangle$. Write $g=
\begin{pmatrix}
0&b\\
c&0
\end{pmatrix}$. Then we have $bc=1$ since $\overline{G}$ contains no non-trivial scalar matrix. However, one has
\begin{equation*}
\begin{pmatrix}
0&b\\
-c&0
\end{pmatrix}=\diag(1,-1)g \in \overline{G}. 
\end{equation*}
This implies $\diag(-1,-1)\in \overline{G}$ since $bc=1$, which is a contradiction. Therefore, we obtain $\overline{G}\subset \langle X,Y\rangle$ as desired. 

\textbf{Case 2.~$p\equiv 3\bmod 4$. }

Let $i\in \Z$ and $j\in \{0,1\}$. By Lemma \ref{cf2s} (ii), we have
\begin{equation*}
(X^{i}Y^{j})^{2}=
\begin{cases}
X^{2i}&\text{if }n=0,\\
X^{2^{s}i}&\text{if }n=1. 
\end{cases}
\end{equation*}
Therefore, $\langle X^{i}Y^{j} \rangle$ does not contain $\diag(-1,-1)$ if and only if $i$ is even and $j=1$. In particular, one has 
\begin{itemize}
\item[(1)] $\diag(-1,-1)\in \overline{G}$; or
\item[(2)] $\overline{G}=\langle X^{i}Y\rangle$ for some $i\in 2\Z$. 
\end{itemize}

Then (1) gives the validity of (i). On the other hand, since $X^{i}Y$ is not a scalar matrix, the characteristic polynomial of $Y$ equals $x^{2}-1$. Consequently, (2) gives the validity of (ii). This completes the proof of Corollary \ref{cor:-1id}. 
\end{proof}

For a prime number $p$, consider two subgroups of $\GL_{2}(\Fp)$ as follows: 
\begin{equation*}
\bT(\Fp):=\{\diag(a,b)\in \GL_{2}(\Fp)\mid a,b\in \F_{p}^{\times}\},\quad \bM(\Fp):=\{\diag(1,b)\in \GL_{2}(\Fp)\mid b\in \F_{p}^{\times}\}. 
\end{equation*}
Note that we have $\bT(\F_{2})=\bM(\F_{2})=\{1\}$. 

\begin{lem}\label{cmtv}
Let $p>2$ be an odd prime number. Take $g\in \GL_{2}(\Fp)$ and $h=\diag(1,e)\in \bM(\Fp)$, where $e \neq 1$. If $ghg^{-1}\in \bM(\Fp)$, then $g$ is contained in $\bT(\Fp)$. 
\end{lem}

\begin{proof}
Since the characteristic polynomials of $h$ and $ghg^{-1}$ coincide, we obtain an equality $ghg^{-1}=h$. Write $g=\begin{pmatrix}
a&b\\
c&d
\end{pmatrix}$, where $a,b,c,d\in \Fp$. Then the equality $gh=hg$ is equivalent to the following: 
\begin{equation*}
\begin{pmatrix}
a&eb\\
c&ed
\end{pmatrix}
=
\begin{pmatrix}
a&b\\
ec&ed
\end{pmatrix}. 
\end{equation*}
Hence $e\neq 1$ implies $b=c=0$, that is, $g$ is contained in $\bT(\Fp)$. 
\end{proof}

\begin{cor}\label{cor:diag}
Let $p>2$ be an odd prime number, $\overline{G}$ a subgroup of $\GL_{2}(\Fp)$ of order coprime to $p$, and $\overline{H}$ a subgroup of $\overline{G}\cap \bM(\Fp)$. 
\begin{enumerate}
\item If $\{1\}\neq N^{\overline{G}}(\overline{H})\subset \bM(\Fp)$, then $\overline{G}$ is contained in $\bT(\Fp)$. 
\item If $(\overline{G}:\overline{H})=2$, then
\begin{itemize}
\item $\overline{G}$ is contained in $\bT(\Fp)$; or
\item $\#\overline{G}=2$ and $(\F_{p}^{\oplus 2})^{\overline{G}}\neq \{0\}$. 
\end{itemize}
\end{enumerate}
\end{cor}

\begin{proof}
(i): Take a non-trivial element $h$ of $N^{\overline{G}}(\overline{H})$. Write $h=\diag(1,e)$, where $e\neq 1$. Take any $g\in \overline{G}$. Then $ghg^{-1}$ is contained in $N^{\overline{G}}(\overline{H})\subset \bM(\Fp)$. Hence the assertion follows from Lemma \ref{cmtv}. 

(ii): Note that we have $N^{\overline{G}}(\overline{H})=\overline{H}$ since $(\overline{G}:\overline{H})=2$. If $\overline{H}\neq \{1\}$, then Lemma \ref{cmtv} implies that $\overline{G}$ is contained in $\bT(\Fp)$. On the other hand, if $\overline{H}$ is trivial, that is, $\#\overline{G}=2$, then the assertion follows from Corollary \ref{cor:-1id}. 
\end{proof}

\begin{prop}\label{prop:trtr}
Let $p$ be a prime number, $G'$ a finite group of order coprime to $p$, $H'$ a subgroup of $G'$, and $(V,\rho)$ an $H'$-extremal $\Fp$-representation of $G'$. Take an $H'$-special line $L$ of $V$. Put $\overline{G}:=\rho(G')$ and $\overline{H}:=\rho(H')$. Then one has $(\overline{G}:\overline{H})\geq 3$ and $N^{\overline{G}}(\overline{H})=\{1\}$. In particular, the action of $N^{G'}(H')$ on $V$ is trivial. 
\end{prop}

\begin{proof}
Let $L$ be an $H'$-special line in $V$. Take $v\in L$ and $v'\in V\setminus L$, where $\{v,v'\}$ is a basis of $V$. We identify $\GL(V)$ and $\GL_{2}(\Fp)$ by the isomorphism induced by $\{v,v'\}$. Then one has $\overline{G}\neq \{1\}$ and $\overline{H} \subset \bM(\Fp)$. In particular, we have $N^{\overline{G}}(\overline{H})\subset \bM(\Fp)$. 

\textbf{Case 1.~$p=2$. }

The equality $N^{\overline{G}}(\overline{H})=\{1\}$ follows from $\bM(\F_{2})=\{1\}$. On the other hand, since $\overline{G}$ is non-trivial, Lemma \ref{lem:ttst} implies $\#\overline{G}=3$. Hence we obtain $(\overline{G}:\overline{H})=3$, which completes the proof in this case. 

\textbf{Case 2.~$p>2$. }

First, we prove the inequality $(\overline{G}:\overline{H})\geq 3$. Suppose $(\overline{G}:\overline{H})\leq 2$. Then Corollary \ref{cor:diag} implies that $\overline{G}\subset \bT(\Fp)$ or $V^{G'}\neq \{0\}$ holds. Moreover, the former condition implies the latter one since $G'=\Stab_{G'}(L)=\Fix_{G'}(L)$. This is absurd since $V^{G'}=\{0\}$. Therefore we obtain $(\overline{G}:\overline{H})\geq 3$ as desired. 

Second, we prove the equality $N^{\overline{G}}(\overline{H})=\{1\}$. Suppose not, then one has $\overline{G}\subset \bT(\Fp)$ by Corollary \ref{cor:diag} (i). In particular, one has $\Stab_{G'}(L)=G'$. Combining this equality with $\Stab_{G'}(L)=\Fix_{G'}(L)$, we obtain that $V^{G'}$ contains $L$. This contradicts the assumption $V^{G'}=\{0\}$, and hence the proof is complete. 
\end{proof}

Let $\bZ(\Fp):=\{\diag(a,a)\in \GL_{2}(\Fp)\mid a\in \F_{p}^{\times}\}$, which is the center of $\GL_{2}(\Fp)$. We denote by $\PGL_{2}(\Fp)$ the quotient of $\GL_{2}(\Fp)$ by $\bZ(\Fp)$. The classification of subgroups of $\PGL_{2}(\Fp)$ is known by \cite{Dickson1901}, which is called \emph{Dickson's classification}. 

\begin{prop}[{\cite[Chap.~XII]{Dickson1901}}]\label{prop:pgl2}
Let $p>2$ be an odd prime number, and $G^{\dagger}$ a subgroup of $\PGL_{2}(\Fp)$ of order coprime to $p$. Then $G^{\dagger}$ satisfies one of the following: 
\begin{enumerate}
\item $G^{\dagger}\cong C_{n}$ for some divisor of $p-1$ or $p+1$; 
\item $G^{\dagger}\cong D_{n}$ for some divisor of $p-1$ or $p+1$; 
\item $G^{\dagger}$ is isomorphic to $\fA_{4}$, $\fS_{4}$ or $\fA_{5}$. 
\end{enumerate}
\end{prop}

Define a subgroup of $\GL_{2}(\Fp)$ as follows: 
\begin{equation*}
\bC(\Fp):=\left\{
\begin{pmatrix}
a&b\zeta_{p-1}\\
b&a
\end{pmatrix}
\in \GL_{2}(\Fp)\,\middle|\,a,b\in \F_{p},a^{2}-b^{2}\zeta_{p-1}\neq 0 \right\}. 
\end{equation*}
Note that a conjugate of $\bC(\Fp)$ is called a \emph{non-split Cartan subgroup}. 

\begin{prop}[{\cite[Chap.~XII]{Dickson1901}}]\label{prop:crtn}
Assume $p>2$, and let $\overline{G}$ be a maximal subgroup of $\GL_{2}(\Fp)$. We denote by $G^{\dagger}$ the image of $\overline{G}$ in $\PGL_{2}(\Fp)$. 
\begin{enumerate}
\item Assume that $G^{\dagger}$ is cyclic of order coprime to $p$. Then $\overline{G}$ is contained in a conjugate of $\bT(\Fp)$ or that of $\bC(\Fp)$. 
\item Assume $G^{\dagger}\cong D_{n}$, where $\gcd(n,p)=1$. Then $\overline{G}$ is contained in a conjugate of $\bT(\Fp)$ or that of $\bC(\Fp)$, where
\begin{equation*}
\bC(\Fp):=\left\{
\begin{pmatrix}
a&b\zeta_{p-1}\\
b&a
\end{pmatrix}
\in \GL_{2}(\Fp)\,\middle|\,a,b\in \F_{p},a^{2}-b^{2}\zeta_{p-1}\neq 0 \right\};
\end{equation*}
\end{enumerate}
\end{prop}

\begin{thm}\label{thm:zggv}
Assume $p>2$. Let $G'$ be a finite group of order prime to $p$, and $H'$ a subgroup of $G'$. Consider an $H'$-extremal $\Fp$-representation $(V,\rho)$, and put $\overline{G}:=\rho(G')$ and $\overline{H}:=\rho(H')$. 
\begin{enumerate}
\item We have one of the following: 
\begin{itemize}
\item[(1)] $\overline{G}\cong C_{d}$ and $\overline{H}=1$, where $d\geq 3$ is a divisor of $p-1$ or $p+1$; 
\item[(2)] $\overline{G}\cong D_{d'}$ and $\#\overline{H}\leq 2$, where $d'\geq 3$ is an odd divisor of $p-1$ or $p+1$. 
\end{itemize}
\item Suppose that the order of $\overline{G}$ is a multiple of $4$, and put $s:=\ord_{2}\#\overline{G}\geq 2$. Then one has $p\equiv 1\bmod 2^{s}$ and $\overline{G}$ is cyclic. 
\end{enumerate}
\end{thm}

\begin{proof}
(i): Take $v\in L\setminus \{0\}$ and $v'\in V\setminus L$, which forms a basis of $V$. We identify $\GL(V)$ and $\GL_{2}(\Fp)$ by the isomorphism induced by the basis $\{v,v'\}$. Then one has $\overline{G}\cap \bZ(\Fp)=\{1\}$ since $V$ is $H'$-extremal. In particular, the composite $\overline{G}\hookrightarrow \GL_{2}(\Fp)\rightarrow \PGL_{2}(\Fp)$ is injective. On the other hand, Corollary \ref{cor:-1id} implies that a $2$-Sylow subgroup of $\overline{G}$ is cyclic. Combining this result with Proposition \ref{prop:pgl2}, we obtain that $\overline{G}$ is isomorphic to a cyclic group or a dihedral group. 

We shall prove the validity of (1) or (2). 

\textbf{Case 1.~$\overline{G}\cong C_{n}$ for some $n\in \Zpn$. }

In this case, we may assume that $\overline{G}$ is contained in $\bT(\Fp)$ or $\bC(\Fp)$, which is a consequence of Proposition \ref{prop:crtn} (i). If $\overline{G}\subset \bT(\Fp)$, then $n$ is a divisor of $p-1$ since the exponent of $\bT(\Fp)$ equals $p-1$. If $\overline{G}\subset \bC(\Fp)$, then $n$ is a divisor of $p+1$, which is a consequence of $\overline{G}\cap \bZ(\Fp)=\{1\}$ and $\bZ(\Fp)\subset \bC(\Fp)$. Finally, the assertion $\overline{H}=\{1\}$ and $n \geq 3$ follow from Proposition \ref{prop:trtr}. 

\textbf{Case 2.~$\overline{G}\cong D_{n}$ for some $n\in \Zpn$. }

In this case, $n$ must be odd since a $2$-Sylow subgroup of $\overline{G}$ is cyclic. Furthermore, we may assume that $\overline{G}$ is contained in $NT(\Fp)$ or $NC(\Fp)$ by Proposition \ref{prop:crtn} (ii). Take the unique subgroup $\overline{N}$ of $\overline{G}$ of order $n$. Then Case 1 implies that $n$ is a divisor of $p-1$ or $p+1$. On the other hand, we have $N^{\overline{G}}(\overline{H})=\{1\}$ by Proposition \ref{prop:trtr}. This implies $\overline{H}\cap \overline{N}=\{1\}$, and hence $\#\overline{H}\leq 2$.  

(ii): By assumption, the order of $\overline{G}$ is a multiple of $4$. Hence $\overline{G}$ is cyclic by (i). Moreover, Corollary \ref{cor:-1id} implies $p\equiv 1\bmod 2^{s}$ since $\overline{G}\cap \bZ(\Fp)$ is trivial. This completes the proof. 
\end{proof}

\begin{cor}\label{cor:ndpi}
Let $p>2$ be an odd prime number, $G'$ a finite group of order coprime to $p$, and $H'$ a subgroup of $G'$. Assume the following: 
\begin{equation*}
\gcd((G':H'),p-1)\leq 2,\quad \gcd((G':H'),p+1)\in 2^{\Znn}. 
\end{equation*}
Then all $\Fp$-representations of $G'$ are not $H'$-extremal. 
\end{cor}

\begin{proof}
Assume that there is an $H'$-extremal $\Fp$-representation $(V,\rho)$ of $G'$. Put $\overline{G}:=\rho(G')$ and $\overline{H}:=\rho(H')$. By assumption, $(\overline{G}:\overline{H})$ is a power of $2$. Hence one has $(\overline{G}:\overline{H})\in 4\Z$ by Theorem \ref{thm:zggv} (i). On the other hand, Theorem \ref{thm:zggv} (ii) implies $p\equiv 1\bmod 4$ and $\overline{G}$ is cyclic. Consequently we obtain that $4$ is a common divisor of $(G':H')$ and $p-1$, which contradicts to the assumption. Hence the proof is complete. 
\end{proof}

\subsection{Representations of particular groups}

Here we discuss the existence of an $H'$-extremal $\Fp$-representation of $G'$ under certain condition on $(G':H')$. 

\begin{prop}\label{prop:tgdp}
Let $G$ be a transitive group of degree $\ell$, where $\ell$ is a prime number distinct from $p$. Assume that $G'$ has order coprime to $p$. Take a corresponding subgroup $H'$ of $G'$. If there exists an $H'$-extremal $\Fp$-representation of $G'$, then $\ell$ is a divisor of $p^{2}-1$ and $G'$ is isomorphic to $C_{\ell}$ or $D_{\ell}$. 
\end{prop}

\begin{proof}
Let $(V,\rho)$ be an $H'$-extremal $\Fp$-representation of $G'$. Fix an isomorphism $V\cong \F_{p}^{\oplus 2}$, and we identify $\GL(V)$ with $\GL_{2}(\Fp)$. Put $\overline{G}:=\rho(G')$ and $\overline{H}:=\rho(H')$. Since $(\overline{G}:\overline{H})$ divides $(G':H')$, one has $(\overline{G}:\overline{H})\in \{1,\ell\}$. Combining this with Theorem \ref{thm:zggv} (i), we obtain 
\begin{itemize}
\item[(1)] $(\overline{G}:\overline{H})=(G':H')=\ell \mid p^{2}-1$; and
\item[(2)] $\overline{G}$ is isomorphic to $C_{\ell}$ or $D_{\ell}$. 
\end{itemize}
By, (1) one has $H'=\rho^{-1}(\overline{H})$, and hence $\Ker(\rho)\subset H'$. On the other hand, Proposition \ref{prop:tggp} (i) implies $N^{G'}(H')=\{1\}$.  Consequently, we have $\Ker(\rho)=\{1\}$. This implies $G'\cong \overline{G}$, which completes the proof. 
\end{proof}

\begin{prop}\label{prop:idx4}
Let $p>2$ be an odd prime number, $G'$ a transitive subgroup of degree $4$ whose order is coprime to $p$, and $H'$ a corresponding subgroup of $G'$. If there is an $H'$-extremal $\Fp$-representation of $G'$, then we have $p\equiv 1 \bmod 4$ and $G'\cong C_{4}$ (in particular, $H'=\{1\}$). 
\end{prop}

\begin{proof}
Let $(V,\rho)$ be an $H'$-extremal representation of $G$. Put $\overline{G}:=\rho(G')$ and $\overline{H}:=\rho(H')$. Then $\overline{G}$ must be cyclic Theorem \ref{thm:zggv} (ii). Moreover, Theorem \ref{thm:zggv} (i) implies $\overline{H}=\{1\}$ and $\#\overline{G}=4$. In particular, we have $\overline{G}\cong C_{4}$ and $\Ker(\rho)=H'$. Combining this equality with $N^{G'}(H')=\{1\}$, we obtain that $\Ker(\rho)$ is trivial. Therefore one has $G'\cong \overline{G}\cong C_{4}$ and $H'=\{1\}$ as desired. 
\end{proof}

In the following, we construct $H'$-extremal $\Fp$-representations of $G'$, where $G'$ is a certain transitive group of degree a coprime to $p$, and $H'$ is a corresponding subgroup of $G'$. 

\begin{lem}\label{lem:chsm}
Let $q$ be a power of a prime number $p$, and $n$ a positive integer dividing $q-1$. Pick a primitive $n$-th root of unity $\zeta_{n}$ in $\F_{q}^{\times}$. For each $j \in \Z/n$, put
\begin{equation*}
\chi_{q,n}^{j}\colon C_{n} \rightarrow \F_{q}^{\times}; c_{n} \mapsto \zeta_{n}^{j}. 
\end{equation*}
We regard $\chi_{q,n}^{j}$ as a $1$-dimensional $\F_{q}$-representation of $C_{n}$. 
\begin{enumerate}
\item Let $j_{1},j_{2},j'_{1},j_{2}\in \Z/n$. Then there is an isomorphism $\chi_{q,n}$
\item Every $2$-dimensional $\F_{q}$-representation of $C_{n}$ is isomorphic to $\chi_{q,n}^{j_1}\oplus \chi_{q,n}^{j_2}$ for some $j_1,j_2\in \Z/n$. 
\end{enumerate}
\end{lem}

\begin{proof}
Since $n \mid q-1$, the polynomial $X^{n}-1$ is decomposed as the product of $X-\zeta_{n}^{i}$ for all $i\in \Z/n$, the homomorphism $\F_{q}[x]\rightarrow \F_{q}[G']$, which sends $x$ to $1\bmod n$, induces an isomorphism
\begin{equation*}
\F_{q}[G']\cong \prod_{i\in \Z/n}\F_{q}[X]/(X-\zeta_{n}^{i}). 
\end{equation*}
Hence, Lemma \ref{lem:frrp} implies that any irreducible representation of $G'$ is isomorphic to $\chi_{q,n}^{i}$ for some $i\in \Z/n$. Combining this result with Maschke's theorem, we obtain the desired assertion. 
\end{proof}

\begin{prop}\label{prop:clrd}
Let $p$ be a prime number which is not a Fermat prime, and $\ell$ an odd prime divisor of $p-1$. For $j_{1},j_{2}\in \Z/\ell$, consider an $\Fp$-representation
\begin{equation*}
U_{p,\ell}^{(j_1,j_2)}:=\chi_{p,\ell}^{j_1}\oplus \chi_{p,\ell}^{j_2}. 
\end{equation*}
\begin{enumerate}
\item The $\Fp$-representation $U_{p,\ell}^{(j_1,j_2)}$ of $C_{\ell}$ is $\{1\}$-extremal if and only if $j_{1}$, $j_{2}$, $j_{1}-j_{2}$ are non-zero. 
\item Suppose that $j_{1}$, $j_{2}$, $j_{1}-j_{2}$ are non-zero. Then a $1$-dimensional subspace $L$ of $V$ is a $\{1\}$-special line if and only if $L\neq \chi_{p,\ell}^{j_{1}}\oplus \{0\}$ and $L\neq \{0\}\oplus \chi_{p,\ell}^{j_{2}}$. 
\end{enumerate}
\end{prop}

\begin{proof}
By definition, (B) is equivalent to $j_{1},j_{2}\neq 0$. Moreover, (C) holds for $H'=\{1\}$ and $L=\langle(1,a)\rangle_{\Fp}$ if and only if $j_{1}\neq j_{2}$. These imply the desired assertions. On the other hand, we have
\begin{equation*}
(\chi_{n}^{j_1}\oplus \chi_{n}^{j_2})(d)\cdot (1,a)=(\zeta_{n}^{dj_1},\zeta_{n}^{dj_2})=\zeta_{n}^{j_1}(1,\zeta_{n}^{d(j_2-j_1)}a)
\end{equation*}
for any $d\in \Z/n$. 

(i): It suffices to prove (C) for $H'=\{0\}$ and $L=\langle(1,a)\rangle_{\Fp}$ under the assumption $j_1 \neq j_2$. Take $d\in \Stab_{G'}(L)$, that satisfies $d(j_2-j_1)=0$ in $\Z/\ell$. Hence we have $d=0$. This means the triviality of $\Stab_{G'}(L)$, which implies the condition (C). 

(ii): First, assume $j_1-j_2 \in 2\Z \setminus 4\Z$. We prove that $\chi_{4}^{i}\oplus \chi_{4}^{j}$ does not satisfy (B) or (C). In this case, the element $2\bmod 4$ in $\Z/4$ stabilizes $L$. Moreover, the assumption $j_1-j_2 \in 2\Z \setminus 4\Z$ implies that $(\chi_{4}^{j_1}\oplus \chi_{4}^{j_2})(d)$ is trivial if and only if $j_1,j_2\in 4\Z$. Hence we obtain the desired assertion. On the other hand, the same proof as (i) implies that (C) holds for $H'=\{1\}$ and $L=\langle(1,a)\rangle_{\Fp}$ if $j_{1}-j_{2}\notin 2\Z$. This completes the proof of (ii). 
\end{proof}

\begin{prop}\label{prop:cyir}
Let $p$ and $\ell$ be prime numbers that satisfy $2<\ell \mid p+1$. Put $G':=C_{\ell}$, and fix a primitive $\ell$-th root of unity $\zeta_{\ell}$, which is contained in $\F_{p^{2}}^{\times}\setminus \F_{p}^{\times}$. For each $j\in (\Z/\ell)^{\times}$, let $V_{p,\ell}^{j}$ be the $2$-dimensional $\Fp$-representation of $G'$ defined by the homomorphism
\begin{equation*}
C_{\ell} \rightarrow \GL_{2}(\Fp); c_{\ell} \mapsto 
\begin{pmatrix}
0&-1\\
1&\zeta_{\ell}^{j}+\zeta_{\ell}^{pj}
\end{pmatrix}. 
\end{equation*}
\begin{enumerate}
\item Let $i,j\in (\Z/\ell)^{\times}$. Then every $v\in V_{p,\ell}^{j}\setminus \{0\}$ satisfies $V_{p,\ell}^{j}=\langle v,c_{n}v\rangle_{\Fp}$ and a non-trivial linear relation
\begin{equation*}
c_{\ell}^{2}v=(\zeta_{\ell}^{ij}+\zeta_{\ell}^{-ij})c_{\ell}v-v. 
\end{equation*}
\item The $\Fp$-representation $V_{p,\ell}^{j}$ is irreducible and $\{1\}$-extremal for any $j\in (\Z/\ell)^{\times}$. Moreover, all $1$-dimensional subspaces of $V_{p,\ell}^{j}$ are $\{1\}$-special lines. 
\item Every non-trivial $2$-dimensional $\Fp$-representation of $C_{\ell}$ is isomorphic to $V_{p,\ell}^{j}$ for some $j\in (\Z/\ell)^{\times}$. 
\end{enumerate}
\end{prop}

\begin{proof}
(i): This follows from direct computation. 

(ii): This is a consequence of (i). 

(iii): The homomorphism $\Fp[x] \rightarrow \Fp[G']$, which maps $x$ to $1\bmod \ell$, induces an isomorphism
\begin{equation*}
\Fp[G']\cong \Fp[x]/(x-1)\times \prod_{j\in (\Z/\ell)^{\times}/(j\sim pj)}\Fp[x]/(x^{2}-(\zeta_{\ell}^{j}+\zeta_{\ell}^{pj})x+1). 
\end{equation*}
Here, all factors of the right-hand side are fields. Hence, Lemma \ref{lem:frrp} implies that any irreducible representation of $G'$ is isomorphic to a trivial representation or $V_{p,\ell}^{j}$ for some $j\in (\Z/\ell)^{\times}$. Now let $V$ be a non-trivial $2$-dimensional $\Fp$-representation. Then Maschke's theorem implies that $V$ is irreducible. Hence the assertion holds. 
\end{proof}

For $n\in \Zpn$, we write $D_{n}$ for the dihedral group of order $2n$, that is, 
\begin{equation*}
D_{n}=\langle \sigma_{n},\tau_{n}\mid \sigma_{n}^{n}=\tau_{n}^{2}=1,\tau_{n}\sigma_{n}\tau_{n}^{-1}=\sigma_{n}^{-1}\rangle. 
\end{equation*}

\begin{prop}\label{prop:dlbc}
Let $p$ and $\ell$ be prime numbers that satisfy $p\geq 5$ and $2<\ell \mid p^{2}-1$. For $j\in \{1,\ldots,(\ell-1)/2\}$, let $W_{p,\ell}^{j}$ be the $\Fp$-representation defined by the homomorphism
\begin{equation*}
D_{\ell} \rightarrow \GL_{2}(\Fp); \sigma_{\ell}^{i}\tau_{\ell}^{i'}\mapsto 
\begin{pmatrix}
0&-1\\
1&\zeta_{\ell}^{j}+\zeta_{\ell}^{-j}
\end{pmatrix}^{i}
\begin{pmatrix}
0&1\\
1&0
\end{pmatrix}^{i'}
\end{equation*}
\begin{enumerate}
\item For any $j\in \{1,\ldots,(\ell-1)/2\}$, the $\Fp$-representation $W_{p,\ell}^{j}$ of $D_{\ell}$ is $\langle \tau_{\ell}\rangle$-extremal. Moreover, $\langle (1,1)\rangle$ is the unique $\langle \tau_{\ell}\rangle$-special line in $W_{p,\ell}^{j}$. 
\item All $\langle \tau_{\ell}\rangle$-extremal $\Fp$-representations of $D_{\ell}$ are isomorphic to $W_{p,\ell}^{j}$ for some $j\in \{1,\ldots,(\ell-1)/2\}$. 
\end{enumerate}
\end{prop}

\begin{proof}
(i): Condition (B) follows from the fact that $W_{p,\ell}^{j}\otimes_{\Fp}\F_{p^2}$ has no non-trivial element that is stable under $D_{\ell}$. On the other hand, put $L_{c}:=\langle (1,c)\rangle$ for $c\in \{\pm 1\}$. By definition, $L_{1}$ and $L_{-1}$ are the $1$-dimensional suspace of $W_{p,\ell}^{j}$ which is stable under the action of $\tau_{\ell}$. Moreover, they are not stable under the action of $\sigma_{\ell}$. On the other hand, $\tau_{\ell}$ acts on $L_{c}$ by the $c$-multiple for each $c\in \{\pm 1\}$. Hence, $L_{1}$ satisfies (C), and $L_{1}$ is the unique $\langle \tau_{\ell}\rangle$-special line in $W_{p,\ell}^{j}$. 

(ii): For any field $F$ of characteristic not a divisor of $2\ell$, we denote by $R_{F}(D_{\ell})$ the Grothendieck group of finite-dimensional $F$-representations of $D_{\ell}$. Moreover, for each $j\in (\Z/\ell)^{\times}$, we denote by $W_{F,\ell}^{j}$ the $2$-dimensional $F$-representation of $D_{\ell}$ induced by the homomorphism
\begin{equation*}
D_{\ell}\rightarrow \GL_{2}(F);\,\sigma_{\ell}^{i}\tau_{\ell}^{i'}\mapsto 
\begin{pmatrix}
0&-1\\
1&\zeta_{\ell}^{j}+\zeta_{\ell}^{-j}
\end{pmatrix}^{i}
\begin{pmatrix}
0&1\\
1&0
\end{pmatrix}^{i'}. 
\end{equation*}
Recall that $W_{F,\ell}^{j}$ is absolutely irreducible. Moreover, the abelian group $R_{\C}(D_{\ell})$ is generated by characters and $W_{\C,\ell}^{j}$ for all $j\in \{1,\ldots,(\ell-1)/2\}$. This is a consequence of \cite[Section 5.3, pp.~37--38]{Serre1977} and \cite[Proposition 40, Section 14.1]{Serre1977}. Now fix an injection $\Qp \hookrightarrow \C$. As mentioned in \cite[Section 14.6, p.~122]{Serre1977}, the canonical homomorphism $R_{\Qp}(D_{\ell})\rightarrow R_{\C}(D_{\ell})$ is injective. Hence it is an isomorphism, which implies that $R_{\Qp}(D_{\ell})$ is generated by characters and $W_{\Qp,\ell}^{j}$ for $j\in \{1,\ldots,(\ell-1)/2\}$. Furthermore, by \cite[Proposition 43, Section 15.5]{Serre1977}, the homomoprhism $R_{\Qp}(D_{\ell})\rightarrow R_{\Fp}(D_{\ell})$ is an isomorphism. Consequently, $R_{\Fp}(D_{\ell})$ is generated by characters and $W_{\Fp,\ell}^{j}=W_{p,\ell}^{j}$ for $j\in \{1,\ldots,(\ell-1)/2\}$. This concludes the desired assertion. 
\end{proof}

\begin{prop}\label{prop:cfrd}
Let $p$ be a prime number with $p\equiv 1\bmod 4$. 
\begin{enumerate}
\item The representation $U_{p,4}^{(j_{1},j_{2})}:=\chi_{p,4}^{j_1}\oplus \chi_{p,4}^{j_2}$ of $C_{4}$ is $\{1\}$-extremal if and only if the subset $\{j_{1},j_{2}\}$ of $\Z/4$ coincides with $\{1,2\}$ or $\{-1,2\}$. 
\item If the above equivalent conditions hold, then a non-trivial subspace $L$ of $\chi_{p,4}^{j_{1}}\oplus \chi_{p,4}^{j_{2}}$ is a $\{1\}$-special line if and only if $L\neq \chi_{p,4}^{j_{1}}\oplus \{0\}$ and $L\neq \{0\}\oplus \chi_{p,4}^{j_{2}}$.  
\end{enumerate}
\end{prop}

\begin{proof}
By definition, (B) is equivalent to $j_{1},j_{2}\neq 0$. Moreover, (C) holds for $H'=\{1\}$ and $L=\langle(1,a)\rangle_{\Fp}$ if and only if $j_{1}\neq j_{2}$ and $2j_{1}\neq 2j_{2}$. Hence the assertions hold. 
\end{proof}

\subsection{Relation with semi-direct product groups}

Here we give a relation between extremal $\Fp$-representations and Theorem \ref{thm:tspt}. 

\begin{lem}\label{lem:ttch}
Let $G'$ be a finite group of coprime to $p$, and $V$ an $\Fp$-representation of $G'$. Then there is an isomorphism
\begin{equation*}
V/I_{G'}(V)\cong V^{G'},
\end{equation*}
where $I_{G'}(V)$ is the subspace of $V$ generated by $gv-v$ for all $g\in G'$ and $v\in V$. 
\end{lem}

\begin{proof}
By \cite[Chap.~VIII, \S 1]{Serre1979}, one has an exact sequence
\begin{equation*}
0\rightarrow \widehat{H}^{-1}(G',V)\rightarrow H_{0}(G',V)\rightarrow H^{0}(G',V)\rightarrow \widehat{H}^{0}(G',V)\rightarrow 0. 
\end{equation*}
By definition, we have $H_{0}(G',V)=V/I_{G'}(V)$ and $H^{0}(G',V)=V^{G'}$. On the other hand, since $\#G'$ is coprime to $p$, the abelian groups $\widehat{H}^{-1}(G',V)$ and $\widehat{H}^{0}(G',V)$ vanish. Therefore the assertion holds. 
\end{proof}

\begin{prop}\label{prop:bcgr}
Let $p$ be a prime number, $G$ a group of order a multiple of $p$degree $n \in p\Z \setminus p^{2}\Z$, and $H$ a corresponding subgroup of $G$. We further assume that a $p$-Sylow subgroup $S_{p}$ of $G$ is normal in $G$. 
Take a subgroup $G'$ of $G$ that admits isomorphisms $G\cong S_{p}\rtimes G'$ and $H\cong (S_{p}\cap H)\rtimes H'$ for some subgroup $H'$ of $G'$. We regard $S_{p}$ as an $\Fp$-representation of $G'$. 
\begin{enumerate}
\item The following are equivalent: 
\begin{itemize}
\item \emph{(a)}, \emph{(b)} and \emph{(c)} in Theorem \ref{thm:tspt} are satisfied; 
\item $S_{p}$ is $H'$-extremal and $S_{p}\cap H$ is an $H'$-special line in $S_{p}$. 
\end{itemize}
\item We further assume that $G$ is a transitive group of degree $n\in p\Z$ and $H$ is a corresponding subgroup of $G$. If \emph{(a)}, \emph{(b)} and \emph{(c)} hold, then $G'$ is a transitive group of degree $n/p$, and $H'$ is a corresponding subgroup of $G'$. 
\end{enumerate}
\end{prop}

\begin{proof}
(i): Condition (a) is equivalent to $\dim_{\Fp}(S_{p})=2$. Moreover, since $I_{G'}(S_{p})=[S_{p},G]$, Lemma \ref{lem:ttch} implies the equivalence between (b) and (B). On the other hand, the following hold, that can be confirmed by direct computation: 
\begin{equation*}
N_{G}(S_{p}\cap H)=S_{p} \rtimes \Stab_{G'}(S_{p}\cap H),\quad Z_{G}(S_{p}\cap H)=S_{p} \rtimes \Fix_{G'}(S_{p}\cap H). 
\end{equation*}
Therefore, conditions (c) and (C) for $L=S_{p}\cap H$ are equivalent. 

(ii): Assume that (a), (b) and (c) are valid. By Proposition \ref{prop:trtr}, the subgroup $N^{G'}(H')$ of $H$ is normal in $G$. This implies $N^{G'}(H')=\{1\}$ since $N^{G}(H)=\{1\}$. Moreover, we have $(G':H')=(G:H)/p$ by definition. Hence the assertion follows from Proposition \ref{prop:tggp} (ii). 
\end{proof}

\begin{cor}\label{cor:nabc}
Let $p>2$ be an odd prime number, $G$ a transitive group of degree $n \in p\Z \setminus p^{2}\Z$, and $H$ a corresponding subgroup of $G$. We further assume that a $p$-Sylow subgroup $S_{p}$ of $G$ is normal in $G$. Assume the following: 
\begin{equation*}
\gcd(n,p-1)\leq 2,\quad \gcd(n,p+1)\in 2^{\Znn}. 
\end{equation*}
Then $G$ does not satisfy one of (a), (b) and (c) in Theorem \ref{thm:tspt}. 
\end{cor}

\begin{proof}
Take a subgroup $G'$ of $G$ which admits isomorphisms $G\cong S_{p}\rtimes G'$ and $H\cong (S_{p}\cap H)\rtimes H'$ for some subgroup $H'$ of $G'$. Then we have the following: 
\begin{equation*}
\gcd((G':H'),p-1)=\gcd(n,p-1),\quad \gcd((G':H'),p+1)=\gcd(n,p+1). 
\end{equation*}
We regard $S_{p}$ as an $\Fp$-representation of $G'$. Then Corollary \ref{cor:ndpi} implies that $S_{p}$ is not $H'$-extremal. Hence the assertion follows from Proposition \ref{prop:bcgr} (i). 
\end{proof}

The following will be used in Section \ref{ssec:pfhn}. 

\begin{lem}\label{lem:abcp}
Let $p$ be a prime number, $G$ a transitive group of degree $\in p\Z \setminus p^{2}\Z$, and $H$ a corresponding subgroup of $G$. We further assume that 
\begin{itemize}
\item a $p$-Sylow subgroup $S_{p}$ of $G$ is normal; and
\item $G$, $H$ and $S_{p}$ satisfy \emph{(a)}, \emph{(b)} and \emph{(c)} in Theorem \ref{thm:tspt}. 
\end{itemize}
Consider an exact sequence of finite groups
\begin{equation*}
1\rightarrow N \rightarrow \widetilde{G} \xrightarrow{\pi} G \rightarrow 1,
\end{equation*}
where $N$ has order coprime to $p$, such that the action of $S_{p}$ on $N$ is trivial. Take a subgroup $\widetilde{H}$ of index coprime to $p$ in $\pi^{-1}(H)$, and denote by $\widetilde{S}_{p}$ a $p$-Sylow subgroup of $\pi^{-1}(S_{p})$. 
\begin{enumerate}
\item We have $(\widetilde{G}:\widetilde{H})\in p\Z \setminus p^{2}\Z$, and $\widetilde{S}_{p}$ a unique $p$-Sylow subgroup of $\widetilde{G}$. In particular,  $\widetilde{S}_{p}$ is normal in $\widetilde{G}$. 
\item Conditions \emph{(a)}, \emph{(b)} and \emph{(c)} in Theorem \ref{thm:tspt} are valid for $\widetilde{G},\widetilde{H}$ and $\widetilde{S}_{p}$. 
\end{enumerate}
\end{lem}

\begin{proof}
(i): The assertion $(\widetilde{G}:\widetilde{H})\in p\Z \setminus p^{2}\Z$ follows from $(G:H)\in p\Z\setminus p^{2}\Z$ and $\#N\in \Z\setminus p\Z$. On the other hand, since $\pi^{-1}(S_{p})$ acts trivially on $N$, we have $\pi^{-1}(S_{p})=\widetilde{S}_{p}\times N$. In particular, $\widetilde{S}_{p}$ is a characteristic subgroup of $\pi^{-1}(S_{p})$. This implies the normality of $\widetilde{S}_{p}$ in $\widetilde{G}$. 

(ii): Take a subgroup $G'$ of $G$ that admits isomorphisms $G\cong S_{p}\rtimes G'$ and $H\cong (S_{p}\cap H)\rtimes H'$ for some subgroup $H'$ of $G'$. Put $\widetilde{G}':=\pi^{-1}(G')$ and $\widetilde{H}'_{0}:=\pi^{-1}(H')$. Then we have $\widetilde{G}\cong S_{p}\rtimes \widetilde{G}'$ and $\pi^{-1}(H)\cong (S_{p}\cap H)\rtimes \widetilde{H}'_{0}$. Furthermore, $S_{p}\cap H$ is contained in $\widetilde{H}$ since $(\pi^{-1}(H):\widetilde{H})\notin p\Z$. Therefore, we have $\widetilde{H}=(S_{p}\cap H)\rtimes \widetilde{H}$ for some subgroup $\widetilde{H}$ of $\widetilde{H}_{0}$. 

By Proposition \ref{prop:bcgr} (i), it suffices to confirm that $S_{p}$ is $\widetilde{H}'$-extremal and $S_{p}\cap H$ is an $\widetilde{H}'$-special line in $S_{p}\cap H$. The equality $\dim_{\Fp}(S_{p})=2$ is clear. The condition $S_{p}^{\widetilde{G}'}=\{0\}$ is a consequence of the equality $S_{p}^{G'}=\{0\}$. Finally, condition $\widetilde{H}'\subset \Stab_{\widetilde{G}'}(S_{p}\cap H)=\Fix_{\widetilde{G}'}(S_{p}\cap H)$ follows from the fact that $N$ acts on $S_{p}$ trivially. 
\end{proof}

Let $G'$ be a finite group whose order is coprime to a prime number $p$, and $H'$ a subgroup of $G'$. In the following, for some specific pairs $(G',H')$, we discuss an equivalent condition on two $H'$-extremal $\Fp$-representations $V_{1}$ and $V_{2}$ of $G'$ for the existence of an isomorphism of groups $V_{1}\rtimes G'\cong V_{2}\rtimes G'$. 

\begin{prop}\label{prop:sdpb}
Let $p$ and $\ell$ be prime numbers satisfying $2<\ell \mid p-1$. Take two elements $\bj=(j_{1},j_{2})$ and $\bj'=(j'_{1},j'_{2})$ of $(\Z/\ell)^{2}$ where $j_{1},j_{2},j'_{1},j'_{2},j_{1}-j_{2},j'_{1}-j'_{2}$ are non-zero. Consider $\{1\}$-extremal $\Fp$-representations $U_{p,\ell}^{\bj}$ and $U_{p,\ell}^{\bj'}$ of $C_{\ell}$ in Proposition \ref{prop:clrd}, and let $G_{\bj}:=U_{p,\ell}^{\bj}\rtimes C_{\ell}$ and $G_{\bj'}:=U_{p,\ell}^{\bj'}\rtimes C_{\ell}$ be the semi-direct products associated to $U_{p,\ell}^{\bj}$ and $U_{p,\ell}^{\bj'}$ respectively. Then the following are equivalent: 
\begin{enumerate}
\item there is an isomorphism of groups $G_{\bj}\cong G_{\bj'}$; 
\item $\{j_{1},j_{2}\}=\{mj'_{1},mj'_{2}\}$ for some $m\in (\Z/\ell)^{\times}$. 
\end{enumerate}
\end{prop}

\begin{proof}
(i) $\Rightarrow$ (ii): Assume that there is an isomorphism $\psi \colon G_{\bj}\xrightarrow{\cong} G_{\bj'}$. Then $\varphi$ induces an isomorphism of $\Fp$-vector spaces
\begin{equation*}
\psi_{p}\colon U_{p,\ell}^{\bj}\xrightarrow{\cong}U_{p,\ell}^{\bj'}. 
\end{equation*}
For each $i\in \{1,2\}$, take $v_{i}\in \chi_{p,\ell}^{j_{i}}\setminus \{0\}$. Since $\ell$ divides $p-1$, $\zeta_{\ell}^{i}$ is contained in $\F_{\ell}^{\times}\cong (\Z/\ell)^{\times}$. Hence we have
\begin{equation*}
\psi(c_{\ell}v_{i}\rtimes 1)=\psi(\zeta_{\ell}^{j_{i}}v_{i}\rtimes 1)=\zeta_{\ell}^{j_{i}}\psi(v_{i}\rtimes 1)=\psi_{p}(\zeta_{\ell}^{j_{i}}(\psi_{p}(v_{i})\rtimes 1)). 
\end{equation*}
In particular, we obtain $\psi_{p}(c_{\ell}v_{i})=\zeta_{\ell}^{j_{i}}\psi_{p}(v_{i})$. Now, pick $v_{0}\in U_{p,\ell}^{\bj}$ and $m\in (\Z/\ell)^{\times}$ so that 
\begin{equation*}
\psi(1\rtimes c_{\ell})=(v_{0}\rtimes c_{\ell}^{m}). 
\end{equation*}
On the other hand, one has
\begin{equation}\label{eq:clpp}
\psi(c_{\ell}v_{i}\rtimes 1)=\psi(\Ad(1,c_{\ell})(v_{i}\rtimes 1))=\Ad(v_{0}\rtimes c_{\ell}^{m})(\psi_{p}(v_{i})\rtimes 1)=(c_{\ell}^{m}\psi_{p}(v_{i})\rtimes 1). 
\end{equation}
Therefore, one has an equality $c_{\ell}^{m}\psi_{p}(v_{i})=\zeta_{\ell}^{j_{i}}\psi_{p}(v_{i})$, which is equivalent to $c_{\ell}\psi_{p}(v_{i})=\zeta_{\ell}^{m^{-1}j_{i}}\psi_{p}(v_{i})$. However, there are exactly two $C_{\ell}$-invariant $1$-dimensional subspaces $U_{p,\ell}^{\bj'}$, and hence we obtain $\{m^{-1}j_{1},m^{-1}j_{2}\}=\{j'_{1},j'_{2}\}$. This is equivalent to $\{j_{1},j_{2}\}=\{mj'_{1},mj'_{2}\}$ as desired. 

(ii) $\Rightarrow$ (i): We may assume $j_{1}=mj'_{1}$ and $j_{2}=mj'_{2}$. We regard $\chi_{p,\ell}^{j_{i}}=\chi_{p,\ell}^{j'_{i}}=\Fp$ as abelian groups for each $i\in \{1,2\}$. Consider the map
\begin{equation*}
\psi \colon G_{\bj}\rightarrow G_{\bj'};\,(v,v')\rtimes c\mapsto (v,v')\rtimes c^{m}. 
\end{equation*}
Then it is a surjective homomorphism of groups. This implies that $\psi$ is an isomorphism since $G_{\bj}$ and $G_{\bj'}$ have the same order. 
\end{proof}

\begin{prop}\label{prop:sdpc}
Let $p$ and $\ell$ be prime numbers satisfying $2<\ell \mid p+1$. Take $j_{1},j_{2}\in (\Z/\ell)^{\times}$, and consider $\{1\}$-extremal $\Fp$-representations $V_{p,\ell}^{j_{1}}$ and $V_{p,\ell}^{j_{2}}$ of $C_{\ell}$ in Proposition \ref{prop:cyir}. For each $i\in \{1,2\}$, we denote by $G_{j_{i}}:=V_{p,\ell}^{j_{i}}\rtimes C_{\ell}$ the semi-direct product associated to $V_{p,\ell}^{j_{i}}$. Then there is an isomorphism of groups $G_{j_1}\cong G_{j_2}$. 
\end{prop}

\begin{proof}
We may assume $j_{2}=1$. We regard $V_{p,\ell}^{j_{1}}=V_{p,\ell}^{1}=\F_{p}^{\oplus 2}$ as abelian groups. Consider a map
\begin{equation*}
\psi \colon G_{j_{1}}\rightarrow G_{1};\,v\rtimes c\mapsto v\rtimes c^{j_{1}}. 
\end{equation*}
Then it is a homomorphism of groups by Proposition \ref{prop:cyir} (iii). Moreover, $\psi$ is surjective, and hence it is an isomorphism since $G_{j_1}$ and $G_{j_2}$ have the same order. 
\end{proof}

\begin{prop}\label{prop:sdpd}
Let $p$ and $\ell$ be prime numbers satisfying $2<\ell \mid p^{2}-1$. Take $j_{1},j_{2}\in (\Z/\ell)^{\times}$, and consider $\langle \tau_{\ell}\rangle$-extremal $\Fp$-representation $W_{p,\ell}^{j_{1}}$ and $W_{p,j}^{j_{2}}$ of $D_{\ell}$ in Proposition \ref{prop:dlbc}. We denote by $G_{j_{i}}:=W_{p,\ell}^{j_{i}}\rtimes D_{\ell}$ the semi-direct product associated to $W_{p,\ell}^{j_{i}}$ for each $i\in \{1,2\}$. Then there is an isomorphism of groups $G_{j_{1}}\cong G_{j_{2}}$. 
\end{prop}

\begin{proof}
We may assume $j_{2}=1$. We regard $W_{p,\ell}^{j_{1}}=W_{p,\ell}^{1}=\F_{p}^{\oplus 2}$ as abelian groups. Consider the map
\begin{equation*}
\psi \colon G_{j_{1}}\rightarrow G_{j_{2}};\,w\rtimes \sigma_{\ell}^{i}\tau_{\ell}^{i'}\mapsto w\rtimes \sigma_{\ell}^{ij_{1}}\tau_{\ell}^{i'}. 
\end{equation*}
Then the same argument as Proposition \ref{prop:sdpb}, (ii) $\Rightarrow$ (i) and Proposition \ref{prop:sdpc} imply that $\psi$ is a group homomorphism. Moreover, $\varphi$ is surjective, and hence it is an isomorphism since $G_{j_{1}}$ and $G_{j_{2}}$ have the same order. 
\end{proof}

\begin{cor}\label{cor:plcd}
Let $G$ be a transitive group of $p\ell$, where $p$ and $\ell$ are two distinct prime numbers, $H$ a corresponding subgroup of $G$. We further assume that a $p$-Sylow subgroup $S_{p}$ of $G$ is normal. 
Then \emph{(a)}, \emph{(b)} and \emph{(c)} in Theorem \ref{thm:tspt} are satisfied if and only if one of the following is valid: 
\begin{itemize}
\item[($\alpha$)] $2<\ell \mid p-1$, $G\cong (C_{p})^{2}\rtimes_{\varphi_{0,m}} C_{\ell}$ and $H\cong \Delta(C_{p}) \rtimes_{\varphi_{0,m}} \{1\}$, where $\varphi_{0,m}$ is defined by
\begin{equation*}
C_{\ell} \rightarrow \GL_{2}(\Fp); c_{\ell} \mapsto \diag(\zeta_{\ell},\zeta_{\ell}^{m}),
\end{equation*}
$\zeta_{\ell}\in \F_{p}^{\times}$ is a primitive $\ell$-th root of unity, and $m\in \{2,\ldots,\ell-1\}$ satisfies $mm'\not\equiv 1 \bmod \ell$ for any $m'\in \{1,\ldots,m-1\}$;
\item[($\beta$)] $2<\ell \mid p+1$, $G\cong (C_{p})^{2}\rtimes_{\varphi_{1}} C_{\ell}$ and $H\cong \Delta(C_{p}) \rtimes_{\varphi_{1}} \{1\}$, where $\varphi_{1}$ is defined by
\begin{equation*}
C_{\ell} \rightarrow \GL_{2}(\Fp); c_{\ell} \mapsto 
\begin{pmatrix}
0&-1\\
1&\zeta_{\ell}+\zeta_{\ell}^{-1}
\end{pmatrix}
\end{equation*}
and $\zeta_{\ell}\in \F_{p^{2}}^{\times}\setminus \F_{p}^{\times}$ is a primitive $\ell$-th root of unity. 
\item[($\gamma$)] $p\geq 5$, $2<\ell \mid p^{2}-1$, $G \cong (C_{p})^{2}\rtimes_{\varphi_{2}} D_{\ell}$ and $H\cong \Delta(C_{p}) \rtimes_{\varphi_{2}} \langle \tau_{\ell} \rangle$, where $\varphi_{2}$ is defined by
\begin{equation*}
D_{\ell} \rightarrow \GL_{2}(\Fp); \sigma_{\ell}^{i}\tau_{\ell}^{i'}\mapsto 
\begin{pmatrix}
0&-1\\
1&\zeta_{\ell}+\zeta_{\ell}^{p}
\end{pmatrix}^{i}
\begin{pmatrix}
0&1\\
1&0
\end{pmatrix}^{i'}
\end{equation*}
and $\zeta_{\ell} \in \F_{\ell^{2}}^{\times}$ is a primitive $\ell$-th root of unity. 
\end{itemize}
\end{cor}

\begin{proof}
It is not difficult to confirm that the groups in ($\alpha$), ($\beta$), ($\gamma$) satisfy (a), (b) and (c). In the following, suppose that (a), (b) and (c) are satisfied. Take a subgroup $G'$ of $G$ that admits isomorphisms $G=S_{p}\rtimes G'$ and $H=(S_{p}\cap H)\rtimes H'$ for some subgroup $H'$ of $G'$. Then Proposition \ref{prop:bcgr} (i) implies that (a), (b) and (c) are satisfied if and only if the $\Fp$-representation $S_{p}$ of $G'$ is $H'$-extremal. Moreover, by Proposition \ref{prop:tgdp}, if $S_{p}$ is $H'$-extremal, then one of the following is valid: 
\begin{itemize}
\item[($\alpha'$)] $\ell \mid p-1$ and $G'\cong C_{\ell}$; 
\item[($\beta'$)] $\ell \mid p+1$ and $G'\cong C_{\ell}$; or
\item[($\gamma'$)] $p\geq 5$, $\ell \mid p^{2}-1$ and $G'\cong D_{\ell}$. 
\end{itemize}
Therefore, we may assume the validity of ($\alpha'$), ($\beta'$) or ($\gamma'$). 

\textbf{Case 1.~}($\alpha'$) \textbf{holds. }

In this case, we have $H'=\{1\}$. By Lemma \ref{lem:chsm} (ii), there is an isomorphism of $\Fp[G']$-modules $S_{p}\cong U_{p,\ell}^{(j_{1},j_{2})}$ for some $j_{1},j_{2}\in \Z/\ell$. Here we use the assumption $\ell \mid p-1$. Moreover, by Proposition \ref{prop:clrd} (i), the $\Fp$-representation $S_{p}$ of $G'$ is $\{1\}$-extremal if and only if $j_{1},j_{2},j_{1}-j_{2}\in (\Z/\ell)^{\times}$. Now, take $m_{1},m_{2}\in \{1,\ldots,\ell-1\}$ that satisfy $m_{1}\equiv j_{1}^{-1}j_{2}$ and $m_{2}\equiv j_{1}j_{2}^{-1}$ modulo $\ell$. Then we have $m_{1}m_{2}\equiv 1 \bmod \ell$. Put $m:=\min\{m_{1},m_{2}\}$, which satisfies $mm'\not\equiv 1\bmod \ell$ for any $m'\in \{1,\ldots,m-1\}$. Then one has an isomorphism of groups
\begin{equation*}
\psi_{0} \colon G\xrightarrow{\cong} U_{p,\ell}^{(1,m)}\rtimes G'. 
\end{equation*}
Note that the right-hand side coincides with $(C_{p})^{2}\rtimes_{\varphi_{0,m}}C_{\ell}$ in ($\alpha$). On the other hand, by assumption, the groups $H$ and $\psi_{0}(H)$ have order $p$. Hence we have $\psi_{0}(H)=\langle (1,a) \rangle \rtimes \{1\}$ for some $a\in \F_{p}^{\times}$ by Proposition \ref{prop:clrd} (ii). Now, consider the automorphism of $\Fp[C_{\ell}]$-modules
\begin{equation*}
U_{p,\ell}^{(1,m)}\xrightarrow{\cong}U_{p,\ell}^{(1,m)};(x_{1},x_{2})\mapsto (x_{1},a^{-1}x_{2}). 
\end{equation*}
This map and the identity map on $C_{\ell}$ induce an automorphism $\psi_{1}$ on $U_{p,\ell}^{(1,m)}\rtimes G'$. Moreover, the composite $\psi_{1} \circ \psi_{0}$ maps $S_{p}\cap H$ onto $\langle (1,1)\rangle \rtimes \{1\}$. Therefore, we obtain the validity of the condition ($\alpha$). 

\textbf{Case 2.~}($\beta'$) \textbf{holds. }

In this case, we have $H'=\{1\}$. Since $S_{p}$ is non-trivial as an $\Fp$-representation of $G'$, Proposition \ref{prop:cyir} (iii) implies that there is an isomorphism $S_{p}\cong U_{p,\ell}^{j}$ for some $j\in (\Z/\ell)^{\times}$. Combining this with Proposition \ref{prop:sdpc}, we obtain an isomorphism of groups
\begin{equation*}
\psi_{0} \colon G\xrightarrow{\cong} U_{p,\ell}^{1}\rtimes G'. 
\end{equation*}
Note that the right-hand side coincides with $(C_{p})^{2}\rtimes_{\varphi_{1}}C_{\ell}$ in ($\beta$). On the other hand, by assumption, the groups $H$ and $\psi_{0}(H)$ have order $p$. Hence Proposition \ref{prop:cyir} (ii) implies an equality $\psi_{0}(H)=L \rtimes \{1\}$ for some $\{1\}$-special line $L$ in $U_{p,\ell}^{1}$. Now, take a generator $v$ of $L$. Then the map 
\begin{equation*}
U_{p,\ell}^{1}\xrightarrow{\cong}U_{p,\ell}^{1};v\mapsto (1,1)
\end{equation*}
is an automorphism of $\Fp[C_{\ell}]$-modules, which follows from Proposition \ref{prop:cyir} (i). This map and the identity map on $C_{\ell}$ induce an automorphism $\psi_{1}$ on $U_{p,\ell}^{1}\rtimes G'$. Moreover, the composite $\psi_{1} \circ \psi_{0}$ maps $S_{p}\cap H$ onto $\langle (1,1)\rangle \rtimes \{1\}$. Therefore, we obtain the validity of the condition ($\beta$). 

\textbf{Case 3.~}($\gamma'$) \textbf{holds. }

We may assume $H'=\langle \tau_{\ell}\rangle$. Since $S_{p}$ is non-trivial as an $\Fp$-representation of $G'$, Proposition \ref{prop:dlbc} (ii) gives the existence of an isomorphism $S_{p}\cong W_{p,\ell}^{j}$ for some $j\in \{1,\ldots,(\ell-1)/2\}$. Combining this with Proposition \ref{prop:sdpd}, we obtain an isomorphism of groups
\begin{equation*}
\psi_{0} \colon G\xrightarrow{\cong} W_{p,\ell}^{1}\rtimes G'. 
\end{equation*}
Note that the right-hand side coincides with $(C_{p})^{2}\rtimes_{\varphi_{2}}D_{\ell}$ in ($\gamma$). On the other hand, by assumption, the groups $H$ and $\psi_{0}(H)$ have order $2p$. Hence Proposition \ref{prop:dlbc} (i) implies an equality $\psi_{0}(H)=\langle (1,1)\rangle \rtimes H'$. Therefore, we obtain the validity of the condition ($\gamma$). 
\end{proof}

\begin{prop}\label{prop:sdcf}
Let $p$ be a prime number satisfying $p\equiv 1\bmod 4$. Take two elements $\bj=(j_{1},j_{2})$ and $\bj'=(j'_{1},j'_{2})$ of $(\Z/4)^{2}$ where $\{j_{1},j_{2}\},\{j'_{1},j'_{2}\}\in \{\{1,2\},\{-1,2\}\}$. Consider $\{1\}$-extremal $\Fp$-representations $U_{p,4}^{\bj}$ and $U_{p,4}^{\bj'}$ of $C_{4}$ in Proposition \ref{prop:cfrd}, and let $G_{\bj}:=U_{p,4}^{\bj}\rtimes C_{4}$ and $G_{\bj'}:=U_{p,4}^{\bj'}\rtimes C_{4}$ the semi-direct products associated to $U_{p,4}^{\bj}$ and $U_{p,4}^{\bj'}$ respectively. Then there is an isomorphism of groups $G_{\bj}\cong G_{\bj'}$. 
\end{prop}

\begin{proof}
We may assume $j_{2}=2$ and $\bj'=(1,2)$. Then we have $j_{1}\in \{\pm 1\}$. If $j_{1}=1$, then the assertion is clear. On the other hand, if $j_{1}=-1$, then we have $j_{1}=-j'_{1}$ and $j_{2}=j'_{2}$. Hence the assertion follows from the same argument as (ii) $\Rightarrow$ (i) in Proposition \ref{prop:sdpb}. 
\end{proof}

\begin{cor}\label{cor:fpcf}
Let $G$ be a finite group of degree $4p$, where $p$ is an odd prime number, and $H$ a corresponding subgroup of $G$. We further assume that a $p$-Sylow subgroup $S_{p}$ of $G$ is normal. Then \emph{(a)}, \emph{(b)} and \emph{(c)} in Theorem \ref{thm:tspt} are satisfied if and only $p\equiv 1 \bmod 4$, $G\cong (C_{p})^{2}\rtimes_{\varphi}C_{4}$ and $H\cong \Delta(C_{4})\rtimes_{\varphi}\{1\}$, where $\varphi$ is defined by
\begin{equation*}
C_{4}\rightarrow \GL_{2}(\Fp);\,c_{4}\mapsto \diag(\zeta_{4},-1). 
\end{equation*}
\end{cor}

\begin{proof}
The proof is the same as that of Corollary \ref{cor:plcd} in the case ($\alpha$) or ($\alpha'$) with Propositions \ref{prop:clrd} and \ref{prop:tgdp} replaced by Propositions \ref{prop:cfrd} and \ref{prop:idx4} respectively. 
\end{proof}

\section{Proof of main theorems}\label{sect:pfmt}

\subsection{Cohomological invariants}

\begin{thm}\label{thm:mtub}
Let $F$ be a field, and $E/F$ a finite separable field extension with its Galois closure $\widetilde{E}/F$. Put $G:=\Gal(\widetilde{E}/F)$ and $H:=\Gal(\widetilde{E}/E)$. We further assume 
\begin{itemize}
\item $[E:F]\in p\Z \setminus p^{2}\Z$; and 
\item there is a unique $p$-Sylow subgroup $S_{p}$ of $G$. 
\end{itemize}
\begin{enumerate}
\item The abelian group $H^{1}(F,\Pic(\overline{X}))[p^{\infty}]$ is trivial or isomorphic to $\Z/p$. Moreover, it is non-trivial if and only if
\begin{itemize}
\item[(a)] $S_{p}\cong (C_{p})^{2}$; 
\item[(b)] $[S_{p},G]=S_{p}$; and
\item[(c)] $N_{G}(S_{p}\cap H)=Z_{G}(S_{p}\cap H)$. 
\end{itemize}
\item Let $E_{0}/F$ be the subextension of $\widetilde{E}/F$ corresponding to $HS_{p}$. Then there is an isomorphism of finite abelian groups
\begin{equation*}
H^{1}(F,\Pic(\overline{X}))^{(p)}\cong H^{1}(F,\Pic(\overline{X}_{0})), 
\end{equation*}
where $\overline{X}_{0}=X_{0}\otimes_{F}F^{\sep}$ and $X_{0}$ is a smooth compactification of $T_{E_{0}/F}$ over $F$. 
\end{enumerate}
\end{thm}

\begin{proof}
By Proposition \ref{prop:host}, there are isomorphisms
\begin{equation*}
H^{1}(F,\Pic(\overline{X}))\cong \Sha_{\omega}^{2}(G,J_{G/H}),\quad 
H^{1}(F,\Pic(\overline{X}_{0}))\cong \Sha_{\omega}^{2}(G,J_{G/HS_{p}}). 
\end{equation*}
Then the assertion follows from the above isomorphisms and Theorem \ref{thm:tspt}. 
\end{proof}

\subsection{The Hasse norm principle}\label{ssec:pfhn}

\begin{thm}\label{thm:ptnt}
Let $K/k$ be a finite separable field extension of a global field with its Galois closure $\widetilde{K}/k$. Put $G:=\Gal(\widetilde{K}/k)$ and $H:=\Gal(\widetilde{K}/K)$. We further assume 
\begin{itemize}
\item $[K:k]\in p\Z \setminus p^{2}\Z$; and 
\item a $p$-Sylow subgroup $S_{p}$ of $G$ is normal in $G$. 
\end{itemize}
\begin{enumerate}
\item If $\Sha(K/k)[p^{\infty}]\neq 1$, then
\begin{itemize}
\item[(a)] $S_{p}\cong (C_{p})^{2}$; 
\item[(b)] $[S_{p},G]=S_{p}$; and
\item[(c)] $N_{G}(S_{p}\cap H)=Z_{G}(S_{p}\cap H)$. 
\end{itemize}
Conversely, if \emph{(a)}, \emph{(b)} and \emph{(c)} hold, then there is an isomorphism
\begin{equation*}
\Sha(K/k)[p^{\infty}]\cong 
\begin{cases}
1&\text{if a decomposition group of $\widetilde{K}/k$ contains $S_{p}$;}\\
\Z/p&\text{otherwise. }
\end{cases}
\end{equation*}
\item Let $K_{0}/k$ be the subextension of $\widetilde{K}/k$ corresponding to $HS_{p}$. Then there is an isomorphism
\begin{equation*}
\Sha(K/k)^{(p)}\cong \Sha(K_{0}/k). 
\end{equation*}
\end{enumerate}
\end{thm}

\begin{proof}
Let $\cD$ be the set of decomposition groups of $\widetilde{K}/k$. Then Corollary \ref{cor:onhn} gives an isomorphism
\begin{equation*}
\Sha(K/k) \cong \Sha_{\cD}^{2}(G,J_{G/H})^{\vee}. 
\end{equation*}
Hence the assertion from Theorem \ref{thm:tspt}. 
\end{proof}

\begin{cor}\label{cor:pttr}
Let $K/k$ be a finite separable field extension of a global field with its Galois closure $\widetilde{K}/k$. Put $G:=\Gal(\widetilde{K}/k)$ and $H:=\Gal(\widetilde{K}/K)$. We further assume 
\begin{itemize}
\item $[K:k]\in p\Z \setminus p^{2}\Z$; 
\item $n:=[K:k]/p$ satisfies $\gcd(n,p-1)\leq 2$ and $\gcd(n,p+1)\in 2^{\Znn}$; and
\item a $p$-Sylow subgroup $S_{p}$ of $G$ is normal. 
\end{itemize}
Then there is an isomorphism
\begin{equation*}
\Sha(K/k) \cong \Sha(K_{0}/k),
\end{equation*}
where $K_{0}/k$ is the subextension of $\widetilde{K}/k$ corresponding to $HS_{p}$. In particular, $\Sha(K/k)$ has no $p$-primary torsion part. 
\end{cor}

\begin{proof}
It suffices to prove $\Sha(K/k)[p^{\infty}]=1$. By Theorem \ref{thm:ptnt}, it suffices to prove that (a), (b) or (c) fails. however, it follows from Corollary \ref{cor:nabc}. 
\end{proof}

\begin{lem}[{\cite[Lemma 4]{Bartels1981a}}]\label{lem:bata}
Let $K/k$ be a finite separable field extension of a global field. If $[K:k]$ is a prime number, then $\Sha(K/k)=1$. 
\end{lem}

\begin{proof}
Let $\widetilde{K}/k$ be the Galois closure of $K/k$. Put $G:=\Gal(\widetilde{K}/k)$ and $H:=\Gal(\widetilde{K}/K)$. By Corollary \ref{cor:onhn}, there is an isomorphism
\begin{equation*}
\Sha(K/k)\cong \Sha_{\cD}^{2}(G,J_{G/H})^{\vee}, 
\end{equation*}
where $\cD$ is the set of decomposition subgroups of $G$. Hence, it suffices to prove $\Sha_{\omega}^{2}(G,J_{G/H})=0$. However, this follows from Proposition \ref{prop:bart} since $(G:H)=[K:k]$ is prime. 
\end{proof}

\begin{thm}\label{thm:pldt}
Let $p$ and $\ell$ be two distinct prime numbers. Consider a finite separable extension $K/k$ of degree $p\ell$ of a global field, and write for $\widetilde{K}/k$ its Galois closure. Put $G:=\Gal(\widetilde{K}/k)$ and $H:=\Gal(\widetilde{K}/K)$. We further assume that a $p$-Sylow subgroup $S_{p}$ of $G$ is normal. If $\Sha(K/k)\neq 1$, then one of the following is satisfied: 
\begin{itemize}
\item[($\alpha$)] $2<\ell \mid p-1$, $G\cong (C_{p})^{2}\rtimes_{\varphi_{0,m}} C_{\ell}$ and $H\cong \Delta(C_{p}) \rtimes_{\varphi_{0,m}} 1$, where $\varphi_{0,m}$ is defined by
\begin{equation*}
C_{\ell} \rightarrow \GL_{2}(\Fp); c_{\ell} \mapsto \diag(\zeta_{\ell},\zeta_{\ell}^{m}),
\end{equation*}
$\zeta_{\ell}\in \F_{p}^{\times}$ is a primitive $\ell$-th root of unity, and $m\in \{2,\ldots,\ell-1\}$ satisfies $mm'\not\equiv 1 \bmod \ell$ for any $m'\in \{1,\ldots,m-1\}$;
\item[($\beta$)] $2<\ell \mid p+1$, $G\cong (C_{p})^{2}\rtimes_{\varphi_{1}} C_{\ell}$ and $H\cong \Delta(C_{p}) \rtimes_{\varphi_{1}} 1$, where $\varphi_{1}$ is defined by
\begin{equation*}
C_{\ell} \rightarrow \GL_{2}(\Fp); c_{\ell} \mapsto 
\begin{pmatrix}
0&-1\\
1&\zeta_{\ell}+\zeta_{\ell}^{-1}
\end{pmatrix}
\end{equation*}
and $\zeta_{\ell}\in \F_{p^{2}}^{\times}\setminus \F_{p}^{\times}$ is a primitive $\ell$-th root of unity. 
\item[($\gamma$)] $p\geq 5$, $2<\ell \mid p^{2}-1$, $G \cong (C_{p})^{2}\rtimes_{\varphi_{2}} D_{\ell}$ and $H\cong \Delta(C_{p}) \rtimes_{\varphi_{2}} \langle \tau_{\ell} \rangle$, where $\varphi_{2}$ is defined by
\begin{equation*}
D_{\ell} \rightarrow \GL_{2}(\Fp); \sigma_{\ell}^{i}\tau_{\ell}^{i'}\mapsto 
\begin{pmatrix}
0&-1\\
1&\zeta_{\ell}+\zeta_{\ell}^{p}
\end{pmatrix}^{i}
\begin{pmatrix}
0&1\\
1&0
\end{pmatrix}^{i'}
\end{equation*}
and $\zeta_{\ell} \in \F_{\ell^{2}}^{\times}$ is a primitive $\ell$-th root of unity. 
\end{itemize}
Conversely, if \emph{($\alpha$)}, \emph{($\beta$)} or \emph{($\gamma$)} holds, then there is an isomorphism
\begin{equation*}
\Sha(K/k) \cong 
\begin{cases}
1&\text{if a decomposition group of $\widetilde{K}/k$ contains $S_{p}$; }\\
\Z/p &\text{otherwise. }
\end{cases}
\end{equation*}
\end{thm}

\begin{proof}
By Theorem \ref{thm:ptnt} (ii), there is an isomorphism
\begin{equation*}
\Sha(K/k)^{(p)}\cong \Sha(K_{0}/k), 
\end{equation*}
where $K_{0}$ is the intermediate field of $\widetilde{K}/k$ corresponding to $HS_{p}$. Since $(G:HS_{p})=\ell$ is a prime number, one has $\Sha(K_{0}/k)=1$ by Lemma \ref{lem:bata} (i). Hence we obtain an equality 
\begin{equation*}
\Sha(K/k)=\Sha(K/k)[p^{\infty}]. 
\end{equation*}

In the following, fix an isomorphism $G\cong S_{p}\rtimes G'$, where $G'$ is a subgroup of $G$, such that $H$ corresponds to $(S_{p}\cap H)\rtimes H'$ for some subgroup $H'$ of $G'$. Note that this is possible by Lemma \ref{lem:sbsd} (v). We regard $S_{p}$ as an $\Fp$-representation of $G'$. By Corollary \ref{cor:plcd}, the validity of (a), (b) and (c) are valid if and only if ($\alpha$), ($\beta$) or ($\gamma$) is satisfied. Hence, if $\Sha(K/k)\neq 1$, then Theorem \ref{thm:ptnt} (i) implies that ($\alpha$), ($\beta$) or ($\gamma$) is satisfied. Conversely, if ($\alpha$), ($\beta$) or ($\gamma$) is satisfied then Theorem \ref{thm:ptnt} (i) also follows that 
\begin{equation*}
\Sha(K/k) \cong 
\begin{cases}
1&\text{if a decomposition group of $\widetilde{K}/k$ contains $S_{p}$; }\\
\Z/p &\text{otherwise. }
\end{cases}
\end{equation*}
Therefore, the proof is complete. 
\end{proof}

\begin{thm}\label{thm:sfhf}
Let $p$ and $\ell$ be two distinct prime numbers satisfying $2<\ell \mid p^{2}-1$, and $d$ a multiple of $p\ell$ that is square-free. Take a global field $k$ whose characteristic is different from $p$. Then there is a finite extension $K$ of $k$ of degree $d$ for which the Hasse norm principle fails. 
\end{thm}

\begin{proof}
Write $d=p\ell d'$, where $d'$ is a positive integer. Put
\begin{equation*}
G:=(C_{p})^{2}\rtimes C_{\ell}, \quad H:=\Delta(C_{p})\rtimes C_{\ell},
\end{equation*}
where the semi-direct product is the same as Theorem \ref{thm:pldt} ($\alpha$) for $m=\ell-1$ if $p\equiv 1\bmod \ell$, and is the same as  Theorem \ref{thm:pldt} ($\beta$) if $p\equiv -1\bmod \ell$. We denote by $S_{p}$ the unique $p$-Sylow subgroup of $G$, and set 
\begin{equation*}
\widetilde{G}:=G\times C_{d'},\quad \widetilde{H}:=H\times \{1\},\quad \widetilde{S}_{p}:=S_{p}\times \{1\}. 
\end{equation*}
Then there is an isomorphism $\widetilde{G}\cong \widetilde{S}_{p}\rtimes C_{\ell d'}$ and $(G:H)=p\ell d'=d$. In particular, Moreover, Lemma \ref{lem:abcp} implies that
\begin{itemize}
\item $\widetilde{S}_{p}$ is a unique $p$-Sylow subgroup of $\widetilde{G}$ (hence $\widetilde{S}_{p}$ is normal in $G$); and
\item $\widetilde{G}$, $\widetilde{H}$, and $\widetilde{S}_{p}$ satisfy (a), (b) and (c) in Theorem \ref{thm:ptnt}. 
\end{itemize}
On the other hand, since $C_{\ell d'}$ is cyclic, Proposition \ref{prop:shf2} gives the existence of a finite Galois extension $K_{0}/k$ with Galois group $\widetilde{G}/\widetilde{S}_{p}$. In addition, by Proposition \ref{prop:shf1}, there is a finite Galois extension $\widetilde{K}/k$ with Galois group $\widetilde{G}$ such that $\widetilde{K}^{\widetilde{S}_{p}}=K_{0}$ and all decomposition groups of $\widetilde{K}/k$ are cyclic. Now, let $K$ be the intermediate field of $\widetilde{K}/k$ corresponding to $\widetilde{H}$. Then, Theorem \ref{thm:ptnt} implies $\Sha(K/k)[p^{\infty}]\cong \Z/p$ since (a), (b) and (c) are satisfied. Furthermore, we have $\Sha(K/k)^{(p)}\cong \Sha(K_{0}/k)$ by Theorem \ref{thm:ptnt} (ii), where $K_{0}$ is the intermediate field of $\widetilde{K}/k$ corresponding to $\widetilde{S}_{p}$. Since $\widetilde{G}/\widetilde{S}_{p}\cong C_{\ell d'}$, we obtain $\Sha(K_{0}/k)=1$ by Lemma \ref{lem:bata} (ii). In summary, there is an isomorphism
\begin{equation*}
\Sha(K/k)\cong \Z/p, 
\end{equation*}
which concludes the proof. 
\end{proof}

\begin{prop}[{cf.~\cite[Proposition 1]{Kunyavskii1984}}]\label{prop:knth}
Let $k$ be a global field, and $K/k$ a finite separable field extension of degree $4$ with its Galois closure $\widetilde{K}/k$. Put $G:=\Gal(\widetilde{K}/k)$ and $H:=\Gal(\widetilde{K}/K)$. 
\begin{enumerate}
\item If $\Sha(K/k)\neq 1$, then we have
\begin{itemize}
\item[(a)] $G\cong (C_{2})^{2}$; or
\item[(b)] $G\cong \fA_{4}$. 
\end{itemize}
\item If \emph{(a)} or \emph{(b)} in \emph{(i)} is satisfied, then there is an isomorphism
\begin{equation*}
\Sha(K/k)\cong 
\begin{cases}
1&\text{if there is a place of $\widetilde{K}/k$ containing $(C_{2})^{2}$; }\\
\Z/2&\text{otherwise. }
\end{cases}
\end{equation*}
\end{enumerate}
\end{prop}

\begin{thm}\label{thm:pctw}
Let $p>2$ be an odd prime number. Consider a finite separable extension $K/k$ of degree $4p$ of a global field, and write for $\widetilde{K}/k$ its Galois closure. Put $G:=\Gal(\widetilde{K}/k)$ and $H:=\Gal(\widetilde{K}/K)$. We further assume that a $p$-Sylow subgroup $S_{p}$ of $G$ is normal. 
\begin{enumerate}
\item If $\Sha(K/k)$ is non-trivial, then it is isomorphic to $\Z/2$ or $\Z/p$. 
\item If $\Sha(K/k)\cong \Z/p$, then $p\equiv 1\bmod 4$ and there exist an isomorphisms
\begin{equation}\label{eq:fpzp}
G\cong (C_{p})^{2}\rtimes_{\varphi_{p,4}}C_{4},\quad H\cong \Delta(C_{p}) \rtimes_{\varphi_{p,4}}\{1\}. 
\end{equation}
Here $\varphi_{p,4}$ is defined by the homomorphism
\begin{equation*}
C_{4} \rightarrow \GL_{2}(\Fp);\,c_{4} \mapsto \diag(-1,\sqrt{-1}). 
\end{equation*}
Conversely, if $p\equiv 1\bmod 4$ and \eqref{eq:fpzp} hold, then one has an isomorphism
\begin{equation*}
\Sha(K/k)\cong 
\begin{cases}
1&\text{if a decomposition group of $\widetilde{K}/k$ contains $S_{p}$; }\\
\Z/p &\text{otherwise. }
\end{cases}
\end{equation*}
\item If $\Sha(K/k)\cong \Z/2$, then we have
\begin{itemize}
\item[(a)] $G/N^{G}(HS_{p})\cong (C_{2})^{2}$; or 
\item[(b)] $p\geq 5$ and $G/N^{G}(HS_{p})\cong \fA_{4}$. 
\end{itemize}
Conversely, if \emph{(a)} or \emph{(b)} is satisfied, then there is an isomorphism
\begin{equation*}
\Sha(K/k)\cong 
\begin{cases}
1&\text{if a decomposition group of $\widetilde{K}/k$ contains $(C_{2})^{2}$; }\\
\Z/2&\text{otherwise. }
\end{cases} 
\end{equation*}
\end{enumerate}
\end{thm}

\begin{proof}
By Theorem \ref{thm:ptnt} (i) and Corollary \ref{cor:fpcf}, we have the following. 
\begin{itemize}
\item[(1)] If $\Sha(K/k)[p^{\infty}]\neq 1$, then $p\equiv 1 \bmod 4$ and \eqref{eq:fpzp} is valid. 
\item[(2)] If $p\equiv 1 \bmod 4$ and \eqref{eq:fpzp} is valid, then
\begin{equation*}
\Sha(K/k)[p^{\infty}]\cong 
\begin{cases}
1&\text{if a decomposition group of $\widetilde{K}/k$ contains $S_{p}$; }\\
\Z/p &\text{otherwise. }
\end{cases}
\end{equation*}
\end{itemize}
On the other hand, since Theorem \ref{thm:ptnt} (ii) gives an isomorphism $\Sha(K/k)^{(p)}\cong \Sha(K_{0}/k)$, where $K_{0}$ is the intermediate field of $\widetilde{K}/k$ corresponding to $HS_{p}$, Proposition \ref{prop:knth} implies the following. 
\begin{itemize}
\item[(3)] If $\Sha(K/k)^{(p)}\neq 1$, then (a) or (b) is satisfied. 
\item[(4)] If (a) or (b) holds, then
\begin{equation*}
\Sha(K/k)^{(p)}\cong 
\begin{cases}
1&\text{if a decomposition group of $\widetilde{K}/k$ contains $(C_{2})^{2}$; }\\
\Z/2&\text{otherwise. }
\end{cases} 
\end{equation*}
\end{itemize}

(i): If $\Sha(K/k)[p^{\infty}]\neq \{1\}$, then we have $G/S_{p}\cong C_{4}$. On the other hand, if $\Sha(K/k)^{(p)}\neq \{1\}$, then $G/S_{p}$ is isomorphic to $(C_{2})^{2}$ or $\fA_{4}$. Since $C_{4}$ is not isomorphic to $(C_{2})^{2}$ or $\fA_{4}$, the assertion follows from (1)--(4). 

(ii): This follows from (i), (1) and (2). 

(iii): The assertion is a consequence of (i), (3) and (4). 
\end{proof}

\begin{lem}[{cf.~\cite[p.~198]{Tate1967}}]\label{lem:tate}
Let $K/k$ be a finite abelian extension of a global field. We further assume that 
\begin{itemize}
\item $\Gal(K/k)\cong C_{n_{1}}\times C_{n_{2}}$ where $n_{1},n_{2}\in \Zpn$; and
\item all decomposition groups of $K/k$ are cyclic. 
\end{itemize}
Then there is an isomorphism
\begin{equation*}
\Sha(K/k)\cong \Z/\gcd(n_{1},n_{2}). 
\end{equation*}
\end{lem}

\begin{proof}
Since all decomposition groups of $K/k$ are cyclic, Corollary \ref{cor:onhn}, gives an isomorphism
\begin{equation*}
\Sha(K/k)\cong \Sha_{\omega}^{2}(G,J_{G})^{\vee}.  
\end{equation*}
Moroever, Proposition \ref{prop:abts} gives an isomorphism $\Sha_{\omega}^{2}(G,J_{G})^{\vee}\cong \Z/\gcd(n_{1},n_{2})$. Hence the assertion holds. 
\end{proof}

\begin{thm}\label{thm:cltt}
Let $p\neq 3$ be a prime number, and $k$ a global field whose characteristic is different from $p$. 
\begin{enumerate}
\item There is a finite separable field extension $K/k$ of degree $9p$ such that
\begin{equation*}
\Sha(K/k)\cong \Z/3p. 
\end{equation*}
\item We further assume that $k$ has characteristic $\neq \ell$, where $\ell\notin \{3,p\}$ is a prime number. Then there is a finite separable field extension $K/k$ of degree $3p\ell$ such that
\begin{equation*}
\Sha(K/k)\cong \Z/p\ell. 
\end{equation*}
\end{enumerate}
\end{thm}

\begin{proof}
Put
\begin{equation*}
G:=(C_{p})^{2}\rtimes C_{3}, \quad H:=\Delta(C_{p})\rtimes \{1\},
\end{equation*}
where the semi-direct product is the same as  Theorem \ref{thm:pldt} ($\alpha$) for $\ell=3$ and $m=2$ if $p\equiv 1\bmod \ell$, and is the same as  Theorem \ref{thm:pldt} ($\beta$) for $\ell=3$ if $p\equiv -1\bmod \ell$. We denote by $S_{p}$ the unique $p$-Sylow subgroup of $G$. 

(i): Set 
\begin{equation*}
\widetilde{G}:=G\times C_{3},\quad \widetilde{H}:=H\times \{1\},\quad \widetilde{S}_{p}:=S_{p}\times \{1\}. 
\end{equation*}
Then $(G:H)=9p$ and $\widetilde{S}_{p}$ is a $p$-Sylow subgroup of $G$. Moreover, Lemma \ref{lem:abcp} implies that
\begin{itemize}
\item $\widetilde{S}_{p}$ is normal in $G$; and
\item $\widetilde{G},\widetilde{H}$ and $\widetilde{S}_{p}$ satisfy (a), (b) and (c) in Theorem \ref{thm:ptnt}. 
\end{itemize}
In particular, we have $\widetilde{G}/\widetilde{S}_{p}\cong (C_{3})^{3}$. On the other hand, by Proposition \ref{prop:shf2}, there is an abelian extension $K_{0}/k$ with Galois group $(C_{3})^{2}$ in which all decomposition groups are cyclic. Furthermore, since the characteristic of $k$ is different from $p$, Proposition \ref{prop:shf1} implies the existence of a finite Galois extension $\widetilde{K}/k$ that contains such that $\Gal(\widetilde{K}/k)\cong \widetilde{G}$, $\widetilde{K}^{\widetilde{S}_{p}}=\widetilde{K}_{0}$ and all decomposition groups of $\widetilde{K}/k$ are cyclic. Now, let $K$ be the intermediate field of $\widetilde{K}/k$ corresponding to $\widetilde{H}$. Then we obtain
\begin{equation*}
\Sha(K/k)[p^{\infty}]\cong \Z/p
\end{equation*}
by Theorem \ref{thm:ptnt} (i). On the other hand, since $\widetilde{H}$ is contained in $\widetilde{S}_{p}$, Theorem \ref{thm:ptnt} (ii) gives an isomorphism $\Sha(K/k)^{(p)}\cong \Sha(\widetilde{K}^{\widetilde{S}_{p}}/k)$. Therefore, we obtain an isomorphism 
\begin{equation*}
\Sha(K/k)^{(p)}\cong \Z/3
\end{equation*}
because $G/S_{p}\cong (C_{3})^{2}$. Consequently, one has
\begin{equation*}
\Sha(K/k)\cong \Z/p\oplus \Z/3\cong \Z/3p
\end{equation*}
as desired since $p\neq 3$. 

(ii): Consider a group extension of $G$ as follows: 
\begin{equation}\label{eq:expf}
1\rightarrow (C_{\ell})^{2}\rightarrow \widetilde{G} \xrightarrow{\pi} G \rightarrow 1, 
\end{equation}
where the action of $G$ on $(C_{\ell})^{2}$ is defined as the composite of the surjection $G\twoheadrightarrow C_{3}$ and the homomorphism
\begin{equation*}
C_{3}\rightarrow \GL_{2}(\Fl);\,c_{3}\mapsto 
\begin{pmatrix}
0&-1\\
1&-1
\end{pmatrix}. 
\end{equation*}
In this case, the exact sequence \eqref{eq:expf} splits because $\#G$ is coprime to $\ell$. Fix an isomorphism $G\cong (C_{\ell})^{2}\rtimes G$, and denote by $\widetilde{H}$ the subgroup of $\widetilde{G}$ corresponding to $\Delta(C_{\ell})\rtimes H$. Then, it is contained in $\pi^{-1}(H)$. Furthermore, Lemma \ref{lem:abcp} implies that 
\begin{itemize}
\item a $p$-Sylow subgroup $\widetilde{S}_{p}$ of $\widetilde{G}$ is normal; and 
\item $\widetilde{G}$, $\widetilde{H}$ and $\widetilde{S}_{p}$ satisfy (a), (b) and (c) in Theorem \ref{thm:ptnt}. 
\end{itemize}
In addition, we have the following: 
\begin{equation}\label{eq:spqt}
\widetilde{G}/\widetilde{S}_{p}\cong (C_{\ell})^{2}\rtimes C_{3},\quad
\widetilde{H}\widetilde{S}_{p}/\widetilde{S}_{p}\cong \Delta(C_{\ell})\rtimes \{1\}
\end{equation}
Here the right-hand sides of two isomorphisms in \eqref{eq:spqt} coincides with ($\alpha$) in Theorem \ref{thm:pldt} attached to the pair of prime numbers $(\ell,3)$ and $m=2$ if $\ell \equiv 1\bmod 3$, and with ($\beta$) in Theorem \ref{thm:pldt} attached to the pair of prime numbers $(\ell,3)$ if $\ell \equiv 2 \bmod 3$. 

On the other hand, by the same proof as (i), there is a finite Galois extension $\widetilde{K}/k$ with Galois group $G$ of which all decomposition groups are cyclic. Then Theorem \ref{thm:ptnt} (i) gives an isomorphism
\begin{equation*}
\Sha(K/k)[p^{\infty}]\cong \Z/p. 
\end{equation*}
In addition, since \eqref{eq:spqt} is valid, Theorem \ref{thm:pldt} gives an isomorphism $\Sha(\widetilde{K}^{\widetilde{H}\widetilde{S}_{p}}/k)\cong \Z/\ell$. Combining this result with Theorem \ref{thm:ptnt} (ii), we obtain an isomorphism
\begin{equation*}
\Sha(K/k)^{(p)}\cong \Z/\ell. 
\end{equation*}
Therefore, one has
\begin{equation*}
\Sha(K/k)\cong \Z/p \oplus \Z/\ell \cong \Z/p\ell
\end{equation*}
as desired since $p$ and $\ell$ are coprime to each other. 
\end{proof}

\subsection{Comparison with previous results in small degree}\label{ssec:cppr}

Here we compare Theorems \ref{thm:pldt} and \ref{thm:pctw} with previous results; \cite{Drakokhrust1987} and \cite{Hoshi2022}. 

\subsubsection{The case of degree $6=2\cdot 3$. }

In this case, there exist $16$ permutation groups of degree $6$ that are not isomorphic to each other. See also \cite{Butler1983} and \cite[p.~60, Table 2.1]{Dixon1996}. We denote them by $6Tm$ where $m\in \{1,\ldots,16\}$. Then, by \cite[\S 10]{Drakokhrust1987}, $\Sha(K/k)$ is trivial or isomorphic to $\Z/2$, and $\Sha(K/k)\cong \Z/2$ may happen only in the cases $6T4 \cong \fA_{4}$ and $6T12 \cong \fA_{5}$. Note that a $2$-Sylow subgroup of $6T4$ is normal. On the other hand, a $2$-Sylow or a $3$-Sylow subgroup is normal if and only if $m$ is one of the $9$ integers as follows: 
\begin{equation*}
1\sim 6,9,10,13. 
\end{equation*}
Hence Theorem \ref{thm:pldt} recovers $1$ of the $2$ cases where $\Sha(K/k)\cong \Z/2$ happens. 

\subsubsection{The case of degree $10=2\cdot 5$. }

In this case, there exist $45$ permutation groups of degree $10$ that are not isomorphic to each other. See also \cite{Butler1983}. We denote them by $10Tm$ where $m\in \{1,\ldots,45\}$. It is known by Hoshi--Kanai--Yamasaki (\cite{Hoshi2022}) that $\Sha(K/k)$ is trivial or isomorphic to $\Z/2$, and $\Sha(K/k)\cong \Z/2$ may happen only in the cases $10T7 \cong \fA_{5}$, $10T26 \cong \fA_{6}$ and $10T32\cong \fS_{6}$. Note that $2$-Sylow subgroups of these groups are not normal. On the other hand, a $2$-Sylow or a $5$-Sylow subgroup is normal if and only if $m$ is one of $18$ integers as follows: 
\begin{equation*}
1\sim 6,8\sim 10,14,17\sim 21,27,28,33. 
\end{equation*}
All cases satisfy $\Sha(K/k)=1$ by Theorem \ref{thm:pldt}, which is consistent with the previous result. 

\subsubsection{The case of degree $12=4\cdot 3$. }

In this case, there exist $301$ permutation groups of degree $12$ that are not isomorphic to each other. See also \cite{Royle1987}. We denote them by $12Tm$ where $m\in \{1,\ldots,301\}$. It is known by Hoshi--Kanai--Yamasaki (\cite{Hoshi2023}) that $\Sha(K/k)$ is trivial or isomorphic to $\Z/2$, and there are $64$ cases where $\Sha(K/k)\cong \Z/2$ may happen. On the other hand, a $3$-Sylow subgroup of $12Tm$ is normal if and only if $m$ is one of the $77$ integers as follows. 
\begin{equation}\label{eq:itho}
\begin{gathered}
1,5,11\sim 15,17,19,35,36,38,39,41,42,46,72,73,78\sim 82,84,\\
116,118\sim 121,125,131,156,167,169,170,173,\\
209,211,212,215\sim 217,243\sim 245,247\sim 249,262\sim 264,266,267,274,\\
2,3,16,18,34,40,47,70,71,117,130,168,171,172,174,246,261,\\
10,37,77,210,214,242
\end{gathered}
\end{equation}
By Theorem \ref{thm:pctw}, the isomorphism $\Sha(K/k) \cong \Z/2$ happens possibly only when $m$ is one of the last two lines in \eqref{eq:itho}. Moreover, if $m$ is one of the integers that lie in the lowest line, then we cannot determine the vanishing of $\Sha(K/k)$ by a condition that is stable under automorphisms of $G$. This follows from the fact that $12Tm$ is contained in \cite[Table 2-2]{Hoshi2023} in that case.

\subsubsection{The case of degree $14$. }

In this case, there exist $63$ permutation groups of degree $14$ that are not isomorphic to each other. See also \cite{Butler1993}. We denote them by $14Tm$ where $m\in \{1,\ldots,63\}$. It is known by Hoshi--Kanai--Yamasaki (\cite{Hoshi2022}) that $\Sha(K/k)$ is trivial or isomorphic to $\Z/2$, and $\Sha(K/k)\cong \Z/2$ happens only in the case $14T30 \cong \PSL_{2}(13)$. On the other hand, a $2$-Sylow or a $7$-Sylow subgroup of $14Tm$ is normal if and only if $m$ is one of the $30$ integers as follows: 
\begin{equation*}
1\sim 9,11\sim 15,18,20\sim 26,29,31,32,35\sim 37,44,45. 
\end{equation*}
All cases satisfy $\Sha(K/k)=1$ by Theorem \ref{thm:pldt}, which is consistent with the previous result. 

\subsubsection{The case of degree $15$. }

In this case, there exist $104$ transitive groups of degree $15$ which are not isomorphic to each other. See also \cite{Butler1993}. We denote them by $15Tm$ where $m\in \{1,\ldots,104\}$. It is known by Hoshi--Kanai--Yamasaki (\cite[Theorem 1.15]{Hoshi2022}) that $\Sha(K/k)$ is trivial or isomorphic to $\Z/5$, and $\Sha(K/k)\cong \Z/5$ happens only in the cases $15T9\cong (C_{5})^{2}\rtimes C_{3}$ and $15T14\cong (C_{5})^{2}\rtimes D_{3}$. On the other hand, the normality of $3$-Sylow subgroup or a $5$-Sylow subgroup holds if and only if $m$ is one of the $64$ integers as follows: 
\begin{gather*}
1\sim 4,6\sim 9,11\sim 14,17\sim 19,25\sim 27,30\sim 46,\\
48\sim 52,54\sim 60,64\sim 68,71,73\sim 75,79\sim 82,84\sim 87. 
\end{gather*}
Hence Theorem \ref{thm:pldt} recovers all the cases where $\Sha(K/k)\cong \Z/5$ may happen.

\end{document}